\documentclass[11pt, bibliography=totocnumbered, one-side]{amsart}

\usepackage{xcolor}
\usepackage[T1]{fontenc}
\usepackage[utf8]{inputenc}
\usepackage{lmodern}
\usepackage{amsmath, amssymb, amsfonts, amsthm, mathtools}
\usepackage{extarrows}
\usepackage{mathrsfs}
\usepackage{tikz}
\usepackage{musicography}
\usepackage{tikz-cd}
\usepackage{graphicx}
\usepackage{faktor}
\usepackage{young}
\usepackage[centertableaux]{ytableau}
\usepackage{verbatim} 
\usepackage{enumitem}
\usepackage{microtype}
\usepackage[style=alphabetic, backend=biber, maxbibnames=99]{biblatex}
\usepackage{csquotes}
\usepackage{todonotes}
\usepackage{color}      
\usepackage{xparse}
\usepackage{setspace}
\usepackage{longtable}
\usepackage[bindingoffset=5mm, width=15cm, height=22cm]{geometry}
\usepackage{kantlipsum}
\usepackage{url}
\def\UrlBreaks{\do\/\do-}
\usepackage{breakurl}
\usepackage[breaklinks]{hyperref}
\allowdisplaybreaks
\usetikzlibrary{cd,quotes,babel,angles}

 \DeclareUnicodeCharacter{0308}{HERE!HERE!}
 \DeclareUnicodeCharacter{03B3}{XXXX Here I am XXX}

\definecolor{light-gray}{gray}{0.85}

\binoppenalty=\maxdimen         
\relpenalty=\maxdimen           

\newtheorem{thm}{Theorem}[section]
\newtheorem{lemma}[thm]{Lemma}
\newtheorem{prop}[thm]{Proposition}

\newtheorem{cor}[thm]{Corollary}
\theoremstyle{definition}
\newtheorem{defi}[thm]{Definition}

\newtheorem{rem}[thm]{Remark}

\DeclareMathOperator{\Hom}{Hom}
\DeclareMathOperator{\Nat}{Nat}
\DeclareMathOperator{\Fun}{Fun}
\DeclareMathOperator{\End}{End}
\DeclareMathOperator{\colim}{colim}
\DeclareMathOperator{\rad}{rad}
\DeclareMathOperator{\K}{\mathbb{K}}
\DeclareMathOperator{\Z}{\mathbb{Z}}
\newcommand{\from}{\colon}
\newcommand{\Ln}{\mathcal{L}_{n}}
\newcommand{\Lnk}{\mathcal{L}_{n,k}}
\DeclareMathOperator{\J}{J}
\DeclareMathOperator{\Max}{Max}
\numberwithin{equation}{section} 
\DeclareMathOperator{\C}{\mathbb{C}}
\DeclareMathOperator{\R}{\mathbb{R}}
\DeclareMathOperator{\N}{\mathbb{N}}
\DeclareMathOperator{\la}{\lambda}
\DeclareMathOperator{\sgn}{sgn}
\DeclareMathOperator{\id}{id}
\DeclareMathOperator{\im}{im}
\DeclareMathOperator{\Lan}{Lan}
\DeclareMathOperator{\Yon}{Yon}
\DeclareMathOperator{\coker}{coker}
\DeclareMathOperator{\D}{D} 
\def\acts{\curvearrowright}
\DeclareMathOperator{\modd}{\mathbf{mod}}
\DeclareMathOperator{\sAb}{\mathbf{sAb}}
\DeclareMathOperator{\Ab}{\mathbf{Ab}}
\DeclareMathOperator{\Top}{\mathbf{Top}}
\DeclareMathOperator{\sSet}{\mathbf{sSet}}
\DeclareMathOperator{\Set}{\mathbf{Set}}
\DeclareMathOperator{\Mod}{\mathbf{Mod}}
\DeclareMathOperator{\proj}{\mathbf{proj}}
\DeclareMathOperator{\inj}{\mathbf{inj}}
\newcommand{\Rl}{R_{T_\lambda}}
\newcommand{\Cl}{C_{T_\lambda}}
\DeclareMathOperator{\ob}{ob}
\DeclareMathOperator{\tr}{tr}
\DeclareMathOperator{\TL}{TL}
\DeclareMathOperator{\Ug}{U}
\DeclareMathOperator{\sl2}{\mathfrak{sl}_2}
\DeclareMathOperator{\gl2}{\mathfrak{gl}_2}
\DeclareMathOperator{\U}{U(\mathfrak{sl}_2)}
\newcommand{\Uq}{\Ug_q(\sl2)}
\DeclareMathOperator{\M}{M}
\DeclareMathOperator{\Sym}{Sym}
\DeclareMathOperator{\GL}{GL}
\DeclareMathOperator{\op}{op}
\DeclareMathOperator{\B}{B}
\DeclareMathOperator{\Catk}{\text{\textbf{Cat}}_{\mathbb{K}}}

\DeclareFieldFormat[article,unpublished]{citetitle}{\mkbibitalic{#1}\isdot}
\DeclareFieldFormat[article,unpublished]{title}{\mkbibitalic{#1}\isdot}
\DeclareFieldFormat[article,book]{number}{\mkbibbold{#1}}
\DeclareFieldFormat[article]{volume}{\mkbibbold{#1}}

\usetikzlibrary{calc}
\newcommand{\cbox}[1]{ \vcenter{ \hbox{#1} } }
\usetikzlibrary{patterns,decorations.pathreplacing}
\tikzset{
	tldiagram/.style={thick, scale=0.35}
}
\newcommand{\tlcoord}[2]
{
	(2*#2 , 3*#1)
}
\newcommand{\makecdots}[2]
{
	\node at (2*#2, 3*#1 + 1.5) {$\cdots$};
}
\newcommand{\lineup}{-- ++(0,3)}
\newcommand{\linedown}{-- ++(0,-3)}
\newcommand{\smalllineup}{-- ++(0,2)}
\newcommand{\smalllinedown}{-- ++(0,-2)}
\newcommand{\linewave}[2]{
	.. controls +(0,1.5*#1) and +(0,1.5*-#1) .. ++(2*#2, 3*#1)
}
\newcommand{\capright}{arc (180:0:1)}
\newcommand{\capleft}{arc (0:180:1)}
\newcommand{\cupright}{arc (180:360:1)}
\newcommand{\cupleft}{arc (360:180:1)}
\newcommand{\xcapright}[1]{arc (180:0:#1)}
\newcommand{\xcapleft}[1]{arc (0:180:#1)}
\newcommand{\xcupright}[1]{arc (180:360:#1)}
\newcommand{\xcupleft}[1]{arc (360:180:#1)}
\newcommand{\drawdots}[3]
{
	\foreach \i in {#2, ..., #3}
	{
		\draw[fill] (2*\i, 3*#1) circle (0.3em);
	}
}
\newcommand{\maketlbox}[2]{
	++(-1,-0.75)
	+(-0.5, -0.3)
	[fill=blue!20!white, rounded corners]
	rectangle
	+(#1*2 + 0.5, 1.5 + 0.3)
	node at ++(#1,0.75) {#2}
}
\newcommand{\maketlellipse}[2]{
	++ (#1 - 1,0)
	[fill=green!20!white]
	ellipse (#1 + 0.5 and 1.2)
	node {#2}
}

\newcommand{\tlcross}
{
	\begin{tikzpicture}[scale=0.092]
		\draw (0,0) -- (2,3);
		\draw (2,0) -- (0,3);
	\end{tikzpicture}
}
\newcommand{\pretlcap}
{
	\begin{tikzpicture}[xscale=0.092,yscale=0.12]
		\draw (0,0) -- +(0,1) \capright -- +(0,-1);
	\end{tikzpicture}
}
\newcommand{\pretlcup}
{
	\begin{tikzpicture}[xscale=0.092,yscale=0.12]
		\draw (0,0) -- +(0,-1) \cupright -- +(0,1);
	\end{tikzpicture}
}
\newcommand{\pretlcupcap}
{
	\begin{tikzpicture}[scale=0.092]
		\draw \tlcoord{1}{0} \cupright;
		\draw \tlcoord{0}{0} \capright;
	\end{tikzpicture}
}
\newcommand{\tlcircle}
{
	\begin{tikzpicture}[scale=0.11]
		\draw (0,0) circle (1);
	\end{tikzpicture}
}
\newcommand{\pretlline}
{
	\vert
}\newcommand{\tltwolines}
{
	\vert\,\vert
}

\newcommand\CROSS{   
\begin{tikzpicture}[anchorbase]
\draw    (-0.25,0.25) to (0.25,-0.25) ;
\draw    (-0.25,-0.25) to (0.25,0.25) ;
\end{tikzpicture}
}
\newcommand\FOURLEGS{     
\begin{tikzpicture}[anchorbase]
\draw    (-0.25,-0.25) to (-0.25,0) ;
\draw    (-0.25,0) to (0.25,0) ;
\draw    (0.25,0) to (0.25,-0.25) ;
\draw    (0,0) to (0,0.25) ;
\draw    (-0.25,0.25) to (0.25,0.25) ;
\draw    (-0.25,0.25) to (-0.25,0.5) ;
\draw    (0.25,0.25) to (0.25,0.5) ;
\end{tikzpicture}
}

\newcommand\Fourlegs{     
\begin{tikzpicture}[anchorbase]
\draw    (-0.15,-0.15) to (-0.15,0) ;
\draw    (-0.15,0) to (0.15,0) ;
\draw    (0.15,0) to (0.15,-0.15) ;
\draw    (0,0) to (0,0.15) ;
\draw    (-0.15,0.15) to (0.15,0.15) ;
\draw    (-0.15,0.15) to (-0.15,0.3) ;
\draw    (0.15,0.15) to (0.15,0.3) ;
\end{tikzpicture}
}

\newcommand\Lowerfourlegs{     
\begin{tikzpicture}[anchorbase]
\draw    (-0.15,-0.15) to (-0.15,0) ;
\draw    (-0.15,0) to (0.15,0) ;
\draw    (0.15,0) to (0.15,-0.15) ;
\draw    (0,0) to (0,0.15) ;
\end{tikzpicture}
}

\newcommand\Upperfourlegs{     
\begin{tikzpicture}[anchorbase]
\draw    (0,0) to (0,0.15) ;
\draw    (-0.15,0.15) to (0.15,0.15) ;
\draw    (-0.15,0.15) to (-0.15,0.3) ;
\draw    (0.15,0.15) to (0.15,0.3) ;
\end{tikzpicture}
}

\newcommand\CAP{
    \begin{tikzpicture}[anchorbase]
      \draw (-0.25,-0.25) to (-0.25,0);
      \draw (-0.25,0) to (0.25,0);
      \draw (0.25,0) to (0.25,-0.25);
    \end{tikzpicture}
}

\newcommand\CUP{
    \begin{tikzpicture}[anchorbase]
      \draw (-0.25,-0.25) to (-0.25,0);
      \draw (-0.25,-0.25) to (0.25,-0.25);
      \draw (0.25,0) to (0.25,-0.25);
    \end{tikzpicture}
}

\DeclareMathOperator*{\tlcap}{\,\pretlcap\,}
\DeclareMathOperator*{\tlcup}{\,\pretlcup\,}
\DeclareMathOperator*{\tlcupcap}{\,\pretlcupcap\,}
\DeclareMathOperator*{\tlline}{\pretlline}

\setlist[enumerate]{label= \roman*)}

\setlist{
	listparindent=\parindent,
	parsep=0pt
}
\usepackage{tikz}
\usetikzlibrary{arrows,patterns}
\usepackage{tikz-cd}

\newcommand{\dotlabel}[1]{$\scriptstyle{#1}$}
\newcommand{\pd}[2][black]{\filldraw[#1] (#2) circle (1.5pt)} 
\newcommand{\pdg}[3]{
    \filldraw[black] (#1) circle (1.5pt) node[anchor=#2] {\dotlabel{#3}}
}
\newcommand{\braidto}{to[out=up,in=down]}
\newcommand\bluedot[1]{\filldraw[blue] #1 circle (2pt)}
\newcommand\opendot[1]{\filldraw[fill=white,draw=black] (#1) circle (2pt)}
\newcommand\token[3]{
  \filldraw[blue] (#2) circle (1.5pt) node[anchor=#1] {\dotlabel{#3}}
}
\newcommand\teleport[2]{
  \draw[blue] (#1) to (#2);
  \filldraw[blue] (#1) circle (1.5pt);
  \filldraw[blue] (#2) circle (1.5pt)
}

\newcommand\Peis[1][]{{\Heis_{\uparrow \downarrow}^{#1}}}

\newcommand\merge{
    \begin{tikzpicture}[anchorbase]
      \draw (-0.25,-0.25) to (0,0);
      \draw (0.25,-0.25) to (0,0);
      \draw (0,0) to (0,0.25);
    \end{tikzpicture}
}

\newcommand\spliter{
    \begin{tikzpicture}[anchorbase]
      \draw (-0.25,0.25) to (0,0);
      \draw (0.25,0.25) to (0,0);
      \draw (0,0) to (0,-0.25);
    \end{tikzpicture}
}

\newcommand\crossing{
    \begin{tikzpicture}[anchorbase]
      \draw (-0.25,-0.25) to (0.25,0.25);
      \draw (0.25,-0.25) to (-0.25,0.25);
    \end{tikzpicture}
}

\newcommand\bottompin{
    \begin{tikzpicture}[anchorbase]
      \draw (0,0) to (0,0.25);
      \opendot{0,0};
    \end{tikzpicture}
}

\newcommand\toppin{
    \begin{tikzpicture}[anchorbase]
      \draw (0,0) to (0,-0.25);
      \opendot{0,0};
    \end{tikzpicture}
}

\newcommand\idstrand{
    \begin{tikzpicture}[anchorbase]
      \draw (0,-0.25) to (0,0.25);
    \end{tikzpicture}
}

\newcommand\tokstrand[1][g]{
    \begin{tikzpicture}[anchorbase]
      \draw (0,-0.25) to (0,0.25);
      \token{east}{0,0}{#1};
    \end{tikzpicture}
}

\newcommand\lolly{
    \begin{tikzpicture}[anchorbase]
        \draw (0,-0.2) -- (0,0.2);
        \opendot{0,-0.2};
        \opendot{0,0.2};
    \end{tikzpicture}
}

\tikzset{anchorbase/.style={>=To,baseline={([yshift=-0.5ex]current bounding box.center)}}}
\tikzset{
    centerzero/.style={>=To,baseline={([yshift=-0.5ex](#1))}},
    centerzero/.default={0,0}
}

\begin{document}

\title[Interpolation categories]{On interpolation categories for the hyperoctahedral group}
\author{{\rm Th. Heidersdorf, G. Tyriard}}

\address{T. H.: Newcastle University}
\email{heidersdorf.thorsten@gmail.com} 
\address{G.T.: Mathematisches Institut Universit\"at Bonn}
\email{georgetyriard@gmail.com}

\date{}

\begin{abstract} Two different types of Deligne categories have been defined to interpolate the finite dimensional complex representations of the hyperoctahedral group. The first one, initially defined by Knop and then further studied by Likeng and Savage, uses a categorical analogue of the permutation representation as a tensor generator. The second one, due to Flake and Maassen, is tensor generated by a categorical analogue of the reflection representation. We construct a symmetric monoidal functor between the two and show that it is an equivalence of symmetric monoidal categories.
\end{abstract}


\maketitle

\section{Introduction}

The prototypical example of an interpolation category is Deligne's category $\underline{\text{Rep}}(S_t)$, $t \in \C$, \cite{Deligne}, a symmetric monoidal category that depends on a complex parameter $t$. These categories are semisimple for $t \notin \mathbb{N}$ and their semisimplification at positive integers $t \in \N$ coincides with the usual finite dimensional complex representations of $S_n$.

This construction has been generalized by many authors (see for example Knop
\cite{Knop_2007} and Etingof \cite{EtingofRepTheoComplexRank1}) and there are now many
interpolation or Deligne categories. These Deligne categories have found
applications in the theory of tensor categories, representations of
supergroups \cite{Heidersdorf}\cite{EhrigStroppel} \cite{EAS} and invariant theory \cite{Coulembier}.

We usually speak of \emph{the} Deligne category $\underline{\text{Rep}}(S_t)$ attached to the symmetric group as if the construction is canonical. The Deligne category $\underline{\text{Rep}}(S_t)$ depends however on the choice of a family of tensor generators $V_n$, one for each $n$ (a categorical analogue of a faithful representation of $S_n$). An obvious choice for $V_n$ is given by the representation of $S_n$ on $\C^n$. However Deligne did not use the
reflection representation $\mathbb{C}^n$ of $S_n$ and its tensor powers to
build his category, but he showed that both give the same category. This phenomenon was systematically studied by Knop \cite{Knop-subobject} in the general context of his interpolation categories. He showed that one can often replace an obvious tensor generator by minimal subobjects to construct the interpolation category.

\subsection{The hyperoctahedral groups}

Among the cases studied by Knop were representations of wreath products $G
\wr S_n$. Knop constructed interpolating categories in this setting and
determined the singular parameters (those $t$ for which the categories are
not semisimple). These wreath product categories have been studied by many authors in recent years \cite{FS} \cite{Harman} \cite{Ryba1} \cite{Ryba2}. Knop's construction yields in particular an interpolation for the representations of the hyperoctahedral group $H_{n}= \Z_{2}\wr S_{n} $, the Weyl group of type B. Recently, Flake and Maassen
\cite{Flake_2021} generalised the ideas of Deligne's $\underline{\text{Rep}}(S_t)$ construction to interpolate the representation
categories of easy quantum groups (in the sense of Banica-Speicher \cite{BS09}). Among their examples is yet another interpolation category $\underline{\text{Rep}}(H_{t})$ for
$\text{Rep}(H_{n})$. Knop's construction was taken up recently by Likeng and Savage in \cite{Nyobe_Likeng_2021},
who gave a description of Knop's interpolation categories for wreath products $G\wr S_{n}$ using generators and relations.

The two different interpolation categories $\underline{\text{Rep}}(H_{t})$ and $\text{Par}(\Z_{2},t)^{Kar}$ for $H_{n}$ given
in \cite{Flake_2021} and \cite{Nyobe_Likeng_2021} interpolate the
representation categories for the hyperoctahedral groups in different objects and for different parameters
$t\in \C$. The morphisms between the tensor powers of the
generating objects in the categories $\underline{\text{Rep}}_{0}(H_{n})$ and
$\text{Par}(\Z_{2},2n)$, the pre-Karoubian envelope versions of the
interpolation categories, mimic the behaviour of the morphism spaces between
the tensor powers of the reflection representation $u$ and of the
permutation representation $V$ respectively. So the interpolation/semisimplification functors
$G:\underline{\text{Rep}}(H_{t}) \to \text{Rep}(H_{n})$ and
$H:\text{Par}(\Z_{2},2n)^{Kar}\to \text{Rep}(H_{n})$ are defined on objects by $[k]\to u^{\otimes k}$ and $[\tilde{k}]\to V^{\otimes k}$ respectively. We compare these different constructions and formulate the universal properties of these categories. For this we derive in Theorem \ref{real Universal property fo RepH0, not Karoubian universal prop} a presentation via generators and relations for the reflection category $\underline{\text{Rep}}(H_t)$. We also describe the semisimplification of $\underline{\text{Rep}}(H_{t})$ in terms of these generators.

The main result of this thesis
is the following theorem (see Theorem \ref{Properties of the functor Omega} and Corollary \ref{commutative square}).

\begin{thm} \nonumber

There is a symmetric monoidal
equivalence
\begin{align*}
     \Omega: \underline{\text{Rep}}(H_{n}) \simeq \text{Par}(\Z_{2},2n)^{Kar}
\end{align*}
such that the diagram 
\begin{equation*}
\begin{tikzcd}
\underline{\text{Rep}}(H_{n}) \arrow{r}{G} \arrow[d,"\Omega"] & \text{Rep}(H_{n})\arrow{d}{=} \ \\ \text{Par}(\mathbb{Z}_{2},2n)^{Kar} \arrow{r}{H} & \text{Rep}(H_{n})
\end{tikzcd}
\end{equation*}
commutes for all $n\in \N$. Here $G$ and $H$ denote the semisimplification functors.
\end{thm}

Remarkably, the much more complicated morphism spaces in $\text{Par}(\Z_{2},2n)^{Kar}$ which are given by \emph{$\mathbb{Z}_2$-coloured partition diagrams}, can be therefore completely modelled by the much easier \emph{even, uncolored partition diagrams} which describe the morphism spaces in $\underline{\text{Rep}}(H_{t})$. As an application we obtain that the isomorphism classes of indecomposable objects of $\text{Par}(\Z_{2},t)^{Kar}$ are parametrized by the set of all bipartitions for $t\in \C \backslash \{0\}$. This was proven by Knop \cite{Knop_2007} in the semisimple $t \neq 2n$-case (see also \cite{Nyobe_Likeng_2021}).

Both categorical actions of the two interpolation categories on tensor products of the permutation representation and the reflection representation via the semisimplification functors can be seen as categorical analogues of a double centralizer property. The explicit description of these two functors therefore generalizes in the special case of $H_n$ the work of Orellana \cite{Orellana} (who studied the endomorphism algebras of tensor powers of the reflection representation) and of Bloss \cite{Bloss} (who studied the endomorphism algebras of tensor powers of the permutation representation). Our theorem can be seen as an upgraded categorical compatibility between these two types of Schur-Weyl duality.

It would be interesting to explore whether such equivalences occur in more
general situations like general wreath products or more specifically for complex reflection groups $G(r,p,d)$. Another interesting direction is the connections to other settings of stable representation theory for the hyperoctahedral group. Wilson \cite{FIWmodules} studied this from the perspective of stable sequences of Weyl group representations. Another setting are tensor representations of the infinite hyperoctahedral group, the inductive limit of the ascending tower of groups $H_1 \subset H_2 \subset...$, similar to Sam-Snowden's category $\text{Rep}(S_{\infty})$ \cite{StevenSnowden}. In the symmetric group case it is known that these different categories are closely related \cite{BarterHeiderdorf}, and one might expect that an analogous theorem holds for the hyperoctahedral case.

\subsection*{Acknowledgements} T.H. would like to thank Johannes Flake for helpful conversations. The research of T.H. was partially funded by the Deutsche Forschungsgemeinschaft (DFG, German Research
Foundation) under Germany's Excellence Strategy – EXC-2047/1 – 390685813.



\section{Background}

\subsection{Conventions about monoidal categories}

We follow \cite{TensorCategories} in our definition of a $\mathbb{C}$-linear symmetric monoidal category. In such a category we have the following morphisms. 

\begin{itemize}  
     \item a bifunctor $\otimes: \mathcal{C}\times \mathcal{C} \to \mathcal{C}$
     \item a unit object $\textbf{1}\in C$
     \item an associator, which is a natural isomorphism $\alpha:((-\otimes -)\otimes-) \to (- \otimes (- \otimes -))$
     \item a left unitor, which is a natural isomorphism $\lambda: (\textbf{1} \otimes -) \to (-)$
     \item a right unitor, which is a natural isomorphism $\rho: (- \otimes \textbf{1}) \to (-)$
     \item natural isomorphism $S_{A,B}:A\otimes B \to B \otimes A$ for all $A,B \in \mathcal{C}$
\end{itemize}

We will also assume that $End_{\mathcal{C}}(\textbf{1})\cong \C$. In case $\mathcal{C}$ is rigid, we have the following additional morphisms:

\begin{itemize}
  \item a contravariant monoidal endofunctor $*:\mathcal{C} \to \mathcal{C}$. The object $A^*$ is called the dual object of $A$. 
   \item an isomorphism $\tau_{A}: {A}\cong((A)^*)^*$ which is natural in all $A\in \mathcal{C}$, so a natural isomorphism $\tau:id_{\mathcal{C}}\to((-)^*)^*$. If this axiom holds, we call the functor $*$ involutive.
   \item an isomorphism $\nu:\textbf{1} \to \textbf{1}^*$.
   \item isomorphisms $\gamma_{A,B}: (A\otimes B)^*\to B^*\otimes A^*$ which are natural in all $A,B \in \mathcal{C}$. 
   \item morphisms called the evaluation and coevaluation
\[ ev_A: A \otimes A^* \to \textbf{1}, \ \ coev_A: \textbf{1} \to A^* \otimes A \]   

   for all $A\in \mathcal{C}$, satisfying the usual axioms.
\end{itemize}

We will also assume that for all $f:A\to B$ the following diagrams commute:
\begin{center}
    
\scalebox{0.8}{
\tikzset{every picture/.style={line width=0.75pt}} 

\begin{tikzpicture}[x=0.75pt,y=0.75pt,yscale=-1,xscale=1]

\draw    (81.33,223.58) -- (141.4,208.08) ;
\draw [shift={(143.33,207.58)}, rotate = 165.53] [color={rgb, 255:red, 0; green, 0; blue, 0 }  ][line width=0.75]    (10.93,-3.29) .. controls (6.95,-1.4) and (3.31,-0.3) .. (0,0) .. controls (3.31,0.3) and (6.95,1.4) .. (10.93,3.29)   ;
\draw    (382.4,230.64) -- (439.45,210.26) ;
\draw [shift={(441.33,209.58)}, rotate = 160.34] [color={rgb, 255:red, 0; green, 0; blue, 0 }  ][line width=0.75]    (10.93,-3.29) .. controls (6.95,-1.4) and (3.31,-0.3) .. (0,0) .. controls (3.31,0.3) and (6.95,1.4) .. (10.93,3.29)   ;
\draw    (212,275) -- (272.38,261.03) ;
\draw [shift={(274.33,260.58)}, rotate = 166.98] [color={rgb, 255:red, 0; green, 0; blue, 0 }  ][line width=0.75]    (10.93,-3.29) .. controls (6.95,-1.4) and (3.31,-0.3) .. (0,0) .. controls (3.31,0.3) and (6.95,1.4) .. (10.93,3.29)   ;
\draw    (80,260) -- (140.4,276.07) ;
\draw [shift={(142.33,276.58)}, rotate = 194.9] [color={rgb, 255:red, 0; green, 0; blue, 0 }  ][line width=0.75]    (10.93,-3.29) .. controls (6.95,-1.4) and (3.31,-0.3) .. (0,0) .. controls (3.31,0.3) and (6.95,1.4) .. (10.93,3.29)   ;
\draw    (212,211) -- (272.4,227.07) ;
\draw [shift={(274.33,227.58)}, rotate = 194.9] [color={rgb, 255:red, 0; green, 0; blue, 0 }  ][line width=0.75]    (10.93,-3.29) .. controls (6.95,-1.4) and (3.31,-0.3) .. (0,0) .. controls (3.31,0.3) and (6.95,1.4) .. (10.93,3.29)   ;
\draw    (517.4,268.64) -- (574.45,248.26) ;
\draw [shift={(576.33,247.58)}, rotate = 160.34] [color={rgb, 255:red, 0; green, 0; blue, 0 }  ][line width=0.75]    (10.93,-3.29) .. controls (6.95,-1.4) and (3.31,-0.3) .. (0,0) .. controls (3.31,0.3) and (6.95,1.4) .. (10.93,3.29)   ;
\draw    (385.4,256.64) -- (442.4,272.06) ;
\draw [shift={(444.33,272.58)}, rotate = 195.14] [color={rgb, 255:red, 0; green, 0; blue, 0 }  ][line width=0.75]    (10.93,-3.29) .. controls (6.95,-1.4) and (3.31,-0.3) .. (0,0) .. controls (3.31,0.3) and (6.95,1.4) .. (10.93,3.29)   ;
\draw    (519.4,211.64) -- (576.4,227.06) ;
\draw [shift={(578.33,227.58)}, rotate = 195.14] [color={rgb, 255:red, 0; green, 0; blue, 0 }  ][line width=0.75]    (10.93,-3.29) .. controls (6.95,-1.4) and (3.31,-0.3) .. (0,0) .. controls (3.31,0.3) and (6.95,1.4) .. (10.93,3.29)   ;

\draw (290,234.4) node [anchor=north west][inner sep=0.75pt]    {$\textbf{1}$};
\draw (53,233.4) node [anchor=north west][inner sep=0.75pt]    {$A\otimes B^{*}$};
\draw (227,186.4) node [anchor=north west][inner sep=0.75pt]    {$ev_{B}$};
\draw (234,283.4) node [anchor=north west][inner sep=0.75pt]    {$ev_{A}$};
\draw (364.79,234.82) node [anchor=north west][inner sep=0.75pt]    {$\textbf{1}$};
\draw (454.2,264.95) node [anchor=north west][inner sep=0.75pt]    {$A^{*} \otimes A$};
\draw (153,196.4) node [anchor=north west][inner sep=0.75pt]    {$B\otimes B^{*}$};
\draw (151,268.4) node [anchor=north west][inner sep=0.75pt]    {$A\otimes A^{*}$};
\draw (76,185.4) node [anchor=north west][inner sep=0.75pt]    {$f\otimes id_{B^{*}}$};
\draw (74,282.4) node [anchor=north west][inner sep=0.75pt]    {$id_{A} \otimes f^{*}$};
\draw (452.2,196.95) node [anchor=north west][inner sep=0.75pt]    {$B^{*} \otimes B$};
\draw (580.33,230.98) node [anchor=north west][inner sep=0.75pt]    {$A^{*} \otimes B$};
\draw (532,189.4) node [anchor=north west][inner sep=0.75pt]    {$f^{*} \otimes id_{B}$};
\draw (537,267.4) node [anchor=north west][inner sep=0.75pt]    {$id_{A^{*}} \otimes f$};
\draw (383,190.4) node [anchor=north west][inner sep=0.75pt]    {$coev_{B}$};
\draw (381,271.4) node [anchor=north west][inner sep=0.75pt]    {$coev_{A}$};
\end{tikzpicture}
}
\end{center}

We let a tensor functor $\mathcal{F}: \mathcal{C}\to \mathcal{D}$ between tensor categories be a $\C$-linear strong monoidal functor in the sense of \cite[Section 2.4, Definition 4.2.5]{TensorCategories}. In particular, this means that there exists a natural isomorphism $\zeta: F\circ ( -\otimes_{\mathcal{C}} - ) \to \mathcal{F}(-)\otimes_{\mathcal{D}} \mathcal{F}(-) $ and an isomorphism $\zeta: \mathcal{F}(\textbf{1}_{\mathcal{C}}) \cong \textbf{1}_{\mathcal{D}}$.

We denote by $Trace_{l},Trace_{r}:\End_{\mathcal{C}}(A) \to \End_{\mathcal{C}}(\textbf{1})= \C$ the left and right trace. A $\C$-linear spherical rigid symmetric monoidal category $\mathcal{C}$ is a $\C$-linear rigid symmetric monoidal category for which the left and the right traces coincide. In this case we denote the trace of morphism $f:A\to A$ in $\mathcal{C}$ by $tr(f)$ and define the categorical dimension of $A$ by $\text{dim}(A):= tr(id_{A})$. 

\subsection{Additive idempotent completion}

\begin{defi}
    An embedding $\mathcal{F}:\mathcal{C} \to \mathcal{D}$ between categories $\mathcal{C}$ and $\mathcal{D}$ is a functor which is injective on objects and faithful. A full embedding is an embedding which is full.
\end{defi}

Let $\mathcal{C}$ be a $\C$-linear category. We denote its additive envelope \cite{Comes_2011} by $\mathcal{C}^{add}$. The additive envelope $\mathcal{C}^{add}$ comes together with an obvious $\C$-linear full embedding $\iota_{add}:\mathcal{C} \to \mathcal{C}^{add}$.

\begin{defi}
Let $\mathcal{C}$ be a $\C$-linear category. The $\C$-linear category $\mathcal{C}^{\musNatural}$, called the idempotent completion of $\mathcal{C}$, has as objects pairs $(A,e)$, where $A \in \ob(\mathcal{C})$ and $e=e^{2}\in \Hom_{\mathcal{C}}(A,A)$ is an idempotent. When the context is clear, we will sometimes write $A$ for $(A,id_{A})$. The morphism sets are obtained by pre- and postcomposing the morphisms of the corresponding morphism sets in $\mathcal{C}$ with the idempotents of the pairs, thus
\begin{align*}
    \Hom_{\mathcal{C}^{\musNatural}}((A,e),(B,f)):= f\Hom_{\mathcal{C}}(A,B) e.
\end{align*}
The composition in $\mathcal{C}^{\musNatural}$ coincides with the composition of morphisms in $\mathcal{C}$. The identity morphism of $(A,e) \in \mathcal{C}^{\musNatural}$ is $id_{(A,e)}:= e$. The idempotent completion $\mathcal{C}^{\musNatural}$ comes together with a $\C$-linear functor $\iota_{\musNatural}: \mathcal{C} \to \mathcal{C}^{\musNatural}$ given by 
\begin{align*}
    A &\mapsto (A,id_{A}),  f \mapsto f.
\end{align*}
for all $A\in \mathcal{C}$ and $f\in \text{Hom}_{\mathcal{C}}(B,C).$ The functor $\iota_{\musNatural}$ is a $\C$-linear full embedding.
\end{defi}

\begin{defi}
    A $\C$-linear category $\mathcal{C}$ is idempotent complete if $\iota_{\musNatural}: \mathcal{C} \simeq \mathcal{C}^{\musNatural}$, is an equivalence i.e. all idempotents split. 
\end{defi}

\begin{defi} \label{THE definition of Karoubian envelope and embedding}
    We define the Karoubian envelope of a $\C$-linear category $\mathcal{C}$ by $\mathcal{C}^{Kar}:= (\mathcal{C}^{add})^{\musNatural}$ and $\iota_{Kar}:=\iota_{\musNatural} \circ \iota_{add}$, which is a $\C$-linear full embedding.   A $\C$-linear category $\mathcal{C}$ is Karoubi if $\iota_{Kar}:\mathcal{C} \simeq \mathcal{C}^{Kar}=(\mathcal{C}^{add})^{\musNatural}$ is an equivalence, i.e. all idempotents split and all finite biproducts exist. 
\end{defi}

The following is well-known (see \cite{Tyriard} for details).
 
\begin{prop} \label{Karoubian envelope spherical}
    The Karoubian envelope $\mathcal{C}^{Kar}$ of a $\C$-linear spherical rigid symmetric monoidal category $\mathcal{C}$ can again be given the structure of a $\C$-linear spherical rigid symmetric monoidal category.
\end{prop}

\begin{prop} \label{Karoubian envelope full, faithful}
     Let $\mathcal{C}$ be a $\C$-linear category.  Let $\mathcal{D}$ be a Karoubi category and $\gamma: \mathcal{C} \to \mathcal{D}$ a $\C$-linear functor. Let $\gamma':\mathcal{C}^{Kar} \to \mathcal{D}$ be the $\C$-linear additive functor given by the universal property of the Karoubian envelope such that $\gamma= \gamma' \circ \iota_{Kar}$. Assume that $\gamma$ is faithful. Then the induced functor $\gamma'$ is also faithful. Assume that $\gamma$ is fully faithful. Then the induced functor $\gamma'$ is also fully faithful.
\end{prop}
\begin{proof}
    We first assume that $\gamma$ is faithful and show the faithfulness of $\gamma'$. Because $\mathcal{D}$ is Karoubi, we can replace it by $\mathcal{D}^{Kar}$. We do this to be able to get a more manageable description of $\gamma'$. The functor $\gamma'$ is given up to isomorphism, but we work with the concrete choice which sends an object $(A,e)\in \mathcal{C}^{Kar}$ to $\gamma'((A,e)):=(\gamma(A),\gamma(e)) \in \mathcal{D}^{Kar}$. Let $g,h: (A,e)\to (B,f)$ be morphisms in $\mathcal{C}^{Kar}$ with $\gamma'(g)=\gamma'(h): (\gamma(A),\gamma(e))\to (\gamma(B),\gamma(f))$. By the fact that the morphisms $g,h:A\to B$ also lie in $\mathcal{C}$, $\gamma'$ sends the commutative square

\begin{equation*}
\begin{tikzcd}
(A,e) \arrow{r}{g,h}  & (B,f)  \\  (A,id_{A}) \arrow[u,"e"] \arrow{r}{g,h}  & (B,id_{B}) \arrow{u}{f}.
\end{tikzcd}
\end{equation*}

   in $\mathcal{C}^{Kar}$ to the commutative square

\begin{equation*}
\begin{tikzcd}
(\gamma(A),\gamma(e)) \arrow{r}  \arrow{r}{\gamma'(g),\gamma'(h)}& (\gamma(B),\gamma(f))  \ \\  (\gamma(A),id_{\gamma(A)}) \arrow{r}{\gamma(g),\gamma(h)} \arrow[u,"\gamma(e)"] & (\gamma(B),id_{\gamma(B)}) \arrow{u}{\gamma(f)}
\end{tikzcd}
\end{equation*}
 in $\mathcal{D}$. This shows that 
 \begin{align*}
     \gamma(g)&=\gamma(f\circ g)=\gamma(f)\circ \gamma(g) = \gamma'(g)\circ \gamma(e)\\&= \gamma'(h)\circ \gamma(e)= \gamma(f)\circ \gamma(h)=\gamma(f\circ g)=\gamma(h).
 \end{align*} Because $\gamma$ is faithful by assumption, we get $g=h$. This shows that $\gamma'$ is faithful. 

Now assume that $\gamma$ is fully faithful. Let $g:(\gamma(A),\gamma(e)) \to (\gamma(B),\gamma(f))$ be a morphism. Then $g$ is also a morphism $(\gamma(A),id_{\gamma(A)})\to (\gamma(B),id_{\gamma(A)})$. Because $\gamma$ is full there exists a $g':A\to B$ such that $\gamma(g')=g$. Because 
\begin{align*}
    \gamma(f \circ g' \circ e)=\gamma(f) \circ \gamma(g') \circ \gamma(e)=\gamma(g')
\end{align*}and $\gamma$ is faithful, we see that $g'=f \circ g' \circ e$. So $g'$ is a morphism $(A,e)\to (B,f)$ with $\gamma'(g')=g$. This shows that $\gamma'$ is full.
\end{proof}

\subsection{The hyperoctahedral group}

\begin{defi}
    The n-th hyperoctahedral group $H_{n}$ is the group of bijective functions $\pi$ on $\{-n,\ldots, -1,1,\ldots, n\}$ where $\pi(i)=-\pi(-i)$ for all $i\in \{-n,\ldots, -1,1,\ldots, n\}$.
\end{defi}

\begin{rem}
\noindent If we consider $\Z_{2}$ as the multiplicative group $\{-1,1\}\subset \C$, then we can describe the group $\mathbb{Z}_{2}^{n}$ by the following generators and relations: \begin{align*}
     \mathbb{Z}_{2}^{n}=\langle y_{1},y_{2},\ldots,y_{n}|&y_{i}^{2}=1 \text{ for all } i\in\{1,\ldots,n\},\\&(y_{i}y_{j})^{2}=1 \text{ for } 1\leqslant i < j\leqslant n \rangle .
 \end{align*} Note that the generators equal $y_{i}=(1,\ldots,1,\underset{i-th}{-1},1,\ldots,1)$ for all $i\in \{1,\ldots,n\}.$ We will sometimes write $1:=(1,\ldots,1)\in \Z_{2}^{n}$ for the neutral element of the group. As a wreath product $H_{n}$ equals $\mathbb{Z}_{2} \wr S_{n}=\mathbb{Z}_{2}^{n} \rtimes S_{n}$. The elements are of the form $a=(a_{1},\ldots,a_{n},\sigma)$, with $a_{i}\in \mathbb{Z}_{2}$ and $\sigma \in S_{n}$. The product is given by \begin{align*}
     (a_{1},\ldots,a_{n},\sigma)(b_{1},\ldots,b_{n},\rho)=(a_{1}b_{\sigma^{-1}(1)}, \ldots, a_{n}b_{\sigma^{-1}(n)}, \sigma \rho).
 \end{align*}
\end{rem}

Note that $H_{n}$ is isomorphic to the Weyl group of type $B_n$ and is also an example of a complex reflection group, namely $G(2,1,n)$.

\begin{defi}
    Let $a=(a_{1},\ldots,a_{n},\sigma) \in H_{n}$ and $(c_{1},\ldots, c_{n})\in \C^{n}$. Then the reflection representation $u=\C^{n}$ is defined by
    \begin{align*}
        a\cdot (c_{1},\ldots, c_{n})=(a_{1}c_{\sigma^{-1}(1)},\ldots,a_{n}c_{\sigma^{-1}(n)})
    \end{align*} where we consider $\Z_{2}=\{-1,1\}\subset \C$. If $e_{i}$ is a canonical basis element of $\C^{n}$, then $a\cdot e_{i}= a_{\sigma(i)}e_{\sigma(i)}$.
\end{defi}

\begin{defi}
    Let $a=(a_{1},\ldots,a_{n},\sigma) \in H_{n}$. The permutation representation \[ V=\C^{2n}=(\C \Z_{2})^{n}=\oplus_{i=1}^{n}(\C e_{1}^{i}\oplus \C e_{-1}^{i})\] of $H_{n}$ is defined by the $\C$-linear extension of the action
 \begin{align*}
        a \cdot e_{j}^{i}=e_{a_{\sigma(i)}\cdot j}^{\sigma(i)}
    \end{align*}
for $j\in \{-1,1\}$ and $ i \in \{1,\ldots,n\}$. Let $c_{j}^{i}\in \C$ for $j\in \Z_{2}$ and $i\in \{1,\ldots,n\}$. Then we get for $c=(c_{1}^{1} e_{1}^{1}+ c_{-1}^{1}  e_{-1}^{1},\ldots, c_{1}^{n} e_{1}^{n}+ c_{-1}^{n} e_{-1}^{n} \in V)$ that
\begin{align*}
    a \cdot c= (c_{a_{1}\cdot 1}^{\sigma^{-1}(1)} e_{1}^{1}+  c_{a_{1}\cdot (-1)}^{\sigma^{-1}(1)} e_{-1}^{1}, \ldots, c_{a_{n}\cdot 1}^{\sigma^{-1}(n)} e_{1}^{n}+  c_{a_{n}\cdot (-1)}^{\sigma^{-1}(n)} e_{-1}^{n}).
\end{align*}

\end{defi}

\begin{lemma}
    The permutation representation and the reflection representation of $H_{n}$ are faithful and self-dual.
\end{lemma}

\begin{rem} \label{u subrepresentatino of V}
    Note that $u\cong \tilde{u}:=\oplus_{i=1}^{n}\C(e_{1}^{i}-e_{-1}^{i})\subset V$ is a subrepresentation of $V$. Another interesting $n$-dimensional subrepresentation of $V$ is $v=\oplus_{i=1}^{n}\C(e_{1}^{i}+e_{-1}^{i})$, this is the complement of $\tilde{u}$ in $V$.

\end{rem}



\section{The Deligne categories attached to the permutation and reflection representation}

In this section we define two different interpolation categories for the hyperoctahedral groups, respectively modelled by the tensor products of the permutation representations and the tensor products of the reflection representations. The first one, $\text{Par}(\Z_{2},t)^{Kar}$ (the permutation category), has been studied by  Knop \cite{Knop_2007} and  Likeng-Savage \cite{Nyobe_Likeng_2021}; and the category $\underline{\text{Rep}}(H_{t})$ by Flake-Maassen \cite{Flake_2021}, the reflection category. 

\subsection{Partition diagrams}

We recall some basics about partition theory (including even partitions and $\mathbb{Z}_2$-coloured partitions) based on \cite{comes2011delignes}, \cite{Flake_2021} and \cite{Nyobe_Likeng_2021}. We define normal forms for coloured and non-coloured partitions. This was inspired on the idea of a standard decomposition given in the proof of \cite[Theorem 4.4]{Nyobe_Likeng_2021} and the normal form in \cite[Remark 1.4.16]{kock_2003}.

\begin{defi}
    Let $k,l\in \N$. We denote the set of all partitions of the set \\$\{1,\ldots,k,1',\ldots,l'\}$ by $P(k,l)$ and the set of all partitions by $P:=\sqcup_{k,l\in\mathbb{N}} P(k,l)$. Note that we will add or remove accents in the sets $\{1,\ldots,k,1',\ldots,l'\}$ depending on the context we are working in. We call $(k,l)$ the size of the partitions in $P(k,l)$. The elements of some partition $p\in P(k,l)$ are called components, parts or blocks. For partitions $p,q \in P(k,l)$ we call $q$ coarser as $p$ if every part in $p$ is a subset of some part in $q$. We can associate a partition diagram to each partition $p \in P(k,l)$ by placing $k$ vertices in a horizontal row and $l$ vertices in a horizontal row, above the first one. We label the vertices from left to right by the elements of the set $\{1,\ldots,k,1',\ldots,l'\}$. We draw a line between two vertices in the diagram if and only if the corresponding labels lie in the same component. We call two such partition diagrams equivalent if the set partitions are the same.
    
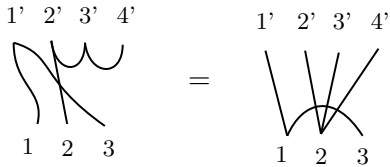
\begin{center}
\tikzset{every picture/.style={line width=0.75pt}} 
\begin{tikzpicture}[x=0.75pt,y=0.75pt,yscale=-1,xscale=1]
\draw  [line width=3] [line join = round][line cap = round] (229.33,90.92) .. controls (229.33,90.92) and (229.33,90.92) .. (229.33,90.92) ;
\draw  [line width=3] [line join = round][line cap = round] (249.33,90.92) .. controls (249.33,90.92) and (249.33,90.92) .. (249.33,90.92) ;
\draw  [line width=3] [line join = round][line cap = round] (266.33,90.92) .. controls (266.33,90.92) and (266.33,90.92) .. (266.33,90.92) ;
\draw  [line width=3] [line join = round][line cap = round] (286.33,90.92) .. controls (286.33,90.92) and (286.33,90.92) .. (286.33,90.92) ;
\draw  [line width=3] [line join = round][line cap = round] (240.33,131.92) .. controls (240.33,131.92) and (240.33,131.92) .. (240.33,131.92) ;
\draw  [line width=3] [line join = round][line cap = round] (257.33,131.92) .. controls (257.33,131.92) and (257.33,131.92) .. (257.33,131.92) ;
\draw  [line width=3] [line join = round][line cap = round] (277.33,131.92) .. controls (277.33,131.92) and (277.33,131.92) .. (277.33,131.92) ;
\draw  [line width=3] [line join = round][line cap = round] (357.33,94.92) .. controls (357.33,94.92) and (357.33,94.92) .. (357.33,94.92) ;
\draw  [line width=3] [line join = round][line cap = round] (377.33,94.92) .. controls (377.33,94.92) and (377.33,94.92) .. (377.33,94.92) ;
\draw  [line width=3] [line join = round][line cap = round] (394.33,94.92) .. controls (394.33,94.92) and (394.33,94.92) .. (394.33,94.92) ;
\draw  [line width=3] [line join = round][line cap = round] (414.33,93.92) .. controls (414.33,93.92) and (414.33,93.92) .. (414.33,93.92) ;
\draw  [line width=3] [line join = round][line cap = round] (368.33,137.92) .. controls (368.33,137.92) and (368.33,137.92) .. (368.33,137.92) ;
\draw  [line width=3] [line join = round][line cap = round] (385.33,136.92) .. controls (385.33,136.92) and (385.33,136.92) .. (385.33,136.92) ;
\draw  [line width=3] [line join = round][line cap = round] (405.33,136.92) .. controls (405.33,136.92) and (405.33,136.92) .. (405.33,136.92) ;
\draw    (242.43,131.86) .. controls (246.43,114.86) and (230.43,112.86) .. (230.43,90.86) ;
\draw    (230.43,90.86) .. controls (251.43,96.86) and (242.43,108.86) .. (276.43,132.86) ;
\draw    (249.43,89.86) -- (257.43,131.86) ;
\draw    (249.43,89.86) .. controls (251.43,106.86) and (266.43,107.86) .. (266.43,90.86) ;
\draw    (266.43,90.86) .. controls (268.43,107.86) and (285.43,108.86) .. (285.43,91.86) ;
\draw    (368.43,137.86) .. controls (376.43,119.86) and (392.43,115.86) .. (406.43,137.86) ;
\draw    (357.43,94.86) -- (368.43,137.86) ;
\draw    (377.43,94.86) -- (385.43,136.86) ;
\draw    (394.43,95.86) -- (385.43,136.86) ;
\draw    (414.43,93.86) -- (385.43,136.86) ;

\draw (317,107) node [anchor=north west][inner sep=0.75pt]   [align=left] {=};
\draw (231,171) node [anchor=north west][inner sep=0.75pt]  [font=\scriptsize] [align=left] {Example:};
\draw (280,170.4) node [anchor=north west][inner sep=0.75pt]  [font=\scriptsize]  {$\{\{1,3,1'\} ,\{2,2',3',4'\}\} \in P( 3,4)$};
\draw (233,138) node [anchor=north west][inner sep=0.75pt]  [font=\footnotesize] [align=left] {1};
\draw (253,139) node [anchor=north west][inner sep=0.75pt]  [font=\footnotesize] [align=left] {2};
\draw (274,139) node [anchor=north west][inner sep=0.75pt]  [font=\footnotesize] [align=left] {3};
\draw (361,142) node [anchor=north west][inner sep=0.75pt]  [font=\footnotesize] [align=left] {1};
\draw (381,143) node [anchor=north west][inner sep=0.75pt]  [font=\footnotesize] [align=left] {2};
\draw (402,143) node [anchor=north west][inner sep=0.75pt]  [font=\footnotesize] [align=left] {3};
\draw (225,71) node [anchor=north west][inner sep=0.75pt]  [font=\footnotesize] [align=left] {1'};
\draw (244,71) node [anchor=north west][inner sep=0.75pt]  [font=\footnotesize] [align=left] {2'};
\draw (262,71) node [anchor=north west][inner sep=0.75pt]  [font=\footnotesize] [align=left] {3'};
\draw (281,71) node [anchor=north west][inner sep=0.75pt]  [font=\footnotesize] [align=left] {4'};
\draw (351,74) node [anchor=north west][inner sep=0.75pt]  [font=\footnotesize] [align=left] {1'};
\draw (372,74) node [anchor=north west][inner sep=0.75pt]  [font=\footnotesize] [align=left] {2'};
\draw (390,74) node [anchor=north west][inner sep=0.75pt]  [font=\footnotesize] [align=left] {3'};
\draw (409,74) node [anchor=north west][inner sep=0.75pt]  [font=\footnotesize] [align=left] {4'};
\end{tikzpicture}
\end{center}
\end{defi}

    There are some important operations on the sets of partitions (see \cite{comes2011delignes}, \cite{Flake_2021}).
    
\begin{itemize}
\item \textit{Horizontal concatenation}: Let $p\in P(k,l)$ and $q\in P(m,n)$, then $p\otimes q \in P(k+m,l+n)$.
\item \textit{Involution}: If $p\in P(k,l)$, then $p^{*}\in P(l,k)$.
\item \textit{Vertical concatenation}: For $p\in P(k,l)$ and $q\in P(l,m)$ we denote the stacking of $q$ on top of $p$ by $q\star p$ or simply $qp$. We denote the number of loops in $q p$ by $l(q,p)$. 
\end{itemize}

\begin{defi} \label{definition permutation partitions}
       For every $n\in \N$, there is an injective monoid homomorphism $\phi:S_{n} \hookrightarrow P(n,n)$, which is defined by sending the cycle $(1,i)$ to the partition \begin{align*}
        \{\{1,i'\},\{2,2'\},\ldots \{i,1'\},\ldots, \{n,n'\}\} \in P(n,n).
    \end{align*}
    Because the cycles $\{(1,i)|2\leqslant i \leqslant n\}$ generate $S_{n}$, this defines the map $\phi$. We call the partitions in the image of $\phi$ permutation partitions.

\end{defi}

\begin{defi} \label{non-crossing form}
    Let $p\in P(k,l)$ and let $B_{1},\ldots,B_{t}$ be its blocks, which are respectively of size $(k_{1},l_{1}),\ldots, (k_{t},l_{t})$. Define the partitions $p_{i}:= \{\{1,\ldots,k_{i},1',\ldots, l_{i}'\}\} \in P(k_{i},l_{i})$ for all $1\leqslant i\leqslant t$. A non-crossing form of $p$ is some horizontal concatenation $p_{i_{1}}\otimes \ldots \otimes p_{i_{t}}$, where $i_{1},\ldots, i_{t}$ is some permutation of $\{1,\ldots,t\}$.
    
\begin{center}
\tikzset{every picture/.style={line width=0.75pt}} 

\begin{tikzpicture}[x=0.75pt,y=0.75pt,yscale=-1,xscale=1]

\draw  [line width=3] [line join = round][line cap = round] (122.13,310.08) .. controls (122.13,310.08) and (122.13,310.08) .. (122.13,310.08) ;
\draw  [line width=3] [line join = round][line cap = round] (151.12,310.08) .. controls (151.12,310.08) and (151.12,310.08) .. (151.12,310.08) ;
\draw  [line width=3] [line join = round][line cap = round] (179.23,310.08) .. controls (179.23,310.08) and (179.23,310.08) .. (179.23,310.08) ;
\draw  [line width=3] [line join = round][line cap = round] (207.22,310.08) .. controls (207.22,310.08) and (207.22,310.08) .. (207.22,310.08) ;
\draw  [line width=3] [line join = round][line cap = round] (119.13,254.08) .. controls (119.13,254.08) and (119.13,254.08) .. (119.13,254.08) ;
\draw  [line width=3] [line join = round][line cap = round] (149.12,255.08) .. controls (149.12,255.08) and (149.12,255.08) .. (149.12,255.08) ;
\draw  [line width=3] [line join = round][line cap = round] (175.23,254.08) .. controls (175.23,254.08) and (175.23,254.08) .. (175.23,254.08) ;
\draw  [line width=3] [line join = round][line cap = round] (204.22,254.08) .. controls (204.22,254.08) and (204.22,254.08) .. (204.22,254.08) ;
\draw    (180.29,309.96) .. controls (174.29,285.96) and (146.29,276.96) .. (118.29,254.96) ;
\draw    (180.29,309.96) .. controls (186.29,297.96) and (197.29,284.96) .. (207.29,310.96) ;
\draw    (203.29,254.96) -- (121.29,309.96) ;
\draw    (149.29,253.96) .. controls (153.29,271.96) and (173.29,264.96) .. (175.29,253.96) ;
\draw    (149.29,253.96) -- (151.29,309.96) ;
\draw  [line width=3] [line join = round][line cap = round] (306.12,316.08) .. controls (306.12,316.08) and (306.12,316.08) .. (306.12,316.08) ;
\draw  [line width=3] [line join = round][line cap = round] (304.12,260.08) .. controls (304.12,260.08) and (304.12,260.08) .. (304.12,260.08) ;
\draw  [line width=3] [line join = round][line cap = round] (330.23,260.08) .. controls (330.23,260.08) and (330.23,260.08) .. (330.23,260.08) ;
\draw    (304.29,259.96) .. controls (308.29,277.96) and (328.29,270.96) .. (330.29,259.96) ;
\draw    (304.29,259.96) -- (306.29,315.96) ;
\draw  [line width=3] [line join = round][line cap = round] (349.23,315.08) .. controls (349.23,315.08) and (349.23,315.08) .. (349.23,315.08) ;
\draw  [line width=3] [line join = round][line cap = round] (376.22,315.08) .. controls (376.22,315.08) and (376.22,315.08) .. (376.22,315.08) ;
\draw  [line width=3] [line join = round][line cap = round] (347.13,260.08) .. controls (347.13,260.08) and (347.13,260.08) .. (347.13,260.08) ;
\draw    (349.29,314.96) .. controls (348.29,291.96) and (351.29,284.96) .. (347.29,259.96) ;
\draw    (349.29,314.96) .. controls (355.29,302.96) and (366.29,289.96) .. (376.29,315.96) ;
\draw  [line width=3] [line join = round][line cap = round] (391.13,316.08) .. controls (391.13,316.08) and (391.13,316.08) .. (391.13,316.08) ;
\draw  [line width=3] [line join = round][line cap = round] (390.22,262.08) .. controls (390.22,262.08) and (390.22,262.08) .. (390.22,262.08) ;
\draw    (390.29,259.96) -- (392.29,315.96) ;
\draw  [line width=3] [line join = round][line cap = round] (468.23,313.08) .. controls (468.23,313.08) and (468.23,313.08) .. (468.23,313.08) ;
\draw  [line width=3] [line join = round][line cap = round] (495.22,313.08) .. controls (495.22,313.08) and (495.22,313.08) .. (495.22,313.08) ;
\draw  [line width=3] [line join = round][line cap = round] (466.13,258.08) .. controls (466.13,258.08) and (466.13,258.08) .. (466.13,258.08) ;
\draw    (468.29,312.96) .. controls (467.29,289.96) and (470.29,282.96) .. (466.29,257.96) ;
\draw    (468.29,312.96) .. controls (474.29,300.96) and (485.29,287.96) .. (495.29,313.96) ;
\draw  [line width=3] [line join = round][line cap = round] (443.13,313.08) .. controls (443.13,313.08) and (443.13,313.08) .. (443.13,313.08) ;
\draw  [line width=3] [line join = round][line cap = round] (442.22,259.08) .. controls (442.22,259.08) and (442.22,259.08) .. (442.22,259.08) ;
\draw    (442.29,256.96) -- (444.29,312.96) ;
\draw  [line width=3] [line join = round][line cap = round] (510.12,315.08) .. controls (510.12,315.08) and (510.12,315.08) .. (510.12,315.08) ;
\draw  [line width=3] [line join = round][line cap = round] (508.12,259.08) .. controls (508.12,259.08) and (508.12,259.08) .. (508.12,259.08) ;
\draw  [line width=3] [line join = round][line cap = round] (534.23,259.08) .. controls (534.23,259.08) and (534.23,259.08) .. (534.23,259.08) ;
\draw    (508.29,258.96) .. controls (512.29,276.96) and (532.29,269.96) .. (534.29,258.96) ;
\draw    (508.29,258.96) -- (510.29,314.96) ;

\draw (157,332.07) node [anchor=north west][inner sep=0.75pt]    {$p$};
\draw (481,279.07) node [anchor=north west][inner sep=0.75pt]    {$\otimes $};
\draw (446,278.07) node [anchor=north west][inner sep=0.75pt]    {$\otimes $};
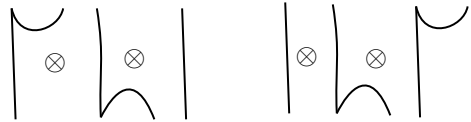
\draw (294,333.67) node [anchor=north west][inner sep=0.75pt]   [align=left] {Examples of non-crossing forms of $\displaystyle p$};
\draw (319,281.07) node [anchor=north west][inner sep=0.75pt]    {$\otimes $};
\draw (359,279.07) node [anchor=north west][inner sep=0.75pt]    {$\otimes $};

\end{tikzpicture}
\end{center}
\end{defi}

\begin{prop} \label{partition can be written by permutation partitions}
    Every partition $p\in P(k,l)$ can be written as $\phi(\sigma) \circ p' \circ \phi(\rho)$, where $\sigma \in S_{l}$, $\rho \in S_{k}$ and $p'$ is some non-crossing form of $p$. We call $\phi(\sigma) \circ p' \circ \phi(\rho)$ a normal form of $p$.
\end{prop}

\begin{defi}
    Let $k,l \in \N$. We call a partition $p\in P(k,l)$ even if all its components contain an even number of vertices. We denote the set of all even partitions in $P(k,l)$ by $P_{even}(k,l)$ and the set of all even partitions by $P_{even}:=\sqcup_{k,l\in\mathbb{N}} P_{even}(k,l)\subset P$. 
\end{defi}

The proofs of the following two propositions are elementary (see \cite{Tyriard} for details).

\begin{prop} \label{peven closed under operations}
    $P_{even}$ is closed under involution, vertical concatenation and horizontal concatenation.
\end{prop}

\begin{prop} \label{even partition sn tn form}
    Every even partition $p\in P_{even}$ can be constructed by applying the involution, horizontal and vertical concatenation operations to the partitions \\\CROSS$=\{\{1,2'\},\{2,1'\}\}$, \FOURLEGS= $\{1,2,1',2'\}$, \idstrand=$\{1,1'\}$ and \CAP=$\{\{1,2\}\}$.  

\end{prop}

\begin{defi}
    We denote the set of $\Z_{2}$-coloured partitions by $P_{\Z_{2}}= \sqcup P_{\Z_{2}}(k,l)$. An element of $(p,z)\in P_{\Z_{2}}(k,l)$ is pair consisting of a partition $p \in P(k,l)$ and a vector $z\in (\Z_{2})^{k+l}$. Such a $\Z_{2}$-coloured partition can be visualised by a partition diagram with labeled vertices. Unlabeled vertices will be assumed to be labeled with $1$. Two $\Z_{2}$-coloured partitions are equivalent, denoted by $\simeq$, when the the corresponding partition diagrams are the same and when for each block the labels of one $\Z_{2}$-coloured partition are obtained by multiplying all labels of the correseponding block in the other $\Z_{2}$-coloured partition by the same element of  $\Z_{2}=\{-1,1\}$.
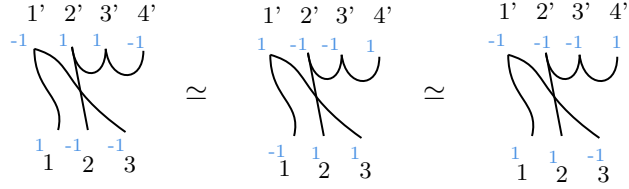
    \begin{center}
        
\tikzset{every picture/.style={line width=0.75pt}} 

\begin{tikzpicture}[x=0.75pt,y=0.75pt,yscale=-1,xscale=1]

\draw  [line width=3] [line join = round][line cap = round] (165.33,81.92) .. controls (165.33,81.92) and (165.33,81.92) .. (165.33,81.92) ;
\draw  [line width=3] [line join = round][line cap = round] (185.33,81.92) .. controls (185.33,81.92) and (185.33,81.92) .. (185.33,81.92) ;
\draw  [line width=3] [line join = round][line cap = round] (202.33,81.92) .. controls (202.33,81.92) and (202.33,81.92) .. (202.33,81.92) ;
\draw  [line width=3] [line join = round][line cap = round] (222.33,81.92) .. controls (222.33,81.92) and (222.33,81.92) .. (222.33,81.92) ;
\draw  [line width=3] [line join = round][line cap = round] (176.33,122.92) .. controls (176.33,122.92) and (176.33,122.92) .. (176.33,122.92) ;
\draw  [line width=3] [line join = round][line cap = round] (193.33,122.92) .. controls (193.33,122.92) and (193.33,122.92) .. (193.33,122.92) ;
\draw  [line width=3] [line join = round][line cap = round] (213.33,122.92) .. controls (213.33,122.92) and (213.33,122.92) .. (213.33,122.92) ;
\draw    (178.43,122.86) .. controls (182.43,105.86) and (166.43,103.86) .. (166.43,81.86) ;
\draw    (166.43,81.86) .. controls (187.43,87.86) and (178.43,99.86) .. (212.43,123.86) ;
\draw    (185.43,80.86) -- (193.43,122.86) ;
\draw    (185.43,80.86) .. controls (187.43,97.86) and (202.43,98.86) .. (202.43,81.86) ;
\draw    (202.43,81.86) .. controls (204.43,98.86) and (221.43,99.86) .. (221.43,82.86) ;
\draw  [line width=3] [line join = round][line cap = round] (284.33,84.92) .. controls (284.33,84.92) and (284.33,84.92) .. (284.33,84.92) ;
\draw  [line width=3] [line join = round][line cap = round] (304.33,84.92) .. controls (304.33,84.92) and (304.33,84.92) .. (304.33,84.92) ;
\draw  [line width=3] [line join = round][line cap = round] (321.33,84.92) .. controls (321.33,84.92) and (321.33,84.92) .. (321.33,84.92) ;
\draw  [line width=3] [line join = round][line cap = round] (341.33,84.92) .. controls (341.33,84.92) and (341.33,84.92) .. (341.33,84.92) ;
\draw  [line width=3] [line join = round][line cap = round] (295.33,125.92) .. controls (295.33,125.92) and (295.33,125.92) .. (295.33,125.92) ;
\draw  [line width=3] [line join = round][line cap = round] (312.33,125.92) .. controls (312.33,125.92) and (312.33,125.92) .. (312.33,125.92) ;
\draw  [line width=3] [line join = round][line cap = round] (332.33,125.92) .. controls (332.33,125.92) and (332.33,125.92) .. (332.33,125.92) ;
\draw    (297.43,125.86) .. controls (301.43,108.86) and (285.43,106.86) .. (285.43,84.86) ;
\draw    (285.43,84.86) .. controls (306.43,90.86) and (297.43,102.86) .. (331.43,126.86) ;
\draw    (304.43,83.86) -- (312.43,125.86) ;
\draw    (304.43,83.86) .. controls (306.43,100.86) and (321.43,101.86) .. (321.43,84.86) ;
\draw    (321.43,84.86) .. controls (323.43,101.86) and (340.43,102.86) .. (340.43,85.86) ;
\draw  [line width=3] [line join = round][line cap = round] (404.33,84.92) .. controls (404.33,84.92) and (404.33,84.92) .. (404.33,84.92) ;
\draw  [line width=3] [line join = round][line cap = round] (424.33,84.92) .. controls (424.33,84.92) and (424.33,84.92) .. (424.33,84.92) ;
\draw  [line width=3] [line join = round][line cap = round] (441.33,84.92) .. controls (441.33,84.92) and (441.33,84.92) .. (441.33,84.92) ;
\draw  [line width=3] [line join = round][line cap = round] (461.33,84.92) .. controls (461.33,84.92) and (461.33,84.92) .. (461.33,84.92) ;
\draw  [line width=3] [line join = round][line cap = round] (415.33,125.92) .. controls (415.33,125.92) and (415.33,125.92) .. (415.33,125.92) ;
\draw  [line width=3] [line join = round][line cap = round] (432.33,125.92) .. controls (432.33,125.92) and (432.33,125.92) .. (432.33,125.92) ;
\draw  [line width=3] [line join = round][line cap = round] (452.33,125.92) .. controls (452.33,125.92) and (452.33,125.92) .. (452.33,125.92) ;
\draw    (417.43,125.86) .. controls (421.43,108.86) and (405.43,106.86) .. (405.43,84.86) ;
\draw    (405.43,84.86) .. controls (426.43,90.86) and (417.43,102.86) .. (451.43,126.86) ;
\draw    (424.43,83.86) -- (432.43,125.86) ;
\draw    (424.43,83.86) .. controls (426.43,100.86) and (441.43,101.86) .. (441.43,84.86) ;
\draw    (441.43,84.86) .. controls (443.43,101.86) and (460.43,102.86) .. (460.43,85.86) ;

\draw (154,171) node [anchor=north west][inner sep=0.75pt]  [font=\scriptsize] [align=left] {Example:};
\draw (203,170.4) node [anchor=north west][inner sep=0.75pt]  [font=\scriptsize]  {$(\{\{1,3,1'\} ,\{2,2',3',4'\}\} ,( 1,-1,-1,-1,1,1,-1)) \in P_{\{\mathbb{Z}_{\{2\}}\}}( 3,4)$};
\draw (169,134) node [anchor=north west][inner sep=0.75pt]  [font=\footnotesize] [align=left] {1};
\draw (189,135) node [anchor=north west][inner sep=0.75pt]  [font=\footnotesize] [align=left] {2};
\draw (210,135) node [anchor=north west][inner sep=0.75pt]  [font=\footnotesize] [align=left] {3};
\draw (161,58) node [anchor=north west][inner sep=0.75pt]  [font=\footnotesize] [align=left] {1'};
\draw (180,58) node [anchor=north west][inner sep=0.75pt]  [font=\footnotesize] [align=left] {2'};
\draw (198,58) node [anchor=north west][inner sep=0.75pt]  [font=\footnotesize] [align=left] {3'};
\draw (217,58) node [anchor=north west][inner sep=0.75pt]  [font=\footnotesize] [align=left] {4'};
\draw (194,73) node [anchor=north west][inner sep=0.75pt]  [font=\tiny,color={rgb, 255:red, 73; green, 145; blue, 228 }  ,opacity=1 ] [align=left] {1};
\draw (165,126) node [anchor=north west][inner sep=0.75pt]  [font=\tiny,color={rgb, 255:red, 73; green, 145; blue, 228 }  ,opacity=1 ] [align=left] {1};
\draw (153,73) node [anchor=north west][inner sep=0.75pt]  [font=\tiny,color={rgb, 255:red, 73; green, 145; blue, 228 }  ,opacity=1 ] [align=left] {\mbox{-}1};
\draw (180.43,125.86) node [anchor=north west][inner sep=0.75pt]  [font=\tiny,color={rgb, 255:red, 73; green, 145; blue, 228 }  ,opacity=1 ] [align=left] {\mbox{-}1};
\draw (202,126) node [anchor=north west][inner sep=0.75pt]  [font=\tiny,color={rgb, 255:red, 73; green, 145; blue, 228 }  ,opacity=1 ] [align=left] {\mbox{-}1};
\draw (177,73) node [anchor=north west][inner sep=0.75pt]  [font=\tiny,color={rgb, 255:red, 73; green, 145; blue, 228 }  ,opacity=1 ] [align=left] {1};
\draw (212,74) node [anchor=north west][inner sep=0.75pt]  [font=\tiny,color={rgb, 255:red, 73; green, 145; blue, 228 }  ,opacity=1 ] [align=left] {\mbox{-}1};
\draw (288,137) node [anchor=north west][inner sep=0.75pt]  [font=\footnotesize] [align=left] {1};
\draw (308,138) node [anchor=north west][inner sep=0.75pt]  [font=\footnotesize] [align=left] {2};
\draw (329,138) node [anchor=north west][inner sep=0.75pt]  [font=\footnotesize] [align=left] {3};
\draw (280,59) node [anchor=north west][inner sep=0.75pt]  [font=\footnotesize] [align=left] {1'};
\draw (299,59) node [anchor=north west][inner sep=0.75pt]  [font=\footnotesize] [align=left] {2'};
\draw (317,59) node [anchor=north west][inner sep=0.75pt]  [font=\footnotesize] [align=left] {3'};
\draw (336,59) node [anchor=north west][inner sep=0.75pt]  [font=\footnotesize] [align=left] {4'};
\draw (310,75) node [anchor=north west][inner sep=0.75pt]  [font=\tiny,color={rgb, 255:red, 73; green, 145; blue, 228 }  ,opacity=1 ] [align=left] {\mbox{-}1};
\draw (283,129) node [anchor=north west][inner sep=0.75pt]  [font=\tiny,color={rgb, 255:red, 73; green, 145; blue, 228 }  ,opacity=1 ] [align=left] {\mbox{-}1};
\draw (277,75) node [anchor=north west][inner sep=0.75pt]  [font=\tiny,color={rgb, 255:red, 73; green, 145; blue, 228 }  ,opacity=1 ] [align=left] {1};
\draw (304.43,129.86) node [anchor=north west][inner sep=0.75pt]  [font=\tiny,color={rgb, 255:red, 73; green, 145; blue, 228 }  ,opacity=1 ] [align=left] {1};
\draw (324,130) node [anchor=north west][inner sep=0.75pt]  [font=\tiny,color={rgb, 255:red, 73; green, 145; blue, 228 }  ,opacity=1 ] [align=left] {1};
\draw (294,75) node [anchor=north west][inner sep=0.75pt]  [font=\tiny,color={rgb, 255:red, 73; green, 145; blue, 228 }  ,opacity=1 ] [align=left] {\mbox{-}1};
\draw (332,76) node [anchor=north west][inner sep=0.75pt]  [font=\tiny,color={rgb, 255:red, 73; green, 145; blue, 228 }  ,opacity=1 ] [align=left] {1};
\draw (241,99.4) node [anchor=north west][inner sep=0.75pt]    {$\simeq $};
\draw (408,139) node [anchor=north west][inner sep=0.75pt]  [font=\footnotesize] [align=left] {1};
\draw (428,140) node [anchor=north west][inner sep=0.75pt]  [font=\footnotesize] [align=left] {2};
\draw (449,140) node [anchor=north west][inner sep=0.75pt]  [font=\footnotesize] [align=left] {3};
\draw (399,57) node [anchor=north west][inner sep=0.75pt]  [font=\footnotesize] [align=left] {1'};
\draw (418,57) node [anchor=north west][inner sep=0.75pt]  [font=\footnotesize] [align=left] {2'};
\draw (436,57) node [anchor=north west][inner sep=0.75pt]  [font=\footnotesize] [align=left] {3'};
\draw (455,57) node [anchor=north west][inner sep=0.75pt]  [font=\footnotesize] [align=left] {4'};
\draw (433,74) node [anchor=north west][inner sep=0.75pt]  [font=\tiny,color={rgb, 255:red, 73; green, 145; blue, 228 }  ,opacity=1 ] [align=left] {\mbox{-}1};
\draw (405,130) node [anchor=north west][inner sep=0.75pt]  [font=\tiny,color={rgb, 255:red, 73; green, 145; blue, 228 }  ,opacity=1 ] [align=left] {1};
\draw (394,73) node [anchor=north west][inner sep=0.75pt]  [font=\tiny,color={rgb, 255:red, 73; green, 145; blue, 228 }  ,opacity=1 ] [align=left] {\mbox{-}1};
\draw (424,132) node [anchor=north west][inner sep=0.75pt]  [font=\tiny,color={rgb, 255:red, 73; green, 145; blue, 228 }  ,opacity=1 ] [align=left] {1};
\draw (444,131) node [anchor=north west][inner sep=0.75pt]  [font=\tiny,color={rgb, 255:red, 73; green, 145; blue, 228 }  ,opacity=1 ] [align=left] {\mbox{-}1};
\draw (417,74) node [anchor=north west][inner sep=0.75pt]  [font=\tiny,color={rgb, 255:red, 73; green, 145; blue, 228 }  ,opacity=1 ] [align=left] {\mbox{-}1};
\draw (456,74) node [anchor=north west][inner sep=0.75pt]  [font=\tiny,color={rgb, 255:red, 73; green, 145; blue, 228 }  ,opacity=1 ] [align=left] {1};
\draw (361,99.4) node [anchor=north west][inner sep=0.75pt]    {$\simeq $};

\end{tikzpicture}

\end{center}

Involution, horizontal concatenation, and loops are defined for the $\Z_{2}$-coloured partitions as for the non-coloured partitions. For the vertical concatenation we have to be more careful. Let $(p,z_{1}) \in P_{\Z_{2}}(k,l)$ and $(q,z_{2}) \in  P_{\Z_{2}}(l,m)$. As for the non-coloured partitions we let $p\in P(k,l)$ and $q\in P(l,m)$ with the sets of partitions containing partitions of $\{1,\ldots,k,1',\ldots,l'\}$ and $\{1',\ldots,l',1'',\ldots,m''\}$ respectively. We label the elements $i'\in \{1',\ldots,l'\}$ in the middle row of the stacking $q\star p$ by $z_{1}^{i'}z_{2}^{i'}$. We call $(p,z_{1})$ and $(q,z_{2})$ compatible if for each $i',j' \in  \{1',\ldots,l'\}$ in $q \star p$ that lie in the same connected component of $q$, we have that $z_{1}^{i'}z_{2}^{i'}=z_{1}^{j'}z_{2}^{j'}$. In this case the vertical concatenation $(q,z_{2})(p,z_{1}):= (qp, z_{2}z_{1})$ is given by the vertical concatenation of the underlying non-coloured partitions, which is then labeled by setting $(z_{2}z_{1})^{i}:=z_{1}^{i}$ for $i\in \{1,\ldots, k\}$ and $(z_{2}z_{1})^{j''}:=z_{2}^{j''}g$ for $j''\in \{1'',\ldots,m''\}$, where $g$ is the labeling of a vertex in the middle row of $q\star p$ which is connected to $j''$ in $q$. If there are no such vertices we set $g=1$.

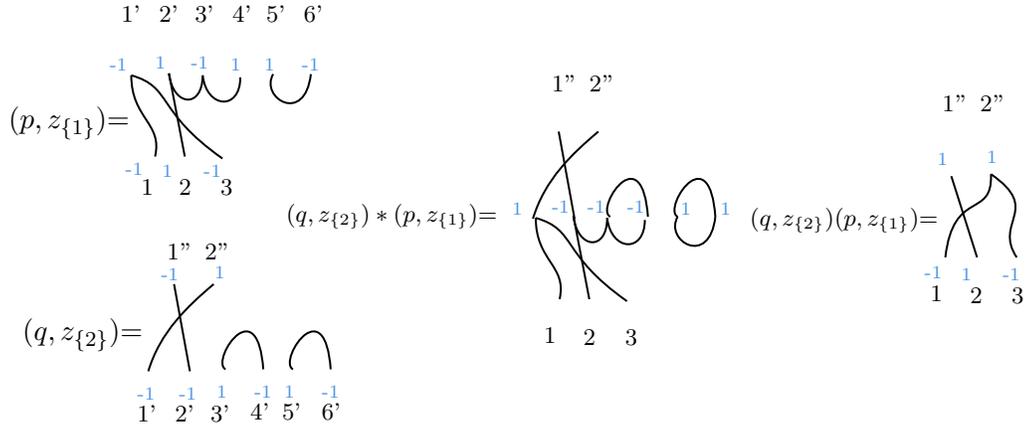
\begin{center}

\tikzset{every picture/.style={line width=0.75pt}} 

\begin{tikzpicture}[x=0.75pt,y=0.75pt,yscale=-1,xscale=1]

\draw  [line width=3] [line join = round][line cap = round] (93.33,62.92) .. controls (93.33,62.92) and (93.33,62.92) .. (93.33,62.92) ;
\draw  [line width=3] [line join = round][line cap = round] (113.33,62.92) .. controls (113.33,62.92) and (113.33,62.92) .. (113.33,62.92) ;
\draw  [line width=3] [line join = round][line cap = round] (130.33,62.92) .. controls (130.33,62.92) and (130.33,62.92) .. (130.33,62.92) ;
\draw  [line width=3] [line join = round][line cap = round] (150.33,62.92) .. controls (150.33,62.92) and (150.33,62.92) .. (150.33,62.92) ;
\draw  [line width=3] [line join = round][line cap = round] (104.33,103.92) .. controls (104.33,103.92) and (104.33,103.92) .. (104.33,103.92) ;
\draw  [line width=3] [line join = round][line cap = round] (121.33,103.92) .. controls (121.33,103.92) and (121.33,103.92) .. (121.33,103.92) ;
\draw  [line width=3] [line join = round][line cap = round] (141.33,103.92) .. controls (141.33,103.92) and (141.33,103.92) .. (141.33,103.92) ;
\draw    (106.43,103.86) .. controls (110.43,86.86) and (94.43,84.86) .. (94.43,62.86) ;
\draw    (94.43,62.86) .. controls (115.43,68.86) and (106.43,80.86) .. (140.43,104.86) ;
\draw    (113.43,61.86) -- (121.43,103.86) ;
\draw    (113.43,61.86) .. controls (115.43,78.86) and (130.43,79.86) .. (130.43,62.86) ;
\draw    (130.43,62.86) .. controls (132.43,79.86) and (149.43,80.86) .. (149.43,63.86) ;
\draw  [line width=3] [line join = round][line cap = round] (103.33,211.92) .. controls (103.33,211.92) and (103.33,211.92) .. (103.33,211.92) ;
\draw  [line width=3] [line join = round][line cap = round] (123.33,211.92) .. controls (123.33,211.92) and (123.33,211.92) .. (123.33,211.92) ;
\draw  [line width=3] [line join = round][line cap = round] (140.33,211.92) .. controls (140.33,211.92) and (140.33,211.92) .. (140.33,211.92) ;
\draw  [line width=3] [line join = round][line cap = round] (160.33,211.92) .. controls (160.33,211.92) and (160.33,211.92) .. (160.33,211.92) ;
\draw  [line width=3] [line join = round][line cap = round] (116.33,167.92) .. controls (116.33,167.92) and (116.33,167.92) .. (116.33,167.92) ;
\draw  [line width=3] [line join = round][line cap = round] (136.33,167.92) .. controls (136.33,167.92) and (136.33,167.92) .. (136.33,167.92) ;
\draw    (136,168.25) .. controls (123,179.25) and (109,192.25) .. (103,212.25) ;
\draw  [line width=3] [line join = round][line cap = round] (165.33,62.92) .. controls (165.33,62.92) and (165.33,62.92) .. (165.33,62.92) ;
\draw  [line width=3] [line join = round][line cap = round] (185.33,62.92) .. controls (185.33,62.92) and (185.33,62.92) .. (185.33,62.92) ;
\draw    (185,62.25) .. controls (182,88.25) and (159,74.25) .. (166,62.25) ;
\draw    (161,211.25) .. controls (159.6,192.15) and (153.65,189.98) .. (148.49,193.51) .. controls (142.58,197.56) and (137.73,209.11) .. (142,211.25) ;
\draw  [line width=3] [line join = round][line cap = round] (174.33,211.92) .. controls (174.33,211.92) and (174.33,211.92) .. (174.33,211.92) ;
\draw  [line width=3] [line join = round][line cap = round] (194.33,211.92) .. controls (194.33,211.92) and (194.33,211.92) .. (194.33,211.92) ;
\draw    (195,211.25) .. controls (193.6,192.15) and (187.65,189.98) .. (182.49,193.51) .. controls (176.58,197.56) and (171.73,209.11) .. (176,211.25) ;
\draw    (116,168.25) -- (124,212.25) ;
\draw  [line width=3] [line join = round][line cap = round] (297.33,134.92) .. controls (297.33,134.92) and (297.33,134.92) .. (297.33,134.92) ;
\draw  [line width=3] [line join = round][line cap = round] (317.33,134.92) .. controls (317.33,134.92) and (317.33,134.92) .. (317.33,134.92) ;
\draw  [line width=3] [line join = round][line cap = round] (334.33,134.92) .. controls (334.33,134.92) and (334.33,134.92) .. (334.33,134.92) ;
\draw  [line width=3] [line join = round][line cap = round] (354.33,134.92) .. controls (354.33,134.92) and (354.33,134.92) .. (354.33,134.92) ;
\draw  [line width=3] [line join = round][line cap = round] (310.33,90.92) .. controls (310.33,90.92) and (310.33,90.92) .. (310.33,90.92) ;
\draw  [line width=3] [line join = round][line cap = round] (330.33,90.92) .. controls (330.33,90.92) and (330.33,90.92) .. (330.33,90.92) ;
\draw    (330,91.25) .. controls (317,102.25) and (303,115.25) .. (297,135.25) ;
\draw    (355,134.25) .. controls (353.6,115.15) and (347.65,112.98) .. (342.49,116.51) .. controls (336.58,120.56) and (331.73,132.11) .. (336,134.25) ;
\draw  [line width=3] [line join = round][line cap = round] (368.33,134.92) .. controls (368.33,134.92) and (368.33,134.92) .. (368.33,134.92) ;
\draw  [line width=3] [line join = round][line cap = round] (388.33,134.92) .. controls (388.33,134.92) and (388.33,134.92) .. (388.33,134.92) ;
\draw    (389,134.25) .. controls (387.6,115.15) and (381.65,112.98) .. (376.49,116.51) .. controls (370.58,120.56) and (365.73,132.11) .. (370,134.25) ;
\draw    (310,91.25) -- (318,135.25) ;
\draw  [line width=3] [line join = round][line cap = round] (297.33,134.92) .. controls (297.33,134.92) and (297.33,134.92) .. (297.33,134.92) ;
\draw  [line width=3] [line join = round][line cap = round] (317.33,134.92) .. controls (317.33,134.92) and (317.33,134.92) .. (317.33,134.92) ;
\draw  [line width=3] [line join = round][line cap = round] (334.33,134.92) .. controls (334.33,134.92) and (334.33,134.92) .. (334.33,134.92) ;
\draw  [line width=3] [line join = round][line cap = round] (354.33,134.92) .. controls (354.33,134.92) and (354.33,134.92) .. (354.33,134.92) ;
\draw  [line width=3] [line join = round][line cap = round] (308.33,175.92) .. controls (308.33,175.92) and (308.33,175.92) .. (308.33,175.92) ;
\draw  [line width=3] [line join = round][line cap = round] (325.33,175.92) .. controls (325.33,175.92) and (325.33,175.92) .. (325.33,175.92) ;
\draw  [line width=3] [line join = round][line cap = round] (345.33,175.92) .. controls (345.33,175.92) and (345.33,175.92) .. (345.33,175.92) ;
\draw    (310.43,175.86) .. controls (314.43,158.86) and (298.43,156.86) .. (298.43,134.86) ;
\draw    (298.43,134.86) .. controls (319.43,140.86) and (310.43,152.86) .. (344.43,176.86) ;
\draw    (317.43,133.86) -- (325.43,175.86) ;
\draw    (317.43,133.86) .. controls (319.43,150.86) and (334.43,151.86) .. (334.43,134.86) ;
\draw    (334.43,134.86) .. controls (336.43,151.86) and (353.43,152.86) .. (353.43,135.86) ;
\draw  [line width=3] [line join = round][line cap = round] (369.33,134.92) .. controls (369.33,134.92) and (369.33,134.92) .. (369.33,134.92) ;
\draw  [line width=3] [line join = round][line cap = round] (389.33,134.92) .. controls (389.33,134.92) and (389.33,134.92) .. (389.33,134.92) ;
\draw    (389,134.25) .. controls (386,160.25) and (363,146.25) .. (370,134.25) ;
\draw  [line width=3] [line join = round][line cap = round] (504.33,154.92) .. controls (504.33,154.92) and (504.33,154.92) .. (504.33,154.92) ;
\draw  [line width=3] [line join = round][line cap = round] (521.33,154.92) .. controls (521.33,154.92) and (521.33,154.92) .. (521.33,154.92) ;
\draw  [line width=3] [line join = round][line cap = round] (541.33,154.92) .. controls (541.33,154.92) and (541.33,154.92) .. (541.33,154.92) ;
\draw  [line width=3] [line join = round][line cap = round] (507.33,112.92) .. controls (507.33,112.92) and (507.33,112.92) .. (507.33,112.92) ;
\draw  [line width=3] [line join = round][line cap = round] (527.33,112.92) .. controls (527.33,112.92) and (527.33,112.92) .. (527.33,112.92) ;
\draw    (508,113.75) -- (521,154.75) ;
\draw    (505,154.75) .. controls (509,119.75) and (529,135.75) .. (528,112.75) ;
\draw    (528,112.75) .. controls (554,130.75) and (529,136.75) .. (541,154.75) ;

\draw (98,114) node [anchor=north west][inner sep=0.75pt]  [font=\footnotesize] [align=left] {1};
\draw (117,114) node [anchor=north west][inner sep=0.75pt]  [font=\footnotesize] [align=left] {2};
\draw (138,114) node [anchor=north west][inner sep=0.75pt]  [font=\footnotesize] [align=left] {3};
\draw (88,26) node [anchor=north west][inner sep=0.75pt]  [font=\footnotesize] [align=left] {1'};
\draw (107,26) node [anchor=north west][inner sep=0.75pt]  [font=\footnotesize] [align=left] {2'};
\draw (125,26) node [anchor=north west][inner sep=0.75pt]  [font=\footnotesize] [align=left] {3'};
\draw (144,26) node [anchor=north west][inner sep=0.75pt]  [font=\footnotesize] [align=left] {4'};
\draw (161,263) node [anchor=north west][inner sep=0.75pt]  [font=\scriptsize] [align=left] {Example: Vertical Concatenation of compatible diagrams with $\displaystyle l( p,q) =1.$};
\draw (31,77) node [anchor=north west][inner sep=0.75pt]   [align=left] {$\displaystyle ( p,z_{\{1\}})$=};
\draw (96,228) node [anchor=north west][inner sep=0.75pt]  [font=\footnotesize] [align=left] {1'};
\draw (115,228) node [anchor=north west][inner sep=0.75pt]  [font=\footnotesize] [align=left] {2'};
\draw (133,228) node [anchor=north west][inner sep=0.75pt]  [font=\footnotesize] [align=left] {3'};
\draw (153,227) node [anchor=north west][inner sep=0.75pt]  [font=\footnotesize] [align=left] {4'};
\draw (111,146) node [anchor=north west][inner sep=0.75pt]  [font=\footnotesize] [align=left] {1''};
\draw (130,146) node [anchor=north west][inner sep=0.75pt]  [font=\footnotesize] [align=left] {2''};
\draw (161,26) node [anchor=north west][inner sep=0.75pt]  [font=\footnotesize] [align=left] {5'};
\draw (180,26) node [anchor=north west][inner sep=0.75pt]  [font=\footnotesize] [align=left] {6'};
\draw (169,227) node [anchor=north west][inner sep=0.75pt]  [font=\footnotesize] [align=left] {5'};
\draw (189,227) node [anchor=north west][inner sep=0.75pt]  [font=\footnotesize] [align=left] {6'};
\draw (305,61) node [anchor=north west][inner sep=0.75pt]  [font=\footnotesize] [align=left] {1''};
\draw (324,61) node [anchor=north west][inner sep=0.75pt]  [font=\footnotesize] [align=left] {2''};
\draw (301,189) node [anchor=north west][inner sep=0.75pt]  [font=\footnotesize] [align=left] {1};
\draw (321,190) node [anchor=north west][inner sep=0.75pt]  [font=\footnotesize] [align=left] {2};
\draw (342,190) node [anchor=north west][inner sep=0.75pt]  [font=\footnotesize] [align=left] {3};
\draw (171,126) node [anchor=north west][inner sep=0.75pt]  [font=\footnotesize] [align=left] {$\displaystyle ( q,z_{\{2\}}) *( p,z_{\{1\}})$=};
\draw (496,168) node [anchor=north west][inner sep=0.75pt]  [font=\footnotesize] [align=left] {1};
\draw (516,169) node [anchor=north west][inner sep=0.75pt]  [font=\footnotesize] [align=left] {2};
\draw (537,169) node [anchor=north west][inner sep=0.75pt]  [font=\footnotesize] [align=left] {3};
\draw (502,71) node [anchor=north west][inner sep=0.75pt]  [font=\footnotesize] [align=left] {1''};
\draw (521,71) node [anchor=north west][inner sep=0.75pt]  [font=\footnotesize] [align=left] {2''};
\draw (90,106) node [anchor=north west][inner sep=0.75pt]  [font=\tiny,color={rgb, 255:red, 73; green, 145; blue, 228 }  ,opacity=1 ] [align=left] {\mbox{-}1};
\draw (108.43,106.86) node [anchor=north west][inner sep=0.75pt]  [font=\tiny,color={rgb, 255:red, 73; green, 145; blue, 228 }  ,opacity=1 ] [align=left] { 1};
\draw (129.43,106.86) node [anchor=north west][inner sep=0.75pt]  [font=\tiny,color={rgb, 255:red, 73; green, 145; blue, 228 }  ,opacity=1 ] [align=left] {\mbox{-}1};
\draw (82,53) node [anchor=north west][inner sep=0.75pt]  [font=\tiny,color={rgb, 255:red, 73; green, 145; blue, 228 }  ,opacity=1 ] [align=left] {\mbox{-}1};
\draw (105,52) node [anchor=north west][inner sep=0.75pt]  [font=\tiny,color={rgb, 255:red, 73; green, 145; blue, 228 }  ,opacity=1 ] [align=left] { 1};
\draw (123,52) node [anchor=north west][inner sep=0.75pt]  [font=\tiny,color={rgb, 255:red, 73; green, 145; blue, 228 }  ,opacity=1 ] [align=left] {\mbox{-}1};
\draw (143,53) node [anchor=north west][inner sep=0.75pt]  [font=\tiny,color={rgb, 255:red, 73; green, 145; blue, 228 }  ,opacity=1 ] [align=left] {1};
\draw (160,53) node [anchor=north west][inner sep=0.75pt]  [font=\tiny,color={rgb, 255:red, 73; green, 145; blue, 228 }  ,opacity=1 ] [align=left] { 1};
\draw (179,53) node [anchor=north west][inner sep=0.75pt]  [font=\tiny,color={rgb, 255:red, 73; green, 145; blue, 228 }  ,opacity=1 ] [align=left] {\mbox{-}1};
\draw (108,159) node [anchor=north west][inner sep=0.75pt]  [font=\tiny,color={rgb, 255:red, 73; green, 145; blue, 228 }  ,opacity=1 ] [align=left] {\mbox{-}1};
\draw (135,158) node [anchor=north west][inner sep=0.75pt]  [font=\tiny,color={rgb, 255:red, 73; green, 145; blue, 228 }  ,opacity=1 ] [align=left] {1};
\draw (96,219) node [anchor=north west][inner sep=0.75pt]  [font=\tiny,color={rgb, 255:red, 73; green, 145; blue, 228 }  ,opacity=1 ] [align=left] {\mbox{-}1};
\draw (117,219) node [anchor=north west][inner sep=0.75pt]  [font=\tiny,color={rgb, 255:red, 73; green, 145; blue, 228 }  ,opacity=1 ] [align=left] {\mbox{-}1};
\draw (136,218) node [anchor=north west][inner sep=0.75pt]  [font=\tiny,color={rgb, 255:red, 73; green, 145; blue, 228 }  ,opacity=1 ] [align=left] { 1};
\draw (155,218) node [anchor=north west][inner sep=0.75pt]  [font=\tiny,color={rgb, 255:red, 73; green, 145; blue, 228 }  ,opacity=1 ] [align=left] {\mbox{-}1};
\draw (170,218) node [anchor=north west][inner sep=0.75pt]  [font=\tiny,color={rgb, 255:red, 73; green, 145; blue, 228 }  ,opacity=1 ] [align=left] {1};
\draw (189,218) node [anchor=north west][inner sep=0.75pt]  [font=\tiny,color={rgb, 255:red, 73; green, 145; blue, 228 }  ,opacity=1 ] [align=left] { \mbox{-}1};
\draw (285,126) node [anchor=north west][inner sep=0.75pt]  [font=\tiny,color={rgb, 255:red, 73; green, 145; blue, 228 }  ,opacity=1 ] [align=left] {1};
\draw (305,125) node [anchor=north west][inner sep=0.75pt]  [font=\tiny,color={rgb, 255:red, 73; green, 145; blue, 228 }  ,opacity=1 ] [align=left] {\mbox{-}1};
\draw (323,125) node [anchor=north west][inner sep=0.75pt]  [font=\tiny,color={rgb, 255:red, 73; green, 145; blue, 228 }  ,opacity=1 ] [align=left] {\mbox{-}1};
\draw (343,125) node [anchor=north west][inner sep=0.75pt]  [font=\tiny,color={rgb, 255:red, 73; green, 145; blue, 228 }  ,opacity=1 ] [align=left] {\mbox{-}1};
\draw (370,126) node [anchor=north west][inner sep=0.75pt]  [font=\tiny,color={rgb, 255:red, 73; green, 145; blue, 228 }  ,opacity=1 ] [align=left] {1};
\draw (390,126) node [anchor=north west][inner sep=0.75pt]  [font=\tiny,color={rgb, 255:red, 73; green, 145; blue, 228 }  ,opacity=1 ] [align=left] {1};
\draw (493,158) node [anchor=north west][inner sep=0.75pt]  [font=\tiny,color={rgb, 255:red, 73; green, 145; blue, 228 }  ,opacity=1 ] [align=left] {\mbox{-}1};
\draw (511.43,158.86) node [anchor=north west][inner sep=0.75pt]  [font=\tiny,color={rgb, 255:red, 73; green, 145; blue, 228 }  ,opacity=1 ] [align=left] { 1};
\draw (532.43,158.86) node [anchor=north west][inner sep=0.75pt]  [font=\tiny,color={rgb, 255:red, 73; green, 145; blue, 228 }  ,opacity=1 ] [align=left] {\mbox{-}1};
\draw (524.43,99.86) node [anchor=north west][inner sep=0.75pt]  [font=\tiny,color={rgb, 255:red, 73; green, 145; blue, 228 }  ,opacity=1 ] [align=left] { 1};
\draw (499.43,100.86) node [anchor=north west][inner sep=0.75pt]  [font=\tiny,color={rgb, 255:red, 73; green, 145; blue, 228 }  ,opacity=1 ] [align=left] { 1};
\draw (38,184) node [anchor=north west][inner sep=0.75pt]   [align=left] {$\displaystyle ( q,z_{\{2\}})$=};
\draw (405,128) node [anchor=north west][inner sep=0.75pt]  [font=\footnotesize] [align=left] {$\displaystyle ( q,z_{\{2\}})( p,z_{\{1\}})$=};
\end{tikzpicture}
\end{center}
\end{defi}

\begin{prop} \label{compatibilityproposition}
    Let $(p,z_{1}),(p,z_{3})\in P_{\Z_{2}}(k,l)$ and $(q,z_{2}),(q,z_{4})\in P_{\Z_{2}}(l,m)$, so that $((q,z_{2}),(p,z_{1}))$ and $((q,z_{4}),(p,z_{3}))$ are compatible pairs. Let $(p,z_{1})\simeq (p,z_{3})$ and $(q,z_{2})\simeq(q,z_{4})$, then $(q,z_{2})(p,z_{1})\simeq (q,z_{4})(p,z_{3})$. So equivalences of $\Z_{2}$-coloured partitions are respected by vertical concatenation.
\end{prop}

\begin{defi}
    Let $(p,z_{1})\in P_{\Z_{2}}(k,l)$ and $(q,z_{2})\in P_{\Z_{2}}(l,m)$. We call a pair of equivalence classes $([(q,z_{2})],[(p,z_{1})])$ compatible if there exist representatives which are compatible. We define 
    \begin{align*}
        [(q,z_{2})][(p,z_{1})]:=[(q,z_{2})(p,z_{1})].
    \end{align*}
\end{defi}

\begin{rem} \label{identify equivalence classes colored partitions}
    From now on we will consider equivalent $\Z_{2}$-coloured partitions to be equal.
\end{rem}

\begin{rem} \label{Special form of coloured partitions}
    We can write every $\Z_{2}$-coloured partition diagram $(p,z)$ in the following form 
    \begin{equation*} 
    (p,z) =
    \begin{tikzpicture}[centerzero]
        \pd{-0.4,-0.2};
        \pd{0.4,-0.2};
        \pd{0.8,-0.2};
        \pd{-0.4,0.2} node[anchor=south] {\dotlabel{z_{1'}}};
        \pd{0.4,0.2} node[anchor=south] {\dotlabel{z_{l-1'}}};
        \pd{0.8,0.2} node[anchor=south] {\dotlabel{z_{l'}}};
        \draw (-0.4,-0.2) -- (-0.4,0.2);
        \draw (0.4,-0.2) -- (0.4,0.2);
        \draw (0.8,-0.2) -- (0.8,0.2);
        \node at (0.04,0) {$\cdots$};
    \end{tikzpicture}
    \circ
    p
    \circ
    \begin{tikzpicture}[centerzero]
        \pd{-0.4,-0.2} node[anchor=north] {\dotlabel{z_1}};
        \pd{0.4,-0.2} node[anchor=north] {\dotlabel{z_{k-1}}};
        \pd{0.8,-0.2} node[anchor=north] {\dotlabel{z_k}};
        \pd{-0.4,0.2};
        \pd{0.4,0.2};
        \pd{0.8,0.2};
        \draw (-0.4,-0.2) -- (-0.4,0.2);
        \draw (0.4,-0.2) -- (0.4,0.2);
        \draw (0.8,-0.2) -- (0.8,0.2);
        \node at (0.04,0) {$\cdots$};
    \end{tikzpicture}
    \ .
\end{equation*}
This will be useful in reducing problems involving $\Z_{2}$-coloured partitions to problems containing only non-coloured partitions.
\end{rem}

\begin{defi} \label{definition normal form Z2 coloured}
     Let $(p,z)\in P_{\Z_{2}}(k,l)$. We extend the permutation partitions from Definition \ref{definition permutation partitions} to $\Z_{2}$-coloured partitions by $\phi:S_{n} \xhookrightarrow{} P(n,n) \xhookrightarrow{} P_{\Z_{2}}(n,n)$. If $p'$ is some non-crossing form of $p$ such that $p=\phi(\sigma) \circ p' \circ \phi(\rho)$, then we call
    \begin{equation*} 
    (p,z) =
    \begin{tikzpicture}[centerzero]
        \pd{-0.4,-0.2};
        \pd{0.4,-0.2};
        \pd{0.8,-0.2};
        \pd{-0.4,0.2} node[anchor=south] {\dotlabel{z_{1'}}};
        \pd{0.4,0.2} node[anchor=south] {\dotlabel{z_{l-1'}}};
        \pd{0.8,0.2} node[anchor=south] {\dotlabel{z_{l'}}};
        \draw (-0.4,-0.2) -- (-0.4,0.2);
        \draw (0.4,-0.2) -- (0.4,0.2);
        \draw (0.8,-0.2) -- (0.8,0.2);
        \node at (0.04,0) {$\cdots$};
    \end{tikzpicture}
    \circ
    \phi(\sigma) \circ p' \circ \phi(\rho)
    \circ
    \begin{tikzpicture}[centerzero]
        \pd{-0.4,-0.2} node[anchor=north] {\dotlabel{z_1}};
        \pd{0.4,-0.2} node[anchor=north] {\dotlabel{z_{k-1}}};
        \pd{0.8,-0.2} node[anchor=north] {\dotlabel{z_k}};
        \pd{-0.4,0.2};
        \pd{0.4,0.2};
        \pd{0.8,0.2};
        \draw (-0.4,-0.2) -- (-0.4,0.2);
        \draw (0.4,-0.2) -- (0.4,0.2);
        \draw (0.8,-0.2) -- (0.8,0.2);
        \node at (0.04,0) {$\cdots$};
    \end{tikzpicture}
    \
\end{equation*}
     a normal form of $(p,z)$.
\end{defi}

\subsection{The permutation category of Knop and Likeng-Savage}

We use the definition of $\text{Par}(\Z_{2},t)^{Kar}$ that is given in \cite[Definition 3.6]{Nyobe_Likeng_2021}.

\begin{defi}
    Define the category $\text{Par}(\Z_{2},t)$ to be the category with objects $\{[\tilde{n}]|n\in \N\}$ and morphism spaces $\C$-linear combinations of equivalence classes of $\Z_{2}$-coloured partition diagrams, so
    \begin{align*}
        \Hom_{\text{Par}(\Z_{2},t)}([\tilde{k}],[\tilde{l}])=\C P_{\Z_{2}}(k,l),
    \end{align*}
     where we identify equivalence classes of $Z_{2}$-coloured partitions, according to Remark \ref{identify equivalence classes colored partitions}. The identity of an object $[\tilde{n}]$ is given by
    \begin{align*}
        id_{[\tilde{n}]}:=(\{\{1,1'\},\ldots,\{n,n'\}\},(1,\ldots,1))
    \end{align*} 
    and the composition is given by extending the rule
    \begin{align*}
        &\circ: \C P_{\Z_{2}}(l,m) \times \C P_{\Z_{2}}(k,l) \to \C P_{\Z_{2}}(k,m)\\
        &(q,p) \mapsto q\circ p := t^{l(q,p)}qp.
    \end{align*}
    $\C$-bilinearly. The category has a tensor product, which is given on objects by $[\tilde{n}]\otimes[\tilde{m}]:=[\tilde{n+m}]$ and on morphisms by a $\C$-bilinear extension of the horizontal concatenation of $\Z_{2}$-coloured partitions. This then extends to the additive Karoubi envelope.
\end{defi}

\begin{prop} \label{Par is spherical}
    $\text{Par}(\Z_{2},t)$ is a strict $\C$-linear spherical rigid symmetric monoidal category.
\end{prop}

\begin{defi}
    We define the Deligne category for the hyperoctahedral groups in the permutation representation for $t\in \C$ by $ \text{Par}(\Z_{2},t)^{Kar}$ and denote the corresponding $\C$-linear full embedding by $\iota_{H}:\text{Par}(\Z_{2},t) \to \text{Par}(\Z_{2},t)^{Kar}$. Proposition \ref{Par is spherical} and Proposition \ref{Karoubian envelope spherical} imply that $\text{Par}(\Z_{2},t)^{Kar}$ is a $\C$-linear spherical rigid symmetric monoidal category.
\end{defi}

\subsection{The reflection category of Flake and Maassen}

The definition of $\underline{\text{Rep}}(H_{t})$ we use can be found in \cite[Definition 2.5]{Flake_2021}. 

\begin{defi}
    Define the category $\underline{\text{Rep}_{0}}(H_{t})$ to be the category with objects $\{[n]|n\in \N\}$ and morphism spaces $\C$-linear combinations of even partition diagrams, so
    \begin{align*}
        \Hom_{\underline{\text{Rep}_{0}}(H_{t})}([k],[l])=\C P_{even}(k,l).
    \end{align*}
      The identity of an object $[n]$ is given by
    \begin{align*}
        id_{[n]}:=\{\{1,1'\},\ldots,\{n,n'\}\}
    \end{align*} 
    and the composition by extending
    \begin{align*}
        &\circ: \C P_{even}(l,m) \times \C P_{even}(k,l) \to \C P_{even}(k,m)\\
        &(q,p) \mapsto q\circ p := t^{l(q,p)}qp
    \end{align*}
    $\C$-bilinearly. The category has a tensor product, which is given on objects by $[n]\otimes[m]:=[n+m]$ and on morphisms by a $\C$-bilinear extension of the horizontal concatenation of partitions.
\end{defi}

\begin{prop} \label{RepH0 spherical}
    $\underline{\text{Rep}_{0}}(H_{t})$ is strict $\C$-linear spherical rigid symmetric monoidal category.
\end{prop}

\begin{defi}
    We define the Deligne category for the hyperoctahedral groups in the reflection representation for $t\in \C$ by $\underline{\text{Rep}}(H_{t}):= (\underline{\text{Rep}}_{0}(H_{t}))^{Kar}$ and denote the corresponding $\C$-linear full embedding by $\iota_{G}:\underline{\text{Rep}}_{0}(H_{t}) \to \underline{\text{Rep}}(H_{t})$. Proposition \ref{Karoubian envelope spherical} shows us that $\underline{\text{Rep}}(H_{t})$ is a $\C$-linear spherical rigid symmetric monoidal category.
\end{defi}



\section{Universal properties}

Presentations for $Rep(S_t)$ and $Par(\mathbb{Z}_2,t)$ were derived in \cite{comes2017jellyfish} and \cite{Nyobe_Likeng_2021}. We derive a presentation for $Rep(H_t)$ and use this to state the universal properties for the two Deligne categories of the hyperoctahedral group.

\subsection{A presentation for the permutation category}


\begin{defi} \label{definition generators relations parZ2}
    $\text{Par}(\mathbb{Z}_{2})/\sim_{t}$ is a strict $\mathbb{C}$-linear monoidal category with a single generating object $W$. The generating morphisms are
    \begin{align*}
         &\idstrand=id_{W}: W\to W, \ 
         &\merge:W\otimes W \to W\\ 
         &\spliter:W \to W \otimes W, \
         &\crossing: W \otimes W \to W \otimes W\\
         &\tokstrand: W \to W, \quad g\in \mathbb{Z}_{2},\ 
         &\bottompin: \textbf{1} \to W\\
         &\toppin:W \to \textbf{1}, \ 
         &\lolly:\textbf{1} \to \textbf{1} \\
    \end{align*}
with the following relations for $g,h \in \Z_{2}$:
  \begin{gather} \label{GPC1}
        \begin{tikzpicture}[anchorbase]
            \draw (-0.4,-0.4) -- (0,0);
            \draw (0.25,-0.25) -- (0,0);
            \draw (0,0) -- (0,0.5);
            \opendot{0.25,-0.25};
        \end{tikzpicture}
        =
        \begin{tikzpicture}[anchorbase]
            \draw (0,-0.45) -- (0,0.45);
        \end{tikzpicture}
        =
        \begin{tikzpicture}[anchorbase]
            \draw (0.4,-0.4) -- (0,0);
            \draw (-0.25,-0.25) -- (0,0);
            \draw (0,0) -- (0,0.5);
            \opendot{-0.25,-0.25};
        \end{tikzpicture}
        \ ,\quad
        \begin{tikzpicture}[anchorbase]
            \draw (-0.4,0.4) -- (0,0);
            \draw (0.25,0.25) -- (0,0);
            \draw (0,0) -- (0,-0.5);
            \opendot{0.25,0.25};
        \end{tikzpicture}
        =
        \begin{tikzpicture}[anchorbase]
            \draw (0,0.45) -- (0,-0.45);
        \end{tikzpicture}
        =
        \begin{tikzpicture}[anchorbase]
            \draw (-0.25,0.25) -- (0,0);
            \draw (0.4,0.4) -- (0,0);
            \draw (0,0) -- (0,-0.5);
            \opendot{-0.25,0.25};
        \end{tikzpicture}
        \ ,\quad
        \begin{tikzpicture}[anchorbase]
            \draw (-0.4,-0.4) -- (-0.4,0) -- (-0.2,0.2) -- (0.2,-0.2) -- (0.4,0) -- (0.4,0.4);
            \draw (-0.2,0.2) -- (-0.2,0.4);
            \draw (0.2,-0.2) -- (0.2,-0.4);
        \end{tikzpicture}
        =
        \begin{tikzpicture}[anchorbase]
            \draw (-0.2,-0.4) -- (0,-0.2) -- (0.2,-0.4);
            \draw (-0.2,0.4) -- (0,0.2) -- (0.2,0.4);
            \draw (0,-0.2) -- (0,0.2);
        \end{tikzpicture}
        =
        \begin{tikzpicture}[anchorbase]
            \draw (0.4,-0.4) -- (0.4,0) -- (0.2,0.2) -- (-0.2,-0.2) -- (-0.4,0) -- (-0.4,0.4);
            \draw (0.2,0.2) -- (0.2,0.4);
            \draw (-0.2,-0.2) -- (-0.2,-0.4);
        \end{tikzpicture}
        \ ,
        \\ \label{GPC2}
        \begin{tikzpicture}[anchorbase]
            \draw (-0.2,-0.4) \braidto (0.2,0) \braidto (-0.2,0.4);
            \draw (0.2,-0.4) \braidto (-0.2,0) \braidto (0.2,0.4);
        \end{tikzpicture}
        =
        \begin{tikzpicture}[anchorbase]
            \draw (-0.2,-0.4) -- (-0.2,0.4);
            \draw (0.2,-0.4) -- (0.2,0.4);
        \end{tikzpicture}
        \ ,\quad
        \begin{tikzpicture}[anchorbase]
            \draw (-0.4,-0.4) -- (0.4,0.4);
            \draw (0.4,-0.4) -- (-0.4,0.4);
            \draw (0,-0.4) \braidto (-0.4,0) \braidto (0,0.4);
        \end{tikzpicture}
        =
        \begin{tikzpicture}[anchorbase]
            \draw (-0.4,-0.4) -- (0.4,0.4);
            \draw (0.4,-0.4) -- (-0.4,0.4);
            \draw (0,-0.4) \braidto (0.4,0) \braidto (0,0.4);
        \end{tikzpicture}
        \ ,
        \\ \label{GPC3}
        \begin{tikzpicture}[anchorbase]
            \draw (0.2,-0.2) -- (-0.3,0.3);
            \draw (-0.3,-0.3) -- (0.3,0.3);
            \opendot{0.2,-0.2};
        \end{tikzpicture}
        =
        \begin{tikzpicture}[anchorbase]
            \draw (-0.2,0) to (-0.2,0.3);
            \draw (0.2,-0.3) to (0.2,0.3);
            \opendot{-0.2,0};
        \end{tikzpicture}
        \ ,\quad
        \begin{tikzpicture}[anchorbase]
            \draw (0.2,0.2) -- (-0.3,-0.3);
            \draw (-0.3,0.3) -- (0.3,-0.3);
            \opendot{0.2,0.2};
        \end{tikzpicture}
        =
        \begin{tikzpicture}[anchorbase]
            \draw (-0.2,0) to (-0.2,-0.3);
            \draw (0.2,0.3) to (0.2,-0.3);
            \opendot{-0.2,0};
        \end{tikzpicture}
        \ ,\quad
        \begin{tikzpicture}[anchorbase]
            \draw (-0.4,-0.5) -- (0.2,0.3) -- (0.4,0.1) -- (0,-0.5);
            \draw (0.2,0.3) -- (0.2,0.6);
            \draw (-0.4,0.6) -- (-0.4,0.1) -- (0.4,-0.5);
        \end{tikzpicture}
        =
        \begin{tikzpicture}[anchorbase]
            \draw (-0.4,-0.4) -- (-0.2,-0.2) -- (-0.2,0) -- (0.2,0.4);
            \draw (0,-0.4) -- (-0.2,-0.2);
            \draw (0.2,-0.4) -- (0.2,0) -- (-0.2,0.4);
        \end{tikzpicture}
        \ ,\quad
        \begin{tikzpicture}[anchorbase]
            \draw (-0.4,0.5) -- (0.2,-0.3) -- (0.4,-0.1) -- (0,0.5);
            \draw (0.2,-0.3) -- (0.2,-0.6);
            \draw (-0.4,-0.6) -- (-0.4,-0.1) -- (0.4,0.5);
        \end{tikzpicture}
        =
        \begin{tikzpicture}[anchorbase]
            \draw (-0.4,0.4) -- (-0.2,0.2) -- (-0.2,0) -- (0.2,-0.4);
            \draw (0,0.4) -- (-0.2,0.2);
            \draw (0.2,0.4) -- (0.2,0) -- (-0.2,-0.4);
        \end{tikzpicture}
        \ ,
        \\ \label{GPC4}
        \begin{tikzpicture}[anchorbase]
          \draw (-0.2,-0.4) to[out=45,in=down] (0.2,0) to[out=up,in=-45] (0,0.2) -- (0,0.5);
          \draw (0.2,-0.4) to[out=135,in=down] (-0.2,0) to[out=up,in=225] (0,0.2);
        \end{tikzpicture}
        =
        \merge
        \ ,\quad
        \begin{tikzpicture}[anchorbase]
            \draw (0,-0.5) -- (0,-0.3) to[out=135,in=down] (-0.2,-0.1) -- (-0.2,0.1) to[out=up,in=225] (0,0.3) -- (0,0.5);
            \draw (0,-0.3) to[out=45,in=down] (0.2,-0.1) -- (0.2,0.1) to[out=up,in=-45] (0,0.3);
            \token{east}{-0.2,0}{g};
            \token{west}{0.2,0}{h};
        \end{tikzpicture}
        = \delta_{g,h}\
        \begin{tikzpicture}[anchorbase]
            \draw (0,-0.4) -- (0,0.4);
            \token{west}{0,0}{g};
        \end{tikzpicture}
        \ ,\quad
        \lolly=t,
        \\ \label{GPC5}
        \begin{tikzpicture}[anchorbase]
            \draw (0,-0.3) -- (0,0.3);
            \token{east}{0,-0.15}{h};
            \token{east}{0,0.15}{g};
        \end{tikzpicture}
        =
        \begin{tikzpicture}[anchorbase]
            \draw (0,-0.3) -- (0,0.3);
            \token{east}{0,0}{gh};
        \end{tikzpicture}
        \ ,\quad
        \begin{tikzpicture}[anchorbase]
            \draw (0,-0.3) -- (0,0.3);
            \token{east}{0,0}{1};
        \end{tikzpicture}
        =
        \begin{tikzpicture}[anchorbase]
            \draw (0,-0.3) -- (0,0.3);
        \end{tikzpicture}
        \ ,\quad
        \begin{tikzpicture}[anchorbase]
            \draw (-0.3,-0.3) -- (0.3,0.3);
            \draw (0.3,-0.3) -- (-0.3,0.3);
            \token{east}{-0.15,-0.15}{g};
        \end{tikzpicture}
        =
        \begin{tikzpicture}[anchorbase]
            \draw (-0.3,-0.3) -- (0.3,0.3);
            \draw (0.3,-0.3) -- (-0.3,0.3);
            \token{west}{0.15,0.15}{g};
        \end{tikzpicture}
        \ ,\quad
        \begin{tikzpicture}[anchorbase]
            \draw (0,-0.3) -- (0,0) -- (0.25,0.25);
            \draw (0,0) -- (-0.25,0.25);
            \token{east}{0,-0.15}{g};
        \end{tikzpicture}
        =
        \begin{tikzpicture}[anchorbase]
            \draw (0,-0.3) -- (0,0) -- (0.25,0.25);
            \draw (0,0) -- (-0.25,0.25);
            \token{east}{-0.125,0.125}{g};
            \token{west}{0.125,0.125}{g};
        \end{tikzpicture}
        \ ,\quad
        \begin{tikzpicture}[anchorbase]
            \draw (0,-0.3) -- (0,0.3);
            \opendot{0,-0.3};
            \token{east}{0,0}{g};
        \end{tikzpicture}
        =
        \bottompin
        \ .
    \end{gather}
\end{defi}

\begin{rem}
    It can be proven that all the possible horizontal and vertical reflections of the above relations will also hold, as a consequence of the given relations, see \cite[Proposition 4.3]{Nyobe_Likeng_2021}.
\end{rem}

%
%
%
%

\begin{thm} \label{Functor H tilde} [\cite{Nyobe_Likeng_2021}]
    Let $t\in \mathbb{C}$. The categories $\text{Par}(\Z_{2},t)$ and $\text{Par}(\mathbb{Z}_{2})/\sim_{t}$ are isomorphic as $\C$-linear symmetric monoidal categories.
\end{thm}

The isomorphism is given by the functor $\tilde{H}: \text{Par}(\Z_{2})/\sim_{t} \to \text{Par}(\Z_{2},t)$. It is defined on objects by $W^{\otimes k} \mapsto [k]$, which makes it clearly bijective on the objects. On the morphisms we define it by sending:
 \begin{gather*}
        \merge \mapsto
        \begin{tikzpicture}[anchorbase]
            \pd{-0.2,0};
            \pd{0.2,0};
            \pd{0,0.4};
            \draw (-0.2,0) \braidto (0,0.4);
            \draw (0.2,0) \braidto (0,0.4);
        \end{tikzpicture}
        \ ,\quad
        \spliter \mapsto
        \begin{tikzpicture}[anchorbase]
            \pd{-0.2,0.4};
            \pd{0.2,0.4};
            \pd{0,0};
            \draw (0,0) \braidto (-0.2,0.4);
            \draw (0,0) \braidto (0.2,0.4);
        \end{tikzpicture}
        \ ,\quad
        \crossing \mapsto
        \begin{tikzpicture}[anchorbase]
            \pd{-0.2,0.4};
            \pd{0.2,0.4};
            \pd{-0.2,0};
            \pd{0.2,0};
            \draw (0.2,0) -- (-0.2,0.4);
            \draw (-0.2,0) -- (0.2,0.4);
        \end{tikzpicture}
        \ ,\quad
        \bottompin \mapsto
        \begin{tikzpicture}[centerzero]
            \pd{0,0.2};
        \end{tikzpicture}
        \ ,\quad
        \toppin \mapsto
        \begin{tikzpicture}[centerzero]
            \pd{0,-0.2};
        \end{tikzpicture}
        \ ,\quad
        \tokstrand \mapsto
        \begin{tikzpicture}[centerzero]
            \pd{0,-0.2} node[anchor=north] {\dotlabel{g}};
            \pd{0,0.2};
            \draw (0,-0.2) -- (0,0.2);
        \end{tikzpicture}
        =
        \begin{tikzpicture}[centerzero]
            \pd{0,-0.2};
            \pd{0,0.2} node[anchor=south] {\dotlabel{g}};
            \draw (0,-0.2) -- (0,0.2);
        \end{tikzpicture}
        \ 
    \end{gather*}
and extending this rule $\C$-linearly. Remember that unlabeled partitions are labeled by 1.

From this presentation one can deduce that $\text{Par}(\Z_{2},t)$ satisfies the following universal property.

\begin{thm} \label{real universal property Par Z2}
    Let $t\in \C$. Let $\mathcal{C}$ be a $\C$-linear symmetric monoidal category with a $t$-dimensional special commutative Frobenius object with involution $(A,\alpha,\delta,\beta,\epsilon,\zeta)$ for an object $A\in Ob(\mathcal{C})$, a multiplication $\alpha: A\otimes A \to A$, a comultiplication $\beta: A \to A \otimes A$, a unit $\delta: \textbf{1} \to A$, a counit $\epsilon: A \to \textbf{1}$ and an involution $\zeta: A \to A$. Let $\gamma:=s^{\mathcal{C}}_{A,A}: A \otimes A \to A\otimes A$ be the swap morphism. These morphisms satisfy:\\
        1.\text{ $A$ is a Frobenius object } 
        \begin{align*}
        \alpha \circ (id_{A} \otimes \delta) &= id_{A}= \alpha \circ (\delta \otimes id_{A}),\\
        (id_{A} \otimes \epsilon) \circ \beta &= id_{A} = (\epsilon \otimes id_{A}) \circ \beta,\\
        (\alpha \otimes id_{A}) \circ (id_{A} \otimes \beta) &= \beta \circ \alpha = (id_{A} \otimes \alpha) \circ ( \beta \otimes id_{A}).
        \end{align*}                   
        2.\text{ $A$ is commutative}:$\alpha \circ \gamma = \alpha$.
        3.\text{ $A$ is special}: $\alpha \circ \beta = id_{A}$.
        4.\text{ $A$ has dimension $t$}: $\epsilon \circ \delta = t$.
        5. $\zeta$ is an involution: $\zeta \circ \zeta =id_{A}$.
        6. As a special commutative Frobenius object, $A$ is compatible with the involution $\gamma$
        \begin{align*}
        \beta \circ \zeta &= (\zeta \otimes \zeta) \circ \beta, \\
         \zeta \circ \delta &= \delta,\\
        \alpha \circ (\zeta \otimes \zeta) \circ \beta &= \zeta,\\
         \alpha \circ (id_{A} \otimes \zeta) \circ \beta &= 0 =\alpha \circ (\zeta \otimes \id_{A}) \circ \beta.
        \end{align*}
        
Then there exists a unique strict $\C$-linear symmetric monoidal functor  $\mathcal{H}': \text{Par}(\Z_{2},t) \to \mathcal{C}$, with $\mathcal{H}'([\tilde{k}]) = A^{\otimes k}$ for all $k\in \N$ and $\mathcal{H}'(\tilde{H}(\merge))= \alpha$, $ \mathcal{H}'(\tilde{H}(\spliter)) =\beta$, $ \mathcal{H}'(\tilde{H}(\CROSS))= \gamma$, $ \mathcal{H}'(\tilde{H}(\bottompin))= \delta$, $\mathcal{H}'(\tilde{H}(\toppin))= \epsilon$ and $\mathcal{H}'(\tilde{H}(\tokstrand[-1]))= \zeta$. Furthermore, if $\mathcal{C}$ is Karoubi, there exists an up to isomorphism uniqe $\C$-linear symmetric monoidal functor $\mathcal{H}:  \text{Par}(\Z_{2},t)^{Kar} \to \mathcal{C}$ with $\mathcal{H}'=\mathcal{H}\circ \iota_{H}$.
\end{thm}
\begin{proof}
    We define the functor $\mathcal{H}'': \text{Par}(\mathbb{Z}_{2})/\sim_{t} \to \mathcal{C}$ on objects by $\mathcal{H}''(W^{\otimes k}):= A^{\otimes k}$ and on morphisms by the $\C$-linear extension of
    $\mathcal{H}''(\merge):= \alpha, \mathcal{H}''(\spliter) :=\beta, \mathcal{H}''(\CROSS):= \gamma$, $\mathcal{H}''(\bottompin):= \delta$, $\mathcal{H}''(\toppin):= \epsilon$ and $\mathcal{H}''(\tokstrand[-1]):= \zeta$. Because $\mathcal{C}$ is a symmetric monoidal category the relations (4.2),(4.3) and the third relation of (4.5) hold. A direct verification shows that all the necessary relations are satisfied. This shows that the functor $\mathcal{H}''$ is a well-defined strict $\C$-linear symmetric monoidal functor.
    We use Theorem \ref{Functor H tilde} and define $\mathcal{H}' := \mathcal{H}'' \circ \tilde{H}^{-1}$, which is the wanted $C$-linear strict symmetric monoidal functor. If $\mathcal{C}$ is Karoubi, the universal property of the Karoubian envelope implies the existence of the functor $\mathcal{H}$.
\end{proof}

\begin{cor} \label{functors category for Parz2}
    Let $t\in \C$ and $\mathcal{C}$ be a $\C$-linear symmetric monoidal Karoubi category. Then there is an equivalence
    \begin{align*}
        \text{Fun}^{\otimes,Symm}_{\C}(\text{Par}(\Z_{2},t)^{Kar},\mathcal{C}) \simeq \text{Frob}_{\C}^{Spec,inv}(\mathcal{C},t)
    \end{align*}
    between the category $\text{Fun}^{\otimes,Symm}_{\C}(\text{Par}(\Z_{2},t)^{Kar},\mathcal{C})$ of $\C$-linear symmetric monoidal functors from $\text{Par}(\Z_{2},t)^{Kar}$ to $\mathcal{C}$ with natural isomorphisms between them, and the subcategory $\text{Frob}_{\C}^{Spec,inv}(\mathcal{C},t)$ of $\mathcal{C}$ consisting of $t$-dimensional special commutative Frobenius objects with involution and their isomorphisms.
\end{cor}
\begin{proof}
    We want to construct an inverse for the functor 
    \begin{align*}
        \text{Fun}^{\otimes,Symm}_{\C}(\text{Par}(\Z_{2},t)^{Kar},\mathcal{C}) &\to \text{Frob}_{\C}^{Spec,inv}(\mathcal{C},t) \end{align*}
defined by \begin{align*}
        \mathcal{H} &\mapsto \mathcal{G}([\tilde{1}]) \text{ for all functors } \mathcal{H}:\text{Par}(\Z_{2},t)^{Kar} \to \mathcal{C} \text{ and }\\
        \kappa &\mapsto \kappa_{A} \text{ for all natural isomorphisms } \kappa:\mathcal{H}_{1} \to \mathcal{H}_{2}.
    \end{align*}

    Let $A$ be a $t$-dimensional special commutative Frobenius object with involution in $\mathcal{C}$. Then theorem \ref{real universal property Par Z2} implies the existence of a functor $\mathcal{H}^{A}:\text{Par}(\Z_{2},t)^{Kar}\to \mathcal{C}$ with the properties that were stated in the theorem.
    Let $\kappa: A \to B$ be an isomorphism in $\text{Frob}_{\C}^{Spec,inv}(\mathcal{C},t)$. We define a natural isomorphism $\mathcal{H}^{\kappa}:\mathcal{H}^{A} \to \mathcal{H}^{B}$ as followed. First we define a natural isomorphism $\mathcal{H}^{\kappa,0}$ between $\mathcal{H}^{A,0}:=\mathcal{G}^{A} \circ \iota_{H}: \text{Par}(\Z_{2},t)\to \mathcal{C}$ and $\mathcal{H}^{B,0}:= \mathcal{H}^{B} \circ \iota_{H}:\text{Par}(\Z_{2},t) \to \mathcal{C}$ by 
    \begin{align*}
        \mathcal{H}^{\kappa,0}_{[k]}:=\kappa^{\otimes k}:A^{\otimes k} \to B^{\otimes k}
    \end{align*}
     for all $k\in \N$. Let $f:[k]\to [l]$ be a morphism in $\text{Par}(\Z_{2},t)$. Because $\kappa$ respects the structure of $A$ and $B$, the following square commutes
    
    \begin{equation*} \label{commutative phi square}
\begin{tikzcd}
A^{\otimes k} \arrow{r}{\mathcal{H}^{A,0}(f)} \arrow[d,"\kappa^{\otimes k}"] & A^{\otimes l} \arrow{d}{\kappa^{\otimes l}} \ \\  B^{\otimes k} \arrow{r}{\mathcal{H}^{B,0}(f)} & B^{\otimes l}.
\end{tikzcd}
\end{equation*}
    This shows that $\mathcal{H}^{k,0}$ is a natural isomorphism. By the properties of the universal property of the Karoubian envelope it extends uniquely to a natural isomorphism $\mathcal{H}^{\kappa}:\mathcal{H}^{A} \to \mathcal{G}^{B}$ such that $\mathcal{H}^{\kappa,0}=\mathcal{H}^{\kappa}\circ \iota_{H}$.   
    We obtained an inverse functor 
    \begin{align*}
        \text{Frob}_{\C}^{Spec,inv}(\mathcal{C},t) &\to \text{Fun}^{\otimes,Symm}_{\C}(\text{Par}(\Z_{2},t)^{Kar},\mathcal{C}),
    \end{align*} 
    and this shows the equivalence.
\end{proof}


\subsection{A presentation for the reflection category}

In this section we describe $\underline{\text{Rep}}_{0}(H_{t})$ using generators and relations and use this result to give a universal property of $\underline{\text{Rep}}(H_{t})$. The description in Definition \ref{generator relations repH0 definition} is a new one. The structure on the generating object is a bit different from the previous case, but it is still possible to discuss the universal property of this category, see Corollary \ref{Functor categories for Ht}. It will turn out to be a very useful tool in proving theorem \ref{Omega0 exists!}.

\begin{defi} \label{generator relations repH0 definition}
    $\text{Par}_{t}$ is a strict $\mathbb{C}$-linear monoidal category with a single generating object $W$. The generating morphisms are 
    \begin{align*}
         &\idstrand=id_{W}: W\to W, \ 
         & \begin{tikzpicture}[anchorbase]
            \draw (-0.2,-0.4) -- (0,-0.2) -- (0.2,-0.4);
            \draw (-0.2,0.4) -- (0,0.2) -- (0.2,0.4);
            \draw (0,-0.2) -- (0,0.2);
        \end{tikzpicture} :W\otimes W \to W\otimes W\\ 
        & \begin{tikzpicture}[anchorbase]
            \draw (-0.2,-0.4) -- (0,-0.2) -- (0.2,-0.4);
            \draw (0,-0.2) -- (0,0.2);
            \opendot{0,0.2};
        \end{tikzpicture}: W \otimes W \to \textbf{1}, \
        & \CROSS: W\otimes W \to W \otimes W\\
        & \begin{tikzpicture}[anchorbase]
            \draw (-0.2,0.4) -- (0,0.2) -- (0.2,0.4);
            \draw (0,-0.2) -- (0,0.2);
            \opendot{0,-0.2};
        \end{tikzpicture}: \textbf{1} \to W \otimes W\\
    \end{align*}
with the following relations and the relations obtained by reflecting them horizontally and/or vertically:
  \begin{gather} \label{GPC1}
  \begin{tikzpicture}[anchorbase]
            \draw (-0.2,-0.4) -- (0,-0.2) -- (0.2,-0.4);
            \draw (-0.2,0.4) -- (0,0.2) -- (0.2,0.4);
            \draw (-0.2,0.4) -- (0,0.6) -- (0.2,0.4);
            \draw (-0.2,1.2) -- (0,1.0) -- (0.2,1.2);
            \draw (0,-0.2) -- (0,0.2);
            \draw (0,0.6) -- (0,1.0);
        \end{tikzpicture}
        =
        \begin{tikzpicture}[anchorbase]
            \draw (-0.2,-0.4) -- (0,-0.2) -- (0.2,-0.4);
            \draw (-0.2,0.4) -- (0,0.2) -- (0.2,0.4);
            \draw (0,-0.2) -- (0,0.2);
        \end{tikzpicture}
        \ , \quad
        \begin{tikzpicture}[anchorbase]
            \draw (-0.2,-0.4) -- (0,-0.2) -- (0.2,-0.4);
            \draw (-0.2,0.4) -- (0,0.2) -- (0.2,0.4);
            \draw (0,-0.2) -- (0,0.2);
            \draw (0.2,0.4) -- (0.4,0.6) -- (0.6,0.4);
            \draw (0.2,1.2) -- (0.4,1.0) -- (0.6,1.2);
            \draw (0.4,0.6) -- (0.4,1.0);
            \draw (-0.2,0.4) -- (-0.2,1.2);
            \draw (0.6,-0.4) -- (0.6,0.4);
        \end{tikzpicture}
        =
        \begin{tikzpicture}[anchorbase]
            \draw (0.2,-0.4) -- (0.4,-0.2) -- (0.6,-0.4);
            \draw (0.2,0.4) -- (0.4,0.2) -- (0.6,0.4);
            \draw (0.4,-0.2) -- (0.4,0.2);
            \draw (-0.2,0.4) -- (0,0.6) -- (0.2,0.4);
            \draw (-0.2,1.2) -- (0,1.0) -- (0.2,1.2);
            \draw (0,0.6) -- (0,1.0);
            \draw (0.6,0.4) -- (0.6,1.2);
            \draw (-0.2,-0.4) -- (-0.2,0.4);
        \end{tikzpicture}
        \ , \quad
        \begin{tikzpicture}[anchorbase]
            \draw (-0.2,0.4) -- (0,0.2) -- (0.2,0.4);
            \draw (-0.2,0.4) -- (0,0.6) -- (0.2,0.4);
            \draw (0,-0.2) -- (0,0.2);
            \draw (0,0.6) -- (0,1.0);
            \opendot{0,-0.2};
            \opendot{0,1.0};
        \end{tikzpicture}
        =
        t
        \ , \quad
        \begin{tikzpicture}[anchorbase]
            \opendot{0,-0.2};
            \draw (-0.2,0.4) -- (0,0.2) -- (0.2,0.4);
            \draw (0,-0.2) -- (0,0.2);
            \draw (0.2,0.4) -- (0.4,0.6) -- (0.6,0.4);
            \opendot{0.4,1.0};
            \draw (0.4,0.6) -- (0.4,1.0);
            \draw (-0.2,0.4) -- (-0.2,1.2);
            \draw (0.6,-0.4) -- (0.6,0.4);
        \end{tikzpicture}
        =
        \begin{tikzpicture}[anchorbase]
            \draw (0,-0.4) -- (0,1.2);
        \end{tikzpicture}
        =
        \begin{tikzpicture}[anchorbase]
            \opendot{0.4,-0.2};
            \draw (0.2,0.4) -- (0.4,0.2) -- (0.6,0.4);
            \draw (0.4,-0.2) -- (0.4,0.2);
            \draw (-0.2,0.4) -- (0,0.6) -- (0.2,0.4);
            \opendot{0,1.0};
            \draw (0,0.6) -- (0,1.0);
            \draw (0.6,0.4) -- (0.6,1.2);
            \draw (-0.2,-0.4) -- (-0.2,0.4);
        \end{tikzpicture}
        \ 
        \\ \label{GPC2}
        \begin{tikzpicture}[anchorbase]
            \opendot{0,-0.2};
            \draw (-0.2,0.4) -- (0,0.2) -- (0.2,0.4);
            \draw (-0.2,0.4) -- (0,0.6) -- (0.2,0.4);
            \draw (-0.2,1.2) -- (0,1.0) -- (0.2,1.2);
            \draw (0,-0.2) -- (0,0.2);
            \draw (0,0.6) -- (0,1.0);
        \end{tikzpicture}
        =
        \begin{tikzpicture}[anchorbase]
            \draw (-0.2,0.4) -- (0,0.2) -- (0.2,0.4);
            \draw (0,-0.2) -- (0,0.2);
            \opendot{0,-0.2};
        \end{tikzpicture}
        \ , \quad
        \begin{tikzpicture}[anchorbase]
            \opendot{0,-0.2};
            \draw (-0.2,0.4) -- (0,0.2) -- (0.2,0.4);
            \draw (0,-0.2) -- (0,0.2);
            \draw (0.2,0.4) -- (0.4,0.6) -- (0.6,0.4);
            \draw (0.2,1.2) -- (0.4,1.0) -- (0.6,1.2);
            \draw (0.4,0.6) -- (0.4,1.0);
            \draw (-0.2,0.4) -- (-0.2,1.2);
            \draw (0.6,-0.4) -- (0.6,0.4);
        \end{tikzpicture}
        =
        \begin{tikzpicture}[anchorbase]
            \opendot{0.4,-0.2};
            \draw (0.2,0.4) -- (0.4,0.2) -- (0.6,0.4);
            \draw (0.4,-0.2) -- (0.4,0.2);
            \draw (-0.2,0.4) -- (0,0.6) -- (0.2,0.4);
            \draw (-0.2,1.2) -- (0,1.0) -- (0.2,1.2);
            \draw (0,0.6) -- (0,1.0);
            \draw (0.6,0.4) -- (0.6,1.2);
            \draw (-0.2,-0.4) -- (-0.2,0.4);
        \end{tikzpicture}
        \ , 
        \\ \label{GPC3}
        \begin{tikzpicture}[anchorbase]
            \draw (-0.2,-0.4) \braidto (0.2,0) \braidto (-0.2,0.4);
            \draw (0.2,-0.4) \braidto (-0.2,0) \braidto (0.2,0.4);
        \end{tikzpicture}
        =
        \begin{tikzpicture}[anchorbase]
            \draw (-0.2,-0.4) -- (-0.2,0.4);
            \draw (0.2,-0.4) -- (0.2,0.4);
        \end{tikzpicture}
        \ ,\quad
        \begin{tikzpicture}[anchorbase]
            \draw (-0.4,-0.4) -- (0.4,0.4);
            \draw (0.4,-0.4) -- (-0.4,0.4);
            \draw (0,-0.4) \braidto (-0.4,0) \braidto (0,0.4);
        \end{tikzpicture}
        =
        \begin{tikzpicture}[anchorbase]
            \draw (-0.4,-0.4) -- (0.4,0.4);
            \draw (0.4,-0.4) -- (-0.4,0.4);
            \draw (0,-0.4) \braidto (0.4,0) \braidto (0,0.4);
        \end{tikzpicture}
        \ ,
        \\ \label{GPC4}
        \begin{tikzpicture}[anchorbase]
          \draw (-0.2,-0.2) -- (0.2,0.2) --(0,0.4)-- (0,0.8);
          \draw (0.2,-0.2) -- (-0.2,0.2) -- (0,0.4);
          \opendot{0,0.8};
        \end{tikzpicture}
        =
        \begin{tikzpicture}[anchorbase]
            \draw (-0.2,-0.4) -- (0,-0.2) -- (0.2,-0.4);
            \draw (0,-0.2) -- (0,0.2);
            \opendot{0,0.2};
        \end{tikzpicture}
        \ ,\quad
        \begin{tikzpicture}[anchorbase]
          \draw (-0.2,-0.2) -- (0.2,0.2) --(0,0.4)-- (0,0.8);
          \draw (0.2,-0.2) -- (-0.2,0.2) -- (0,0.4);
          \draw (-0.2,1.0) -- (0,0.8) -- (0.2,1.0);
        \end{tikzpicture}
        =
        \begin{tikzpicture}[anchorbase]
            \draw (-0.2,-0.4) -- (0,-0.2) -- (0.2,-0.4);
            \draw (-0.2,0.4) -- (0,0.2) -- (0.2,0.4);
            \draw (0,-0.2) -- (0,0.2);
        \end{tikzpicture}
        \ ,\quad
        \begin{tikzpicture}[anchorbase]
            \draw (-0.2,-0.2) -- (-0.2,0.2);
            \draw (0.2,-0.2) -- (0.6, 0.2);
            \draw (0.6,-0.2) -- (0.2, 0.2);
            \draw (-0.2,0.2) -- (0.2, 0.6);
            \draw ( 0.2, 0.2) -- (-0.2, 0.6);
            \draw (0.6, 0.2) -- (0.6, 0.6);
            \draw (-0.2, 0.6) -- (-0.2,1.4);
            \draw (0.2, 0.6) -- (0.4, 0.8) -- (0.6, 0.6);
            \draw (0.4, 0.8) -- (0.4, 1.2);
            \draw (0.2, 1.4) -- (0.4, 1.2) -- (0.6, 1.4);
        \end{tikzpicture}
        =
        \begin{tikzpicture}[anchorbase]
            \draw (-0.2,-0.2) -- (-0.2,0.2);
            \draw (0.2,-0.2) -- (0.6, 0.2);
            \draw (0.6,-0.2) -- (0.2, 0.2);
            \draw (-0.2,0.2) -- (0.2, 0.6);
            \draw ( 0.2, 0.2) -- (-0.2, 0.6);
            \draw (0.6, 0.2) -- (0.6, 0.6);
            \draw (0.6,-0.2) -- (0.6,-1.0);
            \draw (-0.2, -0.2) -- (0, -0.4) -- (0.2, -0.2);
            \draw (0, -0.4) -- (0,-0.8);
            \draw (-0.2, -1.0) -- (0,-0.8) -- (0.2, -1.0);
        \end{tikzpicture}
        \ ,\quad
        \begin{tikzpicture}[anchorbase]
            \draw (-0.2,-0.2) -- (-0.2,0.2);
            \draw (0.2,-0.2) -- (0.6, 0.2);
            \draw (0.6,-0.2) -- (0.2, 0.2);
            \draw (-0.2,0.2) -- (0.2, 0.6);
            \draw ( 0.2, 0.2) -- (-0.2, 0.6);
            \draw (0.6, 0.2) -- (0.6, 0.6);
            \draw (-0.2, 0.6) -- (-0.2,1.4);
            \draw (0.2, 0.6) -- (0.4, 0.8) -- (0.6, 0.6);
            \draw (0.4, 0.8) -- (0.4, 1.2);
            \opendot{0.4, 1.2};
        \end{tikzpicture}
        =
        \begin{tikzpicture}[anchorbase]
            \draw (-0.2,0) -- (0,0.2) -- (0.2,0);
            \draw (0,0.2) -- (0, 0.6);
            \opendot{0,0.6};
            \draw (0.4,0) -- (0.4,0.6);
        \end{tikzpicture}.
        \ 
    \end{gather}
\end{defi}

%

\begin{thm} \label{Functor G tilde}
    The categories $\underline{\text{Rep}}_{0}(H_{t})$ and $\text{Par}_{t}$ are isomorphic as $\C$-linear symmetric monoidal categories.
\end{thm}
\begin{proof}

Define the functor $\tilde{G}: \text{Par}_{t} \to \underline{\text{Rep}}_{0}(H_{t})$ on objects by $W^{\otimes k} \mapsto [k]$. This makes it bijective on the objects. On the morphisms we define it by sending:

\begin{gather*}
        \begin{tikzpicture}[anchorbase]
            \draw (-0.2,-0.4) -- (0,-0.2) -- (0.2,-0.4);
            \draw (-0.2,0.4) -- (0,0.2) -- (0.2,0.4);
            \draw (0,-0.2) -- (0,0.2);
        \end{tikzpicture} \mapsto
        \FOURLEGS
        \ ,\quad
        \begin{tikzpicture}[anchorbase]
            \draw (-0.2,-0.4) -- (0,-0.2) -- (0.2,-0.4);
            \draw (0,-0.2) -- (0,0.2);
            \opendot{0,0.2};
        \end{tikzpicture} \mapsto
        \CAP
        \ ,\quad
        \crossing \mapsto
        \CROSS
        \ ,\quad
        \begin{tikzpicture}[anchorbase]
            \draw (-0.2,0.4) -- (0,0.2) -- (0.2,0.4);
            \draw (0,-0.2) -- (0,0.2);
            \opendot{0,-0.2};
        \end{tikzpicture} \mapsto
        \CUP
        \ 
    \end{gather*}
and extending this choice $\C$-linearly. The fact that this choice is well-defined can be seen by checking that the relations of Definition \ref{generator relations repH0 definition} are preserved by $\tilde{G}$ and this amounts to the fact that the partition diagrams are the same when the parts are the same. The fullness of $\tilde{G}$ follows directly from Proposition \ref{even partition sn tn form}, so we are left to prove the faithfulness of $\tilde{G}$.

To prove the faithfulness we proceed as follows: We fix for every even partition $p$ some canonical preimage $f_{p}$ under $\tilde{G}$. We then show that every morphism $g$ in $\text{Par}_{t}$ with $\tilde{G}(g)=p$ can be transformed in $f_{p}$ using the relations in \ref{generator relations repH0 definition}.

Let $p$ be an even partition. We fix a normal form $\phi(\sigma) \circ p' \circ \phi(\rho)$ of $p$, see Proposition \ref{partition can be written by permutation partitions} The non-crossing form $p'$ can be written as a tensor product of even blocks $B_{1}\otimes \cdots \otimes B_{r}$. We construct canonical preimages of the permutation partitions and of the even blocks. 
Preimages of the permutation partitions can be build using only the morphisms $\CROSS$ and $\idstrand$ . One could for example compose morphisms which correspond to the transpositions of the symmetric groups. Note that it doesn't matter how we build the preimages of the permutation partitions: as long as we only use the morphisms $\CROSS$ and $\idstrand$ to construct our preimages, the relations in (4.8) ensure us that we always get the same morphism in $\text{Par}_{t}$. We denote the preimage of a permutation partition $\phi(\sigma)$ by $\phi'(\sigma).$

Assume that $B_{i}$ is an even block of size $(k,l)$ with $k>l\geqslant 1$. We define morphisms $s_{n}$ and $t_{n}$ by

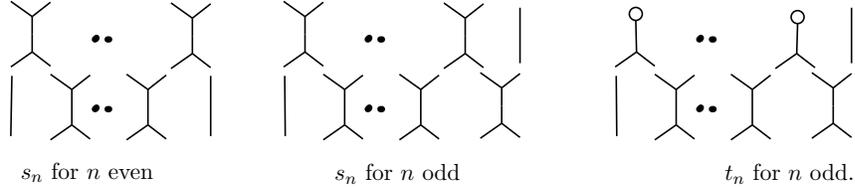
\begin{center}
\scalebox{0.8}{
\tikzset{every picture/.style={line width=0.75pt}} 

\begin{tikzpicture}[x=0.75pt,y=0.75pt,yscale=-1,xscale=1]

\draw    (67.79,293.19) -- (67.79,282.12) ;
\draw    (54.33,302.61) -- (67.79,293.19) ;
\draw    (67.79,293.19) -- (79.32,302.06) ;
\draw    (67.79,270.84) -- (67.79,282.12) ;
\draw    (54.33,261.25) -- (67.79,270.84) ;
\draw    (67.79,270.84) -- (79.32,261.81) ;
\draw    (92.79,339.19) -- (92.79,328.12) ;
\draw    (79.33,348.61) -- (92.79,339.19) ;
\draw    (92.79,339.19) -- (104.32,348.06) ;
\draw    (92.79,316.84) -- (92.79,328.12) ;
\draw    (79.33,307.25) -- (92.79,316.84) ;
\draw    (92.79,316.84) -- (104.32,307.81) ;
\draw    (168.79,294.19) -- (168.79,283.12) ;
\draw    (155.33,303.61) -- (168.79,294.19) ;
\draw    (168.79,294.19) -- (180.32,303.06) ;
\draw    (168.79,271.84) -- (168.79,283.12) ;
\draw    (155.33,262.25) -- (168.79,271.84) ;
\draw    (168.79,271.84) -- (180.32,262.81) ;
\draw    (140.79,339.19) -- (140.79,328.12) ;
\draw    (127.33,348.61) -- (140.79,339.19) ;
\draw    (140.79,339.19) -- (152.32,348.06) ;
\draw    (140.79,316.84) -- (140.79,328.12) ;
\draw    (127.33,307.25) -- (140.79,316.84) ;
\draw    (140.79,316.84) -- (152.32,307.81) ;
\draw    (55,308.25) -- (54.33,346.25) ;
\draw    (180.33,308.25) -- (180.33,346.25) ;
\draw  [line width=3] [line join = round][line cap = round] (107.33,329.25) .. controls (107.8,329.25) and (108.33,328.72) .. (108.33,328.25) ;
\draw  [line width=3] [line join = round][line cap = round] (116.33,329.25) .. controls (116,329.25) and (115.67,329.25) .. (115.33,329.25) ;
\draw  [line width=3] [line join = round][line cap = round] (123.33,329.25) .. controls (123.33,329.25) and (123.33,329.25) .. (123.33,329.25) ;
\draw  [line width=3] [line join = round][line cap = round] (107.33,285.25) .. controls (107.8,285.25) and (108.33,284.72) .. (108.33,284.25) ;
\draw  [line width=3] [line join = round][line cap = round] (116.33,285.25) .. controls (116,285.25) and (115.67,285.25) .. (115.33,285.25) ;
\draw  [line width=3] [line join = round][line cap = round] (123.33,285.25) .. controls (123.33,285.25) and (123.33,285.25) .. (123.33,285.25) ;
\draw    (239.79,293.19) -- (239.79,282.12) ;
\draw    (226.33,302.61) -- (239.79,293.19) ;
\draw    (239.79,293.19) -- (251.32,302.06) ;
\draw    (239.79,270.84) -- (239.79,282.12) ;
\draw    (226.33,261.25) -- (239.79,270.84) ;
\draw    (239.79,270.84) -- (251.32,261.81) ;
\draw    (264.79,339.19) -- (264.79,328.12) ;
\draw    (251.33,348.61) -- (264.79,339.19) ;
\draw    (264.79,339.19) -- (276.32,348.06) ;
\draw    (264.79,316.84) -- (264.79,328.12) ;
\draw    (251.33,307.25) -- (264.79,316.84) ;
\draw    (264.79,316.84) -- (276.32,307.81) ;
\draw    (340.79,294.19) -- (340.79,283.12) ;
\draw    (327.33,303.61) -- (340.79,294.19) ;
\draw    (340.79,294.19) -- (352.32,303.06) ;
\draw    (340.79,271.84) -- (340.79,283.12) ;
\draw    (327.33,262.25) -- (340.79,271.84) ;
\draw    (340.79,271.84) -- (352.32,262.81) ;
\draw    (312.79,339.19) -- (312.79,328.12) ;
\draw    (299.33,348.61) -- (312.79,339.19) ;
\draw    (312.79,339.19) -- (324.32,348.06) ;
\draw    (312.79,316.84) -- (312.79,328.12) ;
\draw    (299.33,307.25) -- (312.79,316.84) ;
\draw    (312.79,316.84) -- (324.32,307.81) ;
\draw    (227,308.25) -- (227.33,346.25) ;
\draw    (375.33,265.25) -- (375.33,300.25) ;
\draw  [line width=3] [line join = round][line cap = round] (279.33,329.25) .. controls (279.8,329.25) and (280.33,328.72) .. (280.33,328.25) ;
\draw  [line width=3] [line join = round][line cap = round] (288.33,329.25) .. controls (288,329.25) and (287.67,329.25) .. (287.33,329.25) ;
\draw  [line width=3] [line join = round][line cap = round] (295.33,329.25) .. controls (295.33,329.25) and (295.33,329.25) .. (295.33,329.25) ;
\draw  [line width=3] [line join = round][line cap = round] (279.33,285.25) .. controls (279.8,285.25) and (280.33,284.72) .. (280.33,284.25) ;
\draw  [line width=3] [line join = round][line cap = round] (288.33,285.25) .. controls (288,285.25) and (287.67,285.25) .. (287.33,285.25) ;
\draw  [line width=3] [line join = round][line cap = round] (295.33,285.25) .. controls (295.33,285.25) and (295.33,285.25) .. (295.33,285.25) ;
\draw    (363.79,338.19) -- (363.79,327.12) ;
\draw    (350.33,347.61) -- (363.79,338.19) ;
\draw    (363.79,338.19) -- (375.32,347.06) ;
\draw    (363.79,315.84) -- (363.79,327.12) ;
\draw    (350.33,306.25) -- (363.79,315.84) ;
\draw    (363.79,315.84) -- (375.32,306.81) ;
\draw    (448.79,293.19) -- (448.33,273.25) ;
\draw    (435.33,302.61) -- (448.79,293.19) ;
\draw    (448.79,293.19) -- (460.32,302.06) ;
\draw    (473.79,339.19) -- (473.79,328.12) ;
\draw    (460.33,348.61) -- (473.79,339.19) ;
\draw    (473.79,339.19) -- (485.32,348.06) ;
\draw    (473.79,316.84) -- (473.79,328.12) ;
\draw    (460.33,307.25) -- (473.79,316.84) ;
\draw    (473.79,316.84) -- (485.32,307.81) ;
\draw    (549.79,294.19) -- (550.33,275.25) ;
\draw    (536.33,303.61) -- (549.79,294.19) ;
\draw    (549.79,294.19) -- (561.32,303.06) ;
\draw    (521.79,339.19) -- (521.79,328.12) ;
\draw    (508.33,348.61) -- (521.79,339.19) ;
\draw    (521.79,339.19) -- (533.32,348.06) ;
\draw    (521.79,316.84) -- (521.79,328.12) ;
\draw    (508.33,307.25) -- (521.79,316.84) ;
\draw    (521.79,316.84) -- (533.32,307.81) ;
\draw    (436,308.25) -- (436.33,346.25) ;
\draw    (584.33,265.25) -- (584.33,300.25) ;
\draw  [line width=3] [line join = round][line cap = round] (488.33,329.25) .. controls (488.8,329.25) and (489.33,328.72) .. (489.33,328.25) ;
\draw  [line width=3] [line join = round][line cap = round] (497.33,329.25) .. controls (497,329.25) and (496.67,329.25) .. (496.33,329.25) ;
\draw  [line width=3] [line join = round][line cap = round] (504.33,329.25) .. controls (504.33,329.25) and (504.33,329.25) .. (504.33,329.25) ;
\draw  [line width=3] [line join = round][line cap = round] (488.33,285.25) .. controls (488.8,285.25) and (489.33,284.72) .. (489.33,284.25) ;
\draw  [line width=3] [line join = round][line cap = round] (497.33,285.25) .. controls (497,285.25) and (496.67,285.25) .. (496.33,285.25) ;
\draw  [line width=3] [line join = round][line cap = round] (504.33,285.25) .. controls (504.33,285.25) and (504.33,285.25) .. (504.33,285.25) ;
\draw    (572.79,338.19) -- (572.79,327.12) ;
\draw    (559.33,347.61) -- (572.79,338.19) ;
\draw    (572.79,338.19) -- (584.32,347.06) ;
\draw    (572.79,315.84) -- (572.79,327.12) ;
\draw    (559.33,306.25) -- (572.79,315.84) ;
\draw    (572.79,315.84) -- (584.32,306.81) ;
\draw   (444.25,269.17) .. controls (444.25,266.91) and (446.08,265.08) .. (448.33,265.08) .. controls (450.59,265.08) and (452.42,266.91) .. (452.42,269.17) .. controls (452.42,271.42) and (450.59,273.25) .. (448.33,273.25) .. controls (446.08,273.25) and (444.25,271.42) .. (444.25,269.17) -- cycle ;
\draw   (546.25,271.17) .. controls (546.25,268.91) and (548.08,267.08) .. (550.33,267.08) .. controls (552.59,267.08) and (554.42,268.91) .. (554.42,271.17) .. controls (554.42,273.42) and (552.59,275.25) .. (550.33,275.25) .. controls (548.08,275.25) and (546.25,273.42) .. (546.25,271.17) -- cycle ;

\draw (59,361.65) node [anchor=north west][inner sep=0.75pt]    {$s_{n} \ \text{for } n\ \text{even} \ $};
\draw (503,362.65) node [anchor=north west][inner sep=0.75pt]    {$t_{n} \ \text{for } n\ \text{odd}. \  $};
\draw (257,362.65) node [anchor=north west][inner sep=0.75pt]    {$s_{n} \ \text{for } n\ \text{odd} \ $};

\end{tikzpicture}
}
\end{center}
 The morphisms $s_{n}$ are of size $(n,n)$ and the morphisms $t_{n}$ are of size $(n,1)$. It follows that
 \begin{align*}
     B_{i}=\tilde{G}(s_{l}\circ (t_{k-l+1} \otimes \underbrace{\idstrand \otimes \ldots \otimes \idstrand\text{ }}_\text{$l-1$ times})).
 \end{align*}  
 
For the cases $0\leqslant k\leqslant l$ and $l=0,k>0$ the preimages are constructed in a similar fashion. We denote the constructed canonical preimage of the even block $B_{i}$ by $B_{i}'$. 

Now let $g \in \text{Hom}_{\text{Par}_{t}}(W^{\otimes k},W^{\otimes l})$ be some morphisms with $\tilde{G}(g)=p$ some even partition. Let $\phi(\sigma) \circ p' \circ \phi(\rho)$ be the fixed normal form for $p$ and $p'=B_{1}\otimes \cdots \otimes B_{r}$, where the $B_{i}$ are even blocks. 
To show that $g$ equals the canonical preimage $\phi'(\sigma) \circ B_{1}'\otimes \cdots \otimes B_{r}' \circ \phi'(\rho)$ of $p$ in $\text{Par}_{t}$ we first note that $g=\phi'(\sigma) \circ \phi'(\sigma^{-1}) \circ g \circ \phi'(\rho^{-1}) \circ \phi'(\rho)$. 
Using the relations (4.8) and (4.9) we can unknot $\phi'(\sigma^{-1}) \circ g \circ \phi'(\rho^{-1})$, which then equals $g_{1} \otimes \cdots \otimes g_{r}$ with $\tilde{G}(g_{i})=B_{i}$. We are done if we show that $g_{i} = B_{i}'$ in $\text{Par}_{t}$. Assume $B_{i}$ is of size $(k,l)$ with $k>l\geqslant 1.$ The arguments for the other cases are similar. We transform the morphism $g_{i}$ into $B_{i}'$ using the relations of definition \ref{generator relations repH0 definition}.

\textit{Step 1:} We draw the diagram in such a way that we have a middle part consisting only of compositions of the morphisms \idstrand, 
 and \CROSS. This is done by adding identities, for example 

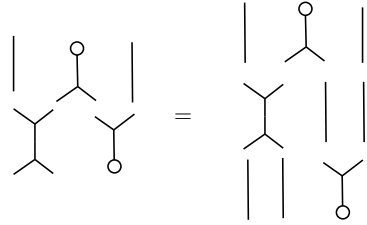
\begin{center}
\scalebox{0.8}{
\tikzset{every picture/.style={line width=0.75pt}} 

%
}
\end{center}

\textit{Step 2:} We assume the middle part does not contain any crosses and change it in the form $s_{n}'$ by using the first two relations of (4.6) and the first one of (4.7). This tells us then that all one-part morphisms of size $(k,l)$ generated only by \idstrand and %
 are the same in $Part_{t}$.

\textit{Step 3:} Assume there are crosses in the middle part. Firstly we note that the crosses don't touch the outer parts by using that the first relation of (4.7) or its reflection, for example 
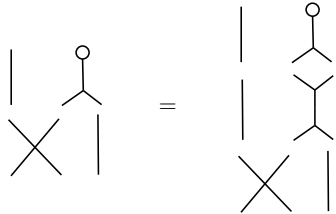
\begin{center}
    
\scalebox{0.8}{
\tikzset{every picture/.style={line width=0.75pt}} 

%
}
\end{center}

Because the morphism exists in one part we can assume by induction on the number of crosses, as we did in the symmetric group case and after applying the second relation of (4.9) to trivially remove crosses, that a cross will appear in one of the following settings:

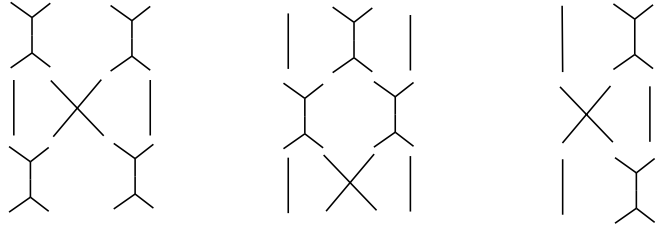
\begin{center}
    
\scalebox{0.8}{
\tikzset{every picture/.style={line width=0.75pt}} 

%
}
\end{center}

The cross can be removed in all these settings by using the second and third relations of (4.9) and the last fact from the previous case. For example:
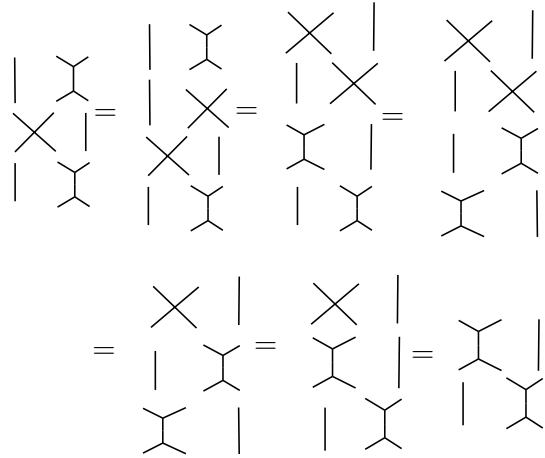
\begin{center}

\scalebox{0.8}{
\tikzset{every picture/.style={line width=0.75pt}} 

%
}

\end{center}

\textit{Step 4:} We move all the %
 to the left using the second relation of (4.7) and its reflection after which we can start removing the superfluous ones using the the third and fourth relation of (4.6). We end without %
 because $k>l$.

 This concludes the proof that $\tilde{G}$ is an isomorphism between $\C$-linear monoidal categories. We note that the swap morphisms in $\text{Par}_{t}$ and $\underline{\text{Rep}}_{0}(H_t)$ can be constructed  by the morphisms $\crossing$ and $%
$ respectively. Because $\tilde{G}$ is functorial and sends $\crossing$ to $%
$, it also preserves the swap morphisms. This shows that $\tilde{G}$ is a $\C$-linear symmetric monoidal functor.

\end{proof}

\begin{rem}
     The isomorphism $\tilde{G}$ induces an isomorphism $\tilde{G}^{Kar}: (\text{Par}_{t})^{Kar} \to\underline{\text{Rep}}(H_t).$
\end{rem}

The theorem implies that $\text{Par}_{t}$ has the structure of a $\C$-linear spherical rigid symmetric monoidal category and that $\underline{\text{Rep}}_{0}(H_t)$ satisfies the following universal property.

\begin{thm} \label{real Universal property fo RepH0, not Karoubian universal prop}
    Let $t\in \C$. Let $\mathcal{C}$ be a $\C$-linear symmetric monoidal category with a $t$-dimensional self-dual rigid object with neutralizer $(A,\alpha,\beta,\delta)$ for an object $A\in Ob(\mathcal{C})$, a neutralizer $\alpha: A\otimes A \to A \otimes A$, an evaluation $\beta: A \otimes A \to \textbf{1}$ and a coevaluation $\delta: \textbf{1} \to A \otimes A$. Let  $\gamma:=s_{A,A}^{\mathcal{C}}: A \otimes A \to A\otimes A$ be the swap morphism. These morphisms satisfy:\\
    1. $A$ is rigid and self-dual
    \begin{align*}
        (id_{A} \otimes \beta) \circ (\delta \otimes id_{A}) = id_{A} = (\beta \otimes id_{A}) \circ (id_{A} \otimes \delta).
    \end{align*}
    2. $A$ has dimension $t$
    \begin{align*}
        \delta \circ \beta =t.
    \end{align*}
     3. $\alpha$ is a neutralizer
    \begin{align*}
        \alpha \circ \alpha &=\alpha,\\
        (id_{A} \otimes \alpha) \circ (\alpha \otimes id_{A}) & = (\alpha \otimes id_{A}) \circ (id_{A} \otimes \alpha).
    \end{align*}
    4. As a rigid object, $A$ is compatible with the neutralizer
    \begin{align*}
        \alpha \circ \delta &= \delta,\\
        \beta \circ \alpha &= \beta,\\
        (id_{A} \otimes \alpha) \circ (\delta \otimes id_{A}) &= (\alpha \otimes id_{A}) \circ (id_{A} \otimes \delta),\\
        (\beta \otimes id_{A}) \circ (id_{A} \otimes \alpha)&= (id_{A} \otimes \beta) \circ (\alpha \otimes id_{A}).
    \end{align*}
 
Then there exists a unique strict $\C$-linear symmetric monoidal functor $\mathcal{G}': \underline{\text{Rep}}_{0}(H_{n}) \to \mathcal{C}$, with $\mathcal{G}'([k])= A^{\otimes k}$ for all $k\in \N$ and  $\mathcal{G}'(\tilde{G}(\begin{tikzpicture}[anchorbase]
            \draw (-0.2,-0.4) -- (0,-0.2) -- (0.2,-0.4);
            \draw (-0.2,0.4) -- (0,0.2) -- (0.2,0.4);
            \draw (0,-0.2) -- (0,0.2);
        \end{tikzpicture}))= \alpha$, $ \mathcal{G}'(\tilde{G}(\begin{tikzpicture}[anchorbase]
            \draw (-0.2,-0.4) -- (0,-0.2) -- (0.2,-0.4);
            \draw (0,-0.2) -- (0,0.2);
            \opendot{0,0.2};
        \end{tikzpicture})) =\beta$, $ \mathcal{G}'(\tilde{G}(\CROSS))= \gamma$ and $\mathcal{G}'(\tilde{G}(\begin{tikzpicture}[anchorbase]
            \draw (-0.2,0.4) -- (0,0.2) -- (0.2,0.4);
            \draw (0,-0.2) -- (0,0.2);
            \opendot{0,-0.2};
        \end{tikzpicture}))= \delta$. Furthermore, if $\mathcal{C}$ is Karoubi, there exists an up to isomorphism unique $\C$-linear symmetric monoidal functor $\mathcal{G}: \underline{\text{Rep}}(H_{t}) \to \mathcal{C}$ with $\mathcal{G}'=\mathcal{G} \circ \iota_{G}$.
\end{thm}

\begin{proof}
    We define the functor $\mathcal{G}'': \text{Par}_{t} \to \mathcal{C}$ on objects by $\mathcal{G}''(W^{\otimes k}):= A^{\otimes k}$ and on morphisms by the $\C$-linear extension of
    $\mathcal{G}''(\begin{tikzpicture}[anchorbase]
            \draw (-0.2,-0.4) -- (0,-0.2) -- (0.2,-0.4);
            \draw (-0.2,0.4) -- (0,0.2) -- (0.2,0.4);
            \draw (0,-0.2) -- (0,0.2);
        \end{tikzpicture}):= \alpha, \mathcal{G}''(\begin{tikzpicture}[anchorbase]
            \draw (-0.2,-0.4) -- (0,-0.2) -- (0.2,-0.4);
            \draw (0,-0.2) -- (0,0.2);
            \opendot{0,0.2};
        \end{tikzpicture}) :=\beta, \mathcal{G}''(\CROSS):= \gamma$ and $\mathcal{G}''(\begin{tikzpicture}[anchorbase]
            \draw (-0.2,0.4) -- (0,0.2) -- (0.2,0.4);
            \draw (0,-0.2) -- (0,0.2);
            \opendot{0,-0.2};
        \end{tikzpicture}):= \delta$. Because $\mathcal{C}$ is a symmetric monoidal category the relations (4.8) and (4.9) hold. A direct verification shows that all the necessary relations are satisfied. This shows that the functor $\mathcal{G}''$ is a well-defined strict $\C$-linear symmetric monoidal functor. We use Theorem \ref{Functor G tilde} to define the strict $\C$-linear symmetric monoidal functor $\mathcal{G}' := \mathcal{G}'' \circ \tilde{G}^{-1}.$ If $\mathcal{C}$ is Karoubi, the universal property of the Karoubian envelope implies the existence of the functor $\mathcal{G}$.
\end{proof}

The following corollary can be proven similarly to the analogous statement for the permutation category.

\begin{cor} \label{Functor categories for Ht}
    Let $t\in \C$ and $\mathcal{C}$ be a $\C$-linear symmetric monoidal Karoubi category. Then there is an equivalence
    \begin{align*}
        \text{Fun}^{\otimes,Symm}_{\C}(\underline{\text{Rep}}(H_{t}),\mathcal{C}) \simeq \text{Rig}_{\C}^{sd,neutr}(\mathcal{C},t)
    \end{align*}
    between the category $\text{Fun}^{\otimes,Symm}_{\C}(\underline{\text{Rep}}(H_{t}),\mathcal{C})$ of $\C$-linear symmetric monoidal functors from $\underline{\text{Rep}}(H_{t})$ to $\mathcal{C}$ with natural isomorphisms between them, and the subcategory $\text{Rig}_{\C}^{sd,neutr}(\mathcal{C},t)$ of $\mathcal{C}$ consisting of $t$-dimensional rigid self-dual objects with neutralizer and their isomorphisms.
\end{cor}



\section{An equivalence of symmetric monoidal categories}

In this section we want to relate the different interpolation categories. We used the sets of partition $P_{even}\subset P \subset P_{\Z_{2}}$ to define the morphism spaces. This leads to functors $\Phi: \underline{\text{Rep}}(H_{t}) \to \underline{\text{Rep}}(S_{t})$ and $\Psi: \underline{\text{Rep}}(S_{t})\to \text{Par}(\Z_{2},t)$. Their composition $\Psi \circ \Phi: \underline{\text{Rep}}(S_{t}) \to \text{Par}(\Z_{2},t)$ however will not be so interesting from a representation theoretical point of view (these are not equivalences, nor are they compatible with semisimplification). For later use in Section \ref{sec:ss} we first recall the relationship of the interpolation categories to $Rep(H_n)$.

\subsection{Relation to the representations of $H_n$}

In the $S_t$-case we have the following description of the relation between morphisms in the Deligne category and the usual morphisms between tensor powers of the permutation representation of $S_n$ \cite[Theorem 2.6]{Comes_2011}.

\begin{defi}
Let $e_{1},\ldots, e_{n}$ be the canonical basis for $u'=\C^{n}$. For all $k,l\in \N$ we define a $\C$-linear map by its image on the basis elements:

\begin{align*}
    T: \C P(k,l) &\to \text{Hom}_{S_{n}}((\mathbb{C}^{n})^{\otimes k},(\mathbb{C}^{n})^{\otimes l})\\
    p  &\mapsto  T_{p}
\end{align*}
where 
\begin{align*}
    T_{p}: (\mathbb{C}^{n})^{\otimes k} &\to (\mathbb{C}^{n})^{\otimes l}\\
       e_{i_{1}}\otimes\ldots \otimes e_{i_{k}} &\mapsto \sum_{1\leqslant j_{1}\ldots j_{l}\leqslant n} \delta_{p}(i,j) (e_{j_{1}}\otimes\ldots \otimes e_{j_{l}}).
\end{align*}
for all $1\leqslant i_{1},\ldots, i_{k} \leqslant n$. We use the abbreviation $i=(i_{1},\ldots, i_{k})$ and  $j=(j_{1},\ldots, j_{l})$. We can label the partition diagram of $p$ by $i$ and $j$ in an obvious way: $i$ labels th lower row from left to right and $j$ the upper row from left to right. By definition $\delta_{p}(i,j)$ equals 1 if and only if all vertices in the same block of the partition are labeled by the same number. Otherwise it equals $0$. 
\end{defi}

\begin{prop} \label{Functor F is full}
    The linear map $T: \C P(k,l) \to \text{Hom}_{S_{n}}((\mathbb{C}^{n})^{\otimes k},(\mathbb{C}^{n})^{\otimes l})$ is surjective and $T$ is an isomorphism in the cases that $k+l\leqslant n.$
\end{prop}

It is straightforward to see that restricting this map to even partitions induces a map $T: \C P_{even}(k,l) \to \text{Hom}_{H_{n}}((\mathbb{C}^{n})^{\otimes k},(\mathbb{C}^{n})^{\otimes l})$ which is an isomorphism in case $k + l \leqslant n$ (see also \cite{BS09}).

The morphism spaces between the tensor products of the permutation representation $V=\C^{2n}$ of $H_{n}$ can be described using the set of $\Z_{2}$-coloured partitions $P_{\Z_{2}}$. We alter the notation that was given in \cite[Proposition 5.2]{Nyobe_Likeng_2021} for consistency. 

\begin{defi}
Recall the set of basis elements $\{e_{j}^{i}|1\leqslant i\leqslant n, j\in\{-1,1\}\}$ for $\C^{2n}$. For all $k,l\in \N$ we define a $\C$-linear map by its image on the basis elements:

\begin{align*}
    T: \C P_{\Z_{2}}(k,l) &\to \text{Hom}_{H_{n}}((\mathbb{C}^{2n})^{\otimes k},(\mathbb{C}^{2n})^{\otimes l})\\
    (p,z)  &\mapsto  T_{(p,z)}
\end{align*}
where 
\begin{align*}
    T_{(p,z)}(e_{b_{1}}^{i_{1}}\otimes\ldots \otimes e_{b_{k}}^{i_{k}}) = \sum_{\substack{1 \leqslant j_{1}\ldots j_{l}\leqslant n\\c_{1}\ldots c_{l}\in \Z_{2}}} \delta_{(p,z)}((i,b),(j,c)) (e_{c_{1}}^{j_{1}}\otimes\ldots \otimes e_{c_{l}}^{j_{l}})
\end{align*}
for all $(i,b) \in \{1,\ldots,n\}^{k} \times (\Z_{2})^{k}$. We use the notation $(i,b)=((i_{1},\ldots, i_{k})(b_{1}, \ldots , b_{k}))$ and  $(j,c)=((j_{1},\ldots, j_{l}),(c_{1}, \ldots , c_{l}))$. We can label the partition $p$ by $(i,j)$ and $(b,c)$ in an obvious way: $i$ and $b$ label the lower row of $p$ from left to right and $j$ and $c$ label the upper row from left to right. Then $\delta_{(p,z)}((i,b),(j,c))$ equals 1 if and only if for vertices in the same part of $p$ the following two condition hold:  the corresponding $(i,j)$-labels are the same and the corresponding $(b,c)$-labels multiplied with the corresponding labels of $z$ are the same. Otherwise it equals $0$. 
\end{defi}

The following proposition is a special case of \cite[Theorem 5.4]{Nyobe_Likeng_2021}.

\begin{prop} \label{Functor H is full} 
    The linear map $T: \C P_{\Z_{2}}(k,l) \to \text{Hom}_{H_{n}}((\mathbb{C}^{2n})^{\otimes k},(\mathbb{C}^{2n})^{\otimes l})$ is surjective and is an isomorphism in case $k+l\leqslant n.$
\end{prop}

These descriptions allow us to define interpolation functors from the interpolation categories to the categories of representations of $H_n$.

\begin{defi} \label{definition G}
     We define a strict $\C$-linear tensor functor between $\C$-linear spherical rigid symmetric monoidal categories
     \begin{align*}
         G': \underline{\text{Rep}}_{0}(H_{n}) \to \text{Rep}(H_{n})
     \end{align*} on the objects by $[k] \to u^{\otimes k}$ and on the morphisms by $p\to T_{p}$ for some partition $p\in P_{even}(k,l)$, after which one linearly extends this rule. The universal property of the Karoubian envelope gives us a $\C$-linear tensor functor 
    \begin{align*}
        G: \underline{\text{Rep}}(H_{n}) \to \text{Rep}(H_{n}),
    \end{align*} which satisfies $G'=G\circ \iota_{G}$ and is unique up to isomorphism. 
\end{defi}

The following description of the interpolation functor differs from \cite[Theorem 5.1]{Nyobe_Likeng_2021} since we do not use the presentation of the interpolation category via generators and relations.

\begin{defi} \label{definition H}
    We define a strict $\C$-linear tensor functor between $\C$-linear spherical rigid symmetric monoidal categories
    \begin{align*}
        H':\text{Par}(\mathbb{Z}_{2},2n) \to \text{Rep}(H_{n})
    \end{align*} on the objects by $[\tilde{k}] \to V^{\otimes k}$ and on the morphisms by a linear extension of the rule $(p,z)\to T_{(p,z)}$. The universal property of the Karoubian envelope gives us a $\C$-linear tensor functor 
    \begin{align*}
         H:\text{Par}(\mathbb{Z}_{2},2n)^{Kar} \to \text{Rep}(H_{n}),
    \end{align*} which satisfies $H'=H\circ \iota_{H}$ and is unique up to isomorphism. We will see that some choices of $H$ will be more convenient in particular situations, for example in the proof of Corollary \ref{commutative square}.
\end{defi}

Both $G$ and $H$ are full and essentially surjective (see \cite{Nyobe_Likeng_2021} \cite{BS09} \cite{Flake_2021}).

\subsubsection{Alternative definition of interpolation functor $H': \text{Par}(\mathbb{Z}_{2},2n) \to \text{Rep}(H_{n})$}

Using the description above, the semisimplification functor is defined by the strict $\C$-linear monoidal functor $H'': \text{Par}(\Z_{2})/\sim _{n} \to \text{Rep}(H_{n})$ on objects by sending $W$ to $V\in \text{Rep}(H_{n})$ \cite[Theorem 5.1]{Nyobe_Likeng_2021}. The images of the generating objects are defined by:
   \begin{align*}
        H''(\merge) &\colon V \otimes V \to V, &
        e^i_g \otimes e^j_h &\mapsto \delta_{g,h} \delta_{i,j}  e^i_g,
        \\
        H''(\spliter) &\colon V \to V \otimes V, &
         e^i_g &\mapsto  e^i_g \otimes  e^i_g,
        \\
        H''(\crossing) &\colon V \otimes V \to V, &
        v \otimes w &\mapsto w \otimes v,
        \\
        H''(\bottompin) &\colon \C \to V, &
        1 &\mapsto \textstyle \sum_{g \in G} \sum_{i=1}^n  e^i_g,
        \\
        H''(\toppin) &\colon V \to \C, &
         e^i_g &\mapsto 1,
        \\
        H''(\tokstrand) & \colon V \to V, &
         e^i_h &\mapsto  e^i_{gh}
    \end{align*}
 for $g,h \in \Z_{2}$.

\subsubsection{Alternative definition of interpolation functor $G': \underline{\text{Rep}}_{0}(H_{n}) \to \text{Rep}(H_{n})$}

The semisimplification functor to $Rep(H_n)$ can be described by the functor $G'': \text{Par}_{n} \to \text{Rep}(H_{n})$ on objects by sending $W$ to $u\in \text{Rep}(H_{n})$. The images of the generating objects are defined by:
   \begin{align*}
        G''(\begin{tikzpicture}[anchorbase]
            \draw (-0.2,-0.4) -- (0,-0.2) -- (0.2,-0.4);
            \draw (-0.2,0.4) -- (0,0.2) -- (0.2,0.4);
            \draw (0,-0.2) -- (0,0.2);
        \end{tikzpicture}) &\colon u \otimes u \to u \otimes u, &
        e_i \otimes e_j &\mapsto \delta_{i,j}  e_i \otimes e_{i},
        \\
        G''(\begin{tikzpicture}[anchorbase]
            \draw (-0.2,-0.4) -- (0,-0.2) -- (0.2,-0.4);
            \draw (0,-0.2) -- (0,0.2);
            \opendot{0,0.2};
        \end{tikzpicture}) &\colon u\otimes u \to \C, &
         e_i \otimes e_{j} &\mapsto \delta_{i,j},
         \\
        G''(\crossing) &\colon u \otimes u \to u, &
        v \otimes w &\mapsto w \otimes v,
        \\
        G''(\begin{tikzpicture}[anchorbase]
            \draw (-0.2,0.4) -- (0,0.2) -- (0.2,0.4);
            \draw (0,-0.2) -- (0,0.2);
            \opendot{0,-0.2};
        \end{tikzpicture}) &\colon \C \to u \otimes u, &
        1 &\mapsto \textstyle \sum_{i=1}^n  e_i \otimes e_{i}.
    \end{align*}
  The fact that this is well-defined is proven similarly as in \cite[Theorem 5.1]{Nyobe_Likeng_2021}.

\subsection{Definition of $\Omega$}

We saw in Remark \ref{u subrepresentatino of V} that the reflection permutation of the hyperoctahedral group $H_{n}$ is isomorphic to a subrepresentation of the permutation representation of $H_{n}$. We want to describe the subrepresentation $\tilde{u} = \bigoplus_{i=1}^{n}\mathbb{C}(e^{i}_{1}-e^{i}_{-1})\subset V$ as the image of some idempotent $e:V \to V$. 

\begin{defi}
Let $\alpha:\tilde{u} \to V$ be the inclusion and define $\beta: V \to \tilde{u}$ by
\begin{align*}
    &e_{1}^{i}\mapsto \frac{e_{1}^{i}-e_{-1}^{i}}{2} \text{ for all } i\in \{1,\ldots, n\}\\
    &e_{-1}^{i} \mapsto \frac{-e_{1}^{i}+e_{-1}^{i}}{2} \text{ for all } i\in \{1,\ldots, n\}.
\end{align*}
Then the morphism $e:=\alpha \circ \beta:V \to V$ is an idempotent with a splitting given by $\alpha,\beta$ and $\tilde{u}$.
    
\end{defi}

\begin{defi} Define the idempotent
    \begin{align*}
    e':=\frac{\begin{tikzpicture}[centerzero]
            \pd{0,-0.2} node[anchor=east] {\dotlabel{1}};
            \pd{0,0.2};
            \draw (0,-0.2) -- (0,0.2);
        \end{tikzpicture}-\begin{tikzpicture}[centerzero]
            \pd{0,-0.2} node[anchor=east] {\dotlabel{-1}};
            \pd{0,0.2};
            \draw (0,-0.2) -- (0,0.2);
        \end{tikzpicture}}{2}: [\tilde{1}] \to [\tilde{1}]            
    \end{align*}
    in $\text{Par}(\Z_{2},2t)$ for arbitrary $t \in \C$. 
\end{defi}

\begin{lemma} Let $H:\text{Par}(\Z_{2},2n)^{Kar}\to \text{Rep}(H_{n})$ denote the semisimplification functor. Then
\begin{align*}
    &H([\tilde{1}])= V \text{ and }\\
    &H(([\tilde{1}],e'))= \tilde{u}\cong u.
\end{align*}
\end{lemma}

\begin{proof} By definition we know that $H'([\tilde{1}])=V$ for $H':\text{Par}(\Z_{2},2n)\to \text{Rep}(H_{n})$, the restriction of $H$ to $\text{Par}(\Z_{2},2n)\subset \text{Par}(\Z_{2},2n)^{Kar}$. It follows directly from the definitions of $T$ that
     \begin{align*}
            H'(e')(e^{i}_{1})&= \frac{T(\begin{tikzpicture}[centerzero]
            \pd{0,-0.2} node[anchor=east] {\dotlabel{1}};
            \pd{0,0.2};
            \draw (0,-0.2) -- (0,0.2);
        \end{tikzpicture})(e^{i}_{1}) -T(\begin{tikzpicture}[centerzero]
            \pd{0,-0.2} node[anchor=east] {\dotlabel{-1}};
            \pd{0,0.2};
            \draw (0,-0.2) -- (0,0.2);
        \end{tikzpicture})(e^{i}_{1})}{2} = \frac{e_{1}^{i}-e_{-1}^{i}}{2} \text{ and }\\
        H'(e')(e^{i}_{-1})&= \frac{T(\begin{tikzpicture}[centerzero]
            \pd{0,-0.2} node[anchor=east] {\dotlabel{1}};
            \pd{0,0.2};
            \draw (0,-0.2) -- (0,0.2);
        \end{tikzpicture})e^{i}_{-1} -T(\begin{tikzpicture}[centerzero]
            \pd{0,-0.2} node[anchor=east] {\dotlabel{-1}};
            \pd{0,0.2};
            \draw (0,-0.2) -- (0,0.2);
        \end{tikzpicture})e^{i}_{-1}}{2} = \frac{-e_{1}^{i}+e_{-1}^{i}}{2},
    \end{align*}
    for all $i\in \{1,\ldots,n\}.$ This implies that $H'(e')=e.$

    Let $\alpha': ([\tilde{1}],e') \to ([\tilde{1}],id_{[\tilde{1}]})$ and $\beta': ([\tilde{1}],id_{[\tilde{1}]}) \to ([\tilde{1}],e')$ be the splitting of the idempotent $e':([\tilde{1}],id_{[\tilde{1}]})\to ([\tilde{1}],id_{[\tilde{1}]})$ in $\text{Par}(\Z_{2},2n)^{Kar}$. The functor $H:\text{Par}(\Z_{2},2n)^{Kar}\to \text{Rep}(H_{n})$ is defined by sending a splitting in $\text{Par}(\Z_{2},2n)^{Kar}$ to some chosen splitting in $\text{Rep}(H_{n})$. So we can make the following choices for $H$: 
    \begin{align*}
        H(([\tilde{1}],e'))=H(im(e'))&:= im(e)=\tilde{u},\\
        H(\alpha')&:=\alpha \text{ and }\\
        H(\beta')&:=\beta.
    \end{align*}
\end{proof}

\begin{lemma} \label{inspiration omega self dual and dimension}
    Let $t\in \C$. The objects $([\tilde{1}],e')$ and $([\tilde{1}],1-e')$ in $\text{Par}(\Z_{2},2t)^{Kar}$ are self-dual and have categorical dimension $t$.
\end{lemma}

\begin{rem} \label{changing signs}
    The object in $(\text{Par}(\Z_{2})/\sim_{2t})^{Kar}$ corresponding to  $([\tilde{1}],e')$ of $\text{Par}(\Z_{2},2t)^{Kar}$ is $(W,e')$ where we will set, by abuse of notation,
\begin{align*}
    e':= \frac{\tokstrand[1]-\tokstrand[-1]}{2}.
\end{align*}
Note that $e' \circ \tokstrand[-1]=-e'=\tokstrand[-1] \circ e'$.
\end{rem}

\begin{thm} \label{Omega0 exists!}
    For all $t\in \C$ there is a well-defined strict $\C$-linear tensor functor
    \begin{align*}
        \Omega_{0}: \text{Par}_{t}\cong \underline{\text{Rep}}_{0}(H_{t}) \to (\text{Par}(\Z_{2},2t))^{Kar} \cong (\text{Par}(\Z_{2})/\sim_{2t})^{Kar}
    \end{align*} which is defined on objects by 
    \begin{align*}
        \Omega_{0}([k]):= ([\tilde{k}],(e')^{\otimes k})
    \end{align*} for all $k\in \N$. On morphisms it is the $\C$-linear extension of the following rule. Let $f:W^{\otimes k}\to W^{\otimes l}$ be a morphism in $\text{Par}_{t}$ which consists of $s$ blocks $B_{1},\ldots,\B_{s}$ of size $m_{1},\ldots,m_{s}$ respectively. Then we set 
    \begin{align*}
        \Omega_{0}(f):= 2^{\frac{(\sum_{i=1}^{s}m_{i})-2s}{2}} (e')^{\otimes l} \circ f \circ (e')^{\otimes k},
    \end{align*}
    where we identify morphisms in $\text{Par}_{t}$ with the corresponding morphisms in $\text{Par}(\Z_{2})/\sim_{2t}$ by using the following associations: 
    \begin{align*} 
    \CROSS &\to  \crossing, & \
    \begin{tikzpicture}[anchorbase]
            \draw (-0.2,-0.4) -- (0,-0.2) -- (0.2,-0.4);
            \draw (-0.2,0.4) -- (0,0.2) -- (0.2,0.4);
            \draw (0,-0.2) -- (0,0.2);
        \end{tikzpicture} &\to  \spliter \circ \merge,\\ 
    \idstrand &\to  \idstrand \ , & \ 
    \begin{tikzpicture}[anchorbase]
            \draw (-0.2,-0.4) -- (0,-0.2) -- (0.2,-0.4);
            \draw (0,-0.2) -- (0,0.2);
            \opendot{0,0.2};
        \end{tikzpicture} &\to  \toppin \circ \merge \text{  and } \\
        \begin{tikzpicture}[anchorbase]
            \draw (-0.2,0.4) -- (0,0.2) -- (0.2,0.4);
            \draw (0,-0.2) -- (0,0.2);
            \opendot{0,-0.2};
        \end{tikzpicture}&\to  \spliter \circ \bottompin.
        \end{align*}
    
\end{thm}
\begin{proof}

By definition of the monoidal structure of the Karoubian envelope of a category we see that  $([\tilde{1}],e')^{\otimes k}=(([\tilde{k}],(e')^{\otimes k}).$ So it is clear that $\Omega_{0}$ respects the tensor product for objects. By definition of the morphisms spaces in the Karoubian envelope, we have \begin{align*}
    \Hom_{\text{Par}(\Z_{2},2t)^{Kar}}(([\tilde{1}],e')^{\otimes k},([\tilde{1}],e')^{\otimes l})= (e')^{\otimes l} \circ \Hom_{\text{Par}(\Z_{2},2t)}([\tilde{k}],[\tilde{l}])\circ (e')^{\otimes k}.
\end{align*} For morphisms $f:W^{\otimes k_{1}}\to W^{\otimes l_{1}}$ with $s_{1}$ blocks $B_{1,1},\ldots,\B_{s,1}$ of size $m_{1},\ldots,m_{s_{1}}$ respectively and $g:W^{\otimes k_{2}}\to W^{\otimes l_{2}}$ with $s_{2}$ blocks $B_{1,2},\ldots,\B_{s_{2},2}$ of size $m_{s_{1}+1},\ldots,m_{s_{1}+s_{2}}$ respectively we see that 
\begin{align*}
    \Omega_{0}(f\otimes g)&= 2^{\frac{(\sum_{i=1}^{s_{1}+s_{2}}m_{i})-2(s_{1}+s_{2})}{2}} (e')^{\otimes l_{1}+l_{2}} \circ f \otimes g \circ (e')^{\otimes k_{1}+k_{2}}\\
    &= 2^{\frac{(\sum_{i=1}^{s_{1}}m_{i})-2(s_{1})}{2}}2^{\frac{(\sum_{i=s_{1}+1}^{s_{1}+s_{2}}m_{i})-2(s_{2})}{2}} ((e')^{\otimes l_{1}} \circ f  \circ (e')^{\otimes k_{1}}) \otimes ((e')^{\otimes l_{2}} \circ g  \circ (e')^{\otimes k_{2}})\\
    &=(2^{\frac{(\sum_{i=1}^{s_{1}}m_{i})-2(s_{1})}{2}} (e')^{\otimes l_{1}} \circ f  \circ (e')^{\otimes k_{1}} )\otimes (2^{\frac{(\sum_{i=s_{1}+1}^{s_{1}+s_{2}}m_{i})-2(s_{2})}{2}} (e')^{\otimes l_{2}} \circ g  \circ (e')^{\otimes k_{2}})\\
    &=\Omega_{0}(f) \otimes \Omega_{0}(g). 
\end{align*}
This shows that $\Omega_{0}$ respects the tensor product of morphisms. We will identify the functor $\Omega_{0}:\underline{\text{Rep}}_{0}(H_{t}) \to (\text{Par}(\Z_{2},2t))^{Kar}$ with the composition \begin{align*}
    \tilde{H}^{-1} \circ \Omega_{0} \circ \tilde{G}: \text{Par}_{t} \to (\text{Par}(\Z_{2})/\sim_{2t})^{Kar},
\end{align*} so we can talk easier about the image of the generators and relations of Definition \ref{generator relations repH0 definition} under the functor $\Omega_{0}$. The definition implies

\begin{align} 
    \Omega_{0}(\CROSS)&= 2^{0}(e'\otimes e') \circ \crossing \circ (e'\otimes e')=(e'\otimes e') \circ \crossing \circ (e'\otimes e'),\\
    \Omega_{0}(\begin{tikzpicture}[anchorbase]
            \draw (-0.2,-0.4) -- (0,-0.2) -- (0.2,-0.4);
            \draw (-0.2,0.4) -- (0,0.2) -- (0.2,0.4);
            \draw (0,-0.2) -- (0,0.2);
        \end{tikzpicture}) &= 2^{\frac{4-2}{2}} (e'\otimes e') \circ \spliter \circ \merge \circ (e'\otimes e') = 2 (e'\otimes e') \circ \begin{tikzpicture}[anchorbase]
            \draw (-0.2,-0.4) -- (0,-0.2) -- (0.2,-0.4);
            \draw (-0.2,0.4) -- (0,0.2) -- (0.2,0.4);
            \draw (0,-0.2) -- (0,0.2);
        \end{tikzpicture} \circ (e'\otimes e'), \\ 
    \Omega_{0}(\text{ } \idstrand \text{ })&= 2^{0} e' \circ \idstrand \circ e'=  e' \circ \idstrand \circ e',\\
    \Omega_{0}(\begin{tikzpicture}[anchorbase]
            \draw (-0.2,-0.4) -- (0,-0.2) -- (0.2,-0.4);
            \draw (0,-0.2) -- (0,0.2);
            \opendot{0,0.2};
        \end{tikzpicture})&= 2^{0} \toppin \circ \merge \circ (e' \otimes e')=\begin{tikzpicture}[anchorbase]
            \draw (-0.2,-0.4) -- (0,-0.2) -- (0.2,-0.4);
            \draw (0,-0.2) -- (0,0.2);
            \opendot{0,0.2};
        \end{tikzpicture} \circ (e' \otimes e') \text{ and }\\
        \Omega_{0}(\begin{tikzpicture}[anchorbase]
            \draw (-0.2,0.4) -- (0,0.2) -- (0.2,0.4);
            \draw (0,-0.2) -- (0,0.2);
            \opendot{0,-0.2};
        \end{tikzpicture} )&= 2^{0} \ (e' \otimes e') \circ \spliter \circ \bottompin= (e' \otimes e') \circ \begin{tikzpicture}[anchorbase]
            \draw (-0.2,0.4) -- (0,0.2) -- (0.2,0.4);
            \draw (0,-0.2) -- (0,0.2);
            \opendot{0,-0.2};
        \end{tikzpicture} .
\end{align}

Now we want to show that $\Omega_{0}$ is functorial. If we prove that the object $(W,e')$ and the morphisms (5.5), (5.6),(5.7),(5.8) and (5.9) satisfy the relations corresponding to (4.6),(4.7),(4.8) and (4.9), then the conditions for Theorem \ref{real Universal property fo RepH0, not Karoubian universal prop} are satisfied. In other words, we want to show that morphisms (5.5), (5.6),(5.7),(5.8) and (5.9), if considered as the swap morphism, neutralizer, identity, evaluation and coevaluation respectively, satisfy the necessary relations of definition \ref{generator relations repH0 definition}.

        The equation in remark \ref{changing signs}, implies the relation

\begin{center}

\tikzset{every picture/.style={line width=0.75pt}} 



\end{center}

The analogue of the first relation of (4.6), namely %
, follows now from

\begin{align*}
    \Omega_{0}(%
).
\end{align*}        
        The analogue of the relation %
        and its reflection can be proven similarly. Now we prove the analogue to the second relation of (4.6). The fact that $\Omega_{0}$ respects tensor products, the relation 

        \begin{center}

\tikzset{every picture/.style={line width=0.75pt}} 

%
    
\end{center}
and the fact that the relation  %
).
\end{align*}

The analogues of the relations
        %
 and their reflections can be proven similarly.

The analogue to the third relation of (4.6), 
\begin{align*}
    \toppin \circ \merge \circ (e' \otimes e') \circ (e'\otimes e') \circ \spliter \circ \bottompin =t,
\end{align*} indicates that the dimension of $(W,e')$ equals $t$ and this follows directly from Lemma \ref{inspiration omega self dual and dimension}. 

The relations in (4.8) and (4.9) follow by some easy calculations from the relation

\begin{center}
\tikzset{every picture/.style={line width=0.75pt}} 

%
\end{center}

and the fact that the corresponding relations hold in $\text{Par}(\Z_{2})/\sim_{2t}$. 

By Theorem \ref{real Universal property fo RepH0, not Karoubian universal prop} we see that $\Omega_{0}:\text{Par}_{t} \to (\text{Par}(\Z_{2})/\sim_{2t})^{Kar}$ is a well-defined $\C$-linear symmetric monoidal functor between $\C$-linear spherical rigid symmetric monoidal categories. The functor is clearly strict and preserves the different monoidal structures. Therefore it is a strict $\C$-linear tensor functor.
\end{proof}

\begin{cor}
    For all $t\in \C$ there is a well-defined $\C$-linear tensor functor
    \begin{align*}
        \Omega:  \underline{\text{Rep}}(H_{t}) \to (\text{Par}(\Z_{2},2t))^{Kar} 
    \end{align*} which restricts to $\Omega \circ \iota_{G}=\Omega_{0}$.
\end{cor}

\subsection{Fully faithfullness}

\begin{thm} \label{Properties of the functor Omega0} 
  The functor  $\Omega_{0}: \underline{\text{Rep}}_{0}(H_{t}) \to (\text{Par}(\Z_{2},2t))^{Kar}$ is a full embedding. 
\end{thm}
\begin{proof}
By definition $k\neq l$ implies that $\Omega_{0}([k])=([\tilde{k}],(e')^{\otimes k})\neq ([\tilde{l}],(e')^{\otimes l})=\Omega_{0}([k])$, showing that functor $\Omega_{0}$ is injective on objects.\\

Next we want to show that $\Omega_{0}$ is full. The functor $\Omega_{0}: \underline{\text{Rep}}_{0}(H_{t}) \to (\text{Par}(\Z_{2},2t))^{Kar}$ is given on morphisms by \begin{align*}
     \Omega_{0}(f)= 2^{\frac{(\sum_{i=1}^{s}m_{i})-2s}{2}} (e')^{\otimes l} \circ f \circ (e')^{\otimes k},
\end{align*}
for some partition $f\in P_{even}(k,l)$ with $s$ blocks of size $m_{1},\ldots,m_{s}$ and then extended $\C$-linearly. Remember that $f$ can be considered a $\Z_{2}$-coloured partition where every vertex is labeled as $1$ and that the tensor product of partitions is the horizontal concatenation of the partitions. 

We consider the morphism
\begin{align*}
    f=(e')^{\otimes l}\circ g \circ (e')^{\otimes k}\in \text{Hom}_{\text{Par}(\Z_{2},2t)^{Kar}}(([\tilde{k}],(e')^{\otimes k}), ([\tilde{l}],(e')^{\otimes l})) \\= (e')^{\otimes l}\circ \text{Hom}_{\text{Par}(\Z_{2},2t)}([\tilde{k}], [\tilde{l}]) \circ (e')^{\otimes k}.
\end{align*}
such that $g \in P_{\Z_{2}}(k,l)$ with labeling $(z_{1},\ldots, z_{k},z_{1'},\ldots,z_{l'}) \in \Z_{2}^{k+l}$.  Then 
 \begin{equation*} 
    g=
    \begin{tikzpicture}[centerzero]
        \pd{-0.4,-0.2};
        \pd{0.4,-0.2};
        \pd{0.8,-0.2};
        \pd{-0.4,0.2} node[anchor=south] {\dotlabel{z_{1'}}};
        \pd{0.4,0.2} node[anchor=south] {\dotlabel{z_{l-1'}}};
        \pd{0.8,0.2} node[anchor=south] {\dotlabel{z_{l'}}};
        \draw (-0.4,-0.2) -- (-0.4,0.2);
        \draw (0.4,-0.2) -- (0.4,0.2);
        \draw (0.8,-0.2) -- (0.8,0.2);
        \node at (0.04,0) {$\cdots$};
    \end{tikzpicture}
    \circ
    g''
    \circ
    \begin{tikzpicture}[centerzero]
        \pd{-0.4,-0.2} node[anchor=north] {\dotlabel{z_1}};
        \pd{0.4,-0.2} node[anchor=north] {\dotlabel{z_{k-1}}};
        \pd{0.8,-0.2} node[anchor=north] {\dotlabel{z_k}};
        \pd{-0.4,0.2};
        \pd{0.4,0.2};
        \pd{0.8,0.2};
        \draw (-0.4,-0.2) -- (-0.4,0.2);
        \draw (0.4,-0.2) -- (0.4,0.2);
        \draw (0.8,-0.2) -- (0.8,0.2);
        \node at (0.04,0) {$\cdots$};
  \end{tikzpicture}  \end{equation*} 
   and the relation 
\begin{align*}
    e' \circ \begin{tikzpicture}[centerzero]
            \pd{0,-0.2} node[anchor=east] {\dotlabel{-1}};
            \pd{0,0.2};
            \draw (0,-0.2) -- (0,0.2);
        \end{tikzpicture} = -e'= \begin{tikzpicture}[centerzero]
            \pd{0,-0.2} node[anchor=east] {\dotlabel{-1}};
            \pd{0,0.2};
            \draw (0,-0.2) -- (0,0.2);
        \end{tikzpicture} \circ e'
\end{align*} 
    implies
\begin{align*}
    (e')^{\otimes l}\circ g \circ (e')^{\otimes k}= (\prod_{\substack{j \in \{1,\ldots,k,1',\ldots,l'\} }}(z_{j}))  (e')^{\otimes l}\circ g'' \circ (e')^{\otimes k}
\end{align*}
where $g''$ is the underlying partition of the $\Z_{2}$-coloured partition $g$. So we can reduce the problem to the case where $g$ lies in $P(k,l)$. 
We first assume that $g$ is a non-even partition in $P(k,l)$. By Proposition \ref{partition can be written by permutation partitions} we can write $g=\phi(\sigma) \circ g' \circ \phi(\rho)$, where $g'$ is some non-crossing form of $g$. Then $g'$ is the horizontal concatenation of its blocks and contains an odd block $B$ of size $(a,b)$. For this block we have
\begin{align*}
   (e')^{\otimes b} \circ  B \circ (e')^{\otimes a}&=(\frac{1}{2})^{a+b} \sum_{\substack{z_{1},\ldots,z_{a}\in \Z_{2}\\z_{1'}\ldots z_{b'}\in \Z_{2}}}   (\prod_{\substack{j \in \{1,\ldots,a,1',\ldots,b'\} }}(z_{j}))   \begin{tikzpicture}[centerzero]
        \pd{-0.4,-0.2};
        \pd{0.4,-0.2};
        \pd{0.8,-0.2};
        \pd{-0.4,0.2} node[anchor=south] {\dotlabel{z_{1'}}};
        \pd{0.4,0.2} node[anchor=south] {\dotlabel{z_{n-1'}}};
        \pd{0.8,0.2} node[anchor=south] {\dotlabel{z_{b'}}};
        \draw (-0.4,-0.2) -- (-0.4,0.2);
        \draw (0.4,-0.2) -- (0.4,0.2);
        \draw (0.8,-0.2) -- (0.8,0.2);
        \node at (0.04,0) {$\cdots$};
    \end{tikzpicture}
    \circ
    B
    \circ
    \begin{tikzpicture}[centerzero]
        \pd{-0.4,-0.2} node[anchor=north] {\dotlabel{z_1}};
        \pd{0.4,-0.2} node[anchor=north] {\dotlabel{z_{a-1}}};
        \pd{0.8,-0.2} node[anchor=north] {\dotlabel{z_a}};
        \pd{-0.4,0.2};
        \pd{0.4,0.2};
        \pd{0.8,0.2};
        \draw (-0.4,-0.2) -- (-0.4,0.2);
        \draw (0.4,-0.2) -- (0.4,0.2);
        \draw (0.8,-0.2) -- (0.8,0.2);
        \node at (0.04,0) {$\cdots$};
    \end{tikzpicture}\\&=0.
\end{align*}
The last equality follows from the fact that 
\begin{align*}
    \begin{tikzpicture}[centerzero]
        \pd{-0.4,-0.2};
        \pd{0.4,-0.2};
        \pd{0.8,-0.2};
        \pd{-0.4,0.2} node[anchor=south] {\dotlabel{z_{1'}}};
        \pd{0.4,0.2} node[anchor=south] {\dotlabel{z_{b-1'}}};
        \pd{0.8,0.2} node[anchor=south] {\dotlabel{z_{b'}}};
        \draw (-0.4,-0.2) -- (-0.4,0.2);
        \draw (0.4,-0.2) -- (0.4,0.2);
        \draw (0.8,-0.2) -- (0.8,0.2);
        \node at (0.04,0) {$\cdots$};
    \end{tikzpicture}
    \circ
    B
    \circ
    \begin{tikzpicture}[centerzero]
        \pd{-0.4,-0.2} node[anchor=north] {\dotlabel{z_1}};
        \pd{0.4,-0.2} node[anchor=north] {\dotlabel{z_{a-1}}};
        \pd{0.8,-0.2} node[anchor=north] {\dotlabel{z_a}};
        \pd{-0.4,0.2};
        \pd{0.4,0.2};
        \pd{0.8,0.2};
        \draw (-0.4,-0.2) -- (-0.4,0.2);
        \draw (0.4,-0.2) -- (0.4,0.2);
        \draw (0.8,-0.2) -- (0.8,0.2);
        \node at (0.04,0) {$\cdots$};
    \end{tikzpicture}&= \begin{tikzpicture}[centerzero]
        \pd{-0.4,-0.2};
        \pd{0.4,-0.2};
        \pd{0.8,-0.2};
        \pd{-0.4,0.2} node[anchor=south] {\dotlabel{-z_{1'}}};
        \pd{0.4,0.2} node[anchor=south] {\dotlabel{-z_{b-1'}}};
        \pd{0.8,0.2} node[anchor=south] {\dotlabel{-z_{b'}}};
        \draw (-0.4,-0.2) -- (-0.4,0.2);
        \draw (0.4,-0.2) -- (0.4,0.2);
        \draw (0.8,-0.2) -- (0.8,0.2);
        \node at (0.04,0) {$\cdots$};
    \end{tikzpicture}
    \circ
    B
    \circ
    \begin{tikzpicture}[centerzero]
        \pd{-0.4,-0.2} node[anchor=north] {\dotlabel{-z_1}};
        \pd{0.4,-0.2} node[anchor=north] {\dotlabel{-z_{a-1}}};
        \pd{0.8,-0.2} node[anchor=north] {\dotlabel{-z_a}};
        \pd{-0.4,0.2};
        \pd{0.4,0.2};
        \pd{0.8,0.2};
        \draw (-0.4,-0.2) -- (-0.4,0.2);
        \draw (0.4,-0.2) -- (0.4,0.2);
        \draw (0.8,-0.2) -- (0.8,0.2);
        \node at (0.04,0) {$\cdots$};
    \end{tikzpicture} \text{ and }\\
    \prod_{\substack{j \in \{1,\ldots,a,1',\ldots,b'\} }}(z_{j}) &= - \prod_{\substack{j \in \{1,\ldots,a,1',\ldots,b'\} }}(-z_{j})
\end{align*}
because there is an odd amount of vertices in the block $B$. This implies $(e')^{\otimes l} \circ g' \circ (e')^{\otimes k}=0 $ and 
\begin{align*}
    f&=(e')^{\otimes l}\circ g \circ (e')^{\otimes k}\\&=(e')^{\otimes l}\circ \phi(\sigma) \circ g' \circ \phi(\rho) \circ (e')^{\otimes k}\\
    &= \phi(\sigma) \circ (e')^{\otimes l} \circ g' \circ (e')^{\otimes k} \circ \phi(\rho)\\
    &=0,
\end{align*}
which lies in the image trivially. The third equality follows from the proven fact that $(e')^{2}$ commutes with $\crossing$.  Now we can assume that $g$ is an even partition in $P(k,l)$, but then $g\in \text{Hom}_{\underline{\text{Rep}}_{0}(H_{t})}([k], [l])$, so we have that 
\begin{align*}
    \Omega_{0}(g)= (e')^{\otimes l}\circ g \circ (e')^{\otimes k} =f.
\end{align*}
This concludes the proof of the fullness of $\Omega_{0}$.\\
Now we want to show that $\Omega_{0}$ is faithful. 
Let $f\in \text{Hom}_{\underline{\text{Rep}}_{0}(H_{t})}([k], [l])$, we consider $\Omega_{0}(f)= (e')^{\otimes l}\circ f \circ (e')^{\otimes k}.$ First assume that $f$ contains only one even block $B=f$. Then similarly as before we see that \begin{align*}
   (e')^{\otimes l} \circ  B \circ (e')^{\otimes k}&=(\frac{1}{2})^{k+l} \sum_{\substack{z_{1},\ldots,z_{k}\in \Z_{2}\\z_{1'}\ldots z_{l'}\in \Z_{2}}}   (\prod_{\substack{j \in \{1,\ldots,k,1',\ldots,l'\} }}(z_{j}))   \begin{tikzpicture}[centerzero]
        \pd{-0.4,-0.2};
        \pd{0.4,-0.2};
        \pd{0.8,-0.2};
        \pd{-0.4,0.2} node[anchor=south] {\dotlabel{z_{1'}}};
        \pd{0.4,0.2} node[anchor=south] {\dotlabel{z_{l-1'}}};
        \pd{0.8,0.2} node[anchor=south] {\dotlabel{z_{l'}}};
        \draw (-0.4,-0.2) -- (-0.4,0.2);
        \draw (0.4,-0.2) -- (0.4,0.2);
        \draw (0.8,-0.2) -- (0.8,0.2);
        \node at (0.04,0) {$\cdots$};
    \end{tikzpicture}
    \circ
    B
    \circ
    \begin{tikzpicture}[centerzero]
        \pd{-0.4,-0.2} node[anchor=north] {\dotlabel{z_1}};
        \pd{0.4,-0.2} node[anchor=north] {\dotlabel{z_{k-1}}};
        \pd{0.8,-0.2} node[anchor=north] {\dotlabel{z_k}};
        \pd{-0.4,0.2};
        \pd{0.4,0.2};
        \pd{0.8,0.2};
        \draw (-0.4,-0.2) -- (-0.4,0.2);
        \draw (0.4,-0.2) -- (0.4,0.2);
        \draw (0.8,-0.2) -- (0.8,0.2);
        \node at (0.04,0) {$\cdots$};
    \end{tikzpicture}\\
    &=2(\frac{1}{2})^{k+l} \sum_{\substack{z_{1}=1\\z_{2},\ldots,z_{k}\in \Z_{2}\\z_{1'}\ldots z_{l'}\in \Z_{2}}}   (\prod_{\substack{j \in \{1,\ldots,k,1',\ldots,l'\} }}(z_{j})) \begin{tikzpicture}[centerzero]
        \pd{-0.4,-0.2};
        \pd{0.4,-0.2};
        \pd{0.8,-0.2};
        \pd{-0.4,0.2} node[anchor=south] {\dotlabel{z_{1'}}};
        \pd{0.4,0.2} node[anchor=south] {\dotlabel{z_{l-1'}}};
        \pd{0.8,0.2} node[anchor=south] {\dotlabel{z_{l'}}};
        \draw (-0.4,-0.2) -- (-0.4,0.2);
        \draw (0.4,-0.2) -- (0.4,0.2);
        \draw (0.8,-0.2) -- (0.8,0.2);
        \node at (0.04,0) {$\cdots$};
    \end{tikzpicture}
    \circ
    B
    \circ
    \begin{tikzpicture}[centerzero]
        \pd{-0.4,-0.2} node[anchor=north] {\dotlabel{z_1}};
        \pd{0.4,-0.2} node[anchor=north] {\dotlabel{z_{k-1}}};
        \pd{0.8,-0.2} node[anchor=north] {\dotlabel{z_k}};
        \pd{-0.4,0.2};
        \pd{0.4,0.2};
        \pd{0.8,0.2};
        \draw (-0.4,-0.2) -- (-0.4,0.2);
        \draw (0.4,-0.2) -- (0.4,0.2);
        \draw (0.8,-0.2) -- (0.8,0.2);
        \node at (0.04,0) {$\cdots$};
    \end{tikzpicture} \\&\neq 0.
\end{align*}
The inequality follows from the fact that the summands are pairwise different basis elements of $\C P_{\Z_{2}}(k,l)$. For the second equality we identify equivalent $\Z_{2}$-coloured partitions, so it follows from the fact that 
\begin{align*}
    \begin{tikzpicture}[centerzero]
        \pd{-0.4,-0.2};
        \pd{0.4,-0.2};
        \pd{0.8,-0.2};
        \pd{-0.4,0.2} node[anchor=south] {\dotlabel{z_{1'}}};
        \pd{0.4,0.2} node[anchor=south] {\dotlabel{z_{l-1'}}};
        \pd{0.8,0.2} node[anchor=south] {\dotlabel{z_{l'}}};
        \draw (-0.4,-0.2) -- (-0.4,0.2);
        \draw (0.4,-0.2) -- (0.4,0.2);
        \draw (0.8,-0.2) -- (0.8,0.2);
        \node at (0.04,0) {$\cdots$};
    \end{tikzpicture}
    \circ
    B
    \circ
    \begin{tikzpicture}[centerzero]
        \pd{-0.4,-0.2} node[anchor=north] {\dotlabel{z_1}};
        \pd{0.4,-0.2} node[anchor=north] {\dotlabel{z_{k-1}}};
        \pd{0.8,-0.2} node[anchor=north] {\dotlabel{z_k}};
        \pd{-0.4,0.2};
        \pd{0.4,0.2};
        \pd{0.8,0.2};
        \draw (-0.4,-0.2) -- (-0.4,0.2);
        \draw (0.4,-0.2) -- (0.4,0.2);
        \draw (0.8,-0.2) -- (0.8,0.2);
        \node at (0.04,0) {$\cdots$};
    \end{tikzpicture}&= \begin{tikzpicture}[centerzero]
        \pd{-0.4,-0.2};
        \pd{0.4,-0.2};
        \pd{0.8,-0.2};
        \pd{-0.4,0.2} node[anchor=south] {\dotlabel{-z_{1'}}};
        \pd{0.4,0.2} node[anchor=south] {\dotlabel{-z_{l-1'}}};
        \pd{0.8,0.2} node[anchor=south] {\dotlabel{-z_{l'}}};
        \draw (-0.4,-0.2) -- (-0.4,0.2);
        \draw (0.4,-0.2) -- (0.4,0.2);
        \draw (0.8,-0.2) -- (0.8,0.2);
        \node at (0.04,0) {$\cdots$};
    \end{tikzpicture}
    \circ
    B
    \circ
    \begin{tikzpicture}[centerzero]
        \pd{-0.4,-0.2} node[anchor=north] {\dotlabel{-z_1}};
        \pd{0.4,-0.2} node[anchor=north] {\dotlabel{-z_{k-1}}};
        \pd{0.8,-0.2} node[anchor=north] {\dotlabel{-z_k}};
        \pd{-0.4,0.2};
        \pd{0.4,0.2};
        \pd{0.8,0.2};
        \draw (-0.4,-0.2) -- (-0.4,0.2);
        \draw (0.4,-0.2) -- (0.4,0.2);
        \draw (0.8,-0.2) -- (0.8,0.2);
        \node at (0.04,0) {$\cdots$};
    \end{tikzpicture} \text{ and }\\
    \prod_{\substack{j \in \{1,\ldots,a,1',\ldots,b'\} }}(z_{j}) &= \prod_{\substack{j \in \{1,\ldots,a,1',\ldots,b'\} }}(-z_{j})
\end{align*}
because we have an even amount of vertices in the block. Now we assume that $f$ is any even partition. Again by Proposition \ref{partition can be written by permutation partitions} we can write $f=\phi(\sigma) \circ f' \circ \phi(\rho)$ where $f'$ is a non-crossing form of $f$. Then 
\begin{align*}
    \Omega_{0}(f)&=(e')^{\otimes l}\circ f \circ (e')^{\otimes k}\\&=(e')^{\otimes l}\circ \phi(\sigma) \circ f' \circ \phi(\rho) \circ (e')^{\otimes k}\\
    &= \phi(\sigma) \circ (e')^{\otimes l} \circ f' \circ (e')^{\otimes k} \circ \phi(\rho)\\
    &\neq 0,
\end{align*}
because $f'$ is a horizontal concatenation of non-zero even blocks. This shows that $\Omega_{0}(f)\neq 0$ for all $f\in P_{even}$, in other words $\Omega_{0}$ is non-zero on the generators of $\text{Hom}_{\underline{\text{Rep}}_{0}(H_{t})}([k], [l])$. Now assume that the set of even partitions $\{f_{1},\ldots,f_{m}\}\subset P_{even}(k,l)$ consists of pairwise non-equivalent even partitions. Then by definition they form a linearly independent set in $\text{Hom}_{\underline{\text{Rep}}_{0}(H_{t})}([k], [l])$. It follows from the definition of linear independency of the morphism spaces in $\text{Par}(\Z_{2},2t)$ that the set 

\begin{align*}
    \{\begin{tikzpicture}[centerzero]
        \pd{-0.4,-0.2};
        \pd{0.4,-0.2};
        \pd{0.8,-0.2};
        \pd{-0.4,0.2} node[anchor=south] {\dotlabel{z_{1'}}};
        \pd{0.4,0.2} node[anchor=south] {\dotlabel{z_{l-1'}}};
        \pd{0.8,0.2} node[anchor=south] {\dotlabel{z_{l'}}};
        \draw (-0.4,-0.2) -- (-0.4,0.2);
        \draw (0.4,-0.2) -- (0.4,0.2);
        \draw (0.8,-0.2) -- (0.8,0.2);
        \node at (0.04,0) {$\cdots$};
    \end{tikzpicture}
    \circ
    f_{j}
    \circ
    \begin{tikzpicture}[centerzero]
        \pd{-0.4,-0.2} node[anchor=north] {\dotlabel{z_1}};
        \pd{0.4,-0.2} node[anchor=north] {\dotlabel{z_{k-1}}};
        \pd{0.8,-0.2} node[anchor=north] {\dotlabel{z_k}};
        \pd{-0.4,0.2};
        \pd{0.4,0.2};
        \pd{0.8,0.2};
        \draw (-0.4,-0.2) -- (-0.4,0.2);
        \draw (0.4,-0.2) -- (0.4,0.2);
        \draw (0.8,-0.2) -- (0.8,0.2);
        \node at (0.04,0) {$\cdots$};
    \end{tikzpicture}|\text{ } j \in \{1,\ldots,m\}, z_{1}=1,z_{2},\ldots,z_{k}\in \Z_{2},z_{1'}\ldots z_{l'}\in \Z_{2}\}
\end{align*}
is also linearly independent. This implies in its turn that the set $\{\Omega_{0}(f_{1}),\ldots,\Omega (f_{m})\}$ is linearly independent. This concludes the proof of the faithfulness of $\Omega_{0}$.
\end{proof}

\begin{thm} \label{Properties of the functor Omega} 
  The functor  $\Omega: \underline{\text{Rep}}(H_{t}) \to \text{Par}(\Z_{2},2t)^{Kar}$ is fully faithful and essentially surjective, therefore it is an equivalence.
\end{thm}
\begin{proof}
We apply Proposition \ref{Karoubian envelope full, faithful} to Theorem \ref{Properties of the functor Omega0} to show that $\Omega$ is fully faithful. We are left to prove the essential surjectivity of $\Omega$. We want to show that $[\tilde{1}]$ lies in the essential image of $\Omega.$ For this we note that $[\tilde{1}] \cong ([\tilde{1}],e') \oplus ([\tilde{1}],id_{[\tilde{1}]}-e')$ and that $\Omega([1])=([\tilde{1}],e')$ by definition. 
We write 
\begin{align*}
    e''=id_{[\tilde{1}]}-e'=\frac{
}{2}: [\tilde{1}] \to [\tilde{1}]            
    \end{align*}
and claim that $([\tilde{1}],e'')$ is isomorphic to the image of $ ([2],\Fourlegs)$ under $\Omega$. To prove this it suffices to show the corresponding claim for $(\tilde{H}^{Kar})^{-1}\circ \Omega \circ \tilde{G}^{Kar}: \text{Par}_{t}^{Kar}\to (\text{Par}(\Z_{2})/\sim_{2t})^{Kar}$, namely that
\begin{align*}
    (\tilde{H}^{Kar})^{-1}\circ \Omega \circ \tilde{G}^{Kar}((W\otimes W,%
))\cong (W,e''),
\end{align*}
where
\begin{align*}
    e''= \frac{\tokstrand[1]+\tokstrand[-1]}{2}.
\end{align*}
and $W$ in the left side of the equation stands for the generating object in $\text{Par}_{t}$ and in the right side for the generating object in $\text{Par}(\Z_{2})/\sim_{2t}$ as was done in Definition \ref{generator relations repH0 definition} and Definition \ref{definition generators relations parZ2} respectively. By definition we have that 
\begin{align*}
     (\tilde{H}^{Kar})^{-1}\circ \Omega \circ \tilde{G}^{Kar}((W\otimes W,%
 \circ e' \otimes e)).
\end{align*}
We define the morphisms
\begin{align*}
    \alpha&:= 2(e'\otimes e' \circ %
 \circ e' \otimes e')) \to (W,e'')   
\end{align*}
and show that they are isomorphisms. We first use the relations in Definition \ref{definition generators relations parZ2} to simplify $\alpha$ and $\beta$. The equality
\begin{center}

\tikzset{every picture/.style={line width=0.75pt}} 

%
\end{center}

With these descriptions, it is an easy calculation to show $\beta \circ \alpha =e''$ and $\alpha \circ \beta = e' \otimes e' \circ \FOURLEGS \circ e' \otimes e'$, which equal $id_{(W,e'')}$ and $id_{(W\otimes W, 2(e'\otimes e' \circ %
 \circ e' \otimes e'))}$ respectively. This proves the claim.
        Because $\Omega$ is an additive functor we get
        \begin{align*}
            \Omega([1]\oplus ([2],\Fourlegs))\cong \Omega([1]) \oplus \Omega(([2],\Fourlegs)) \cong  ([\tilde{1}],e') \oplus ([\tilde{1}],id_{[\tilde{1}]}-e')  \cong [\tilde{1}]=([\tilde{1}],id_{[\tilde{1}]}).
        \end{align*}
        Let $([\tilde{k}],f)$ be any element of $\text{Par}(\Z_{2},2t)^{Kar}$, where $k\in \N$ and $f:[\tilde{k}]\to [\tilde{k}]$ is an idempotent. Because $ \Omega(([1]\oplus ([2],\Fourlegs))^{\otimes k})\cong [\tilde{k}]=([\tilde{k}],id_{[\tilde{k}]})$ and $\Omega$ is full, there exists a morphism $f':([1]\oplus ([2],\Fourlegs))^{\otimes k} \to ([1]\oplus ([2],\Fourlegs))^{\otimes k}$ with $\Omega(f')\cong f$. Because $\Omega$ is a faithful functor and 
        \begin{align*}
            \Omega(f' \circ f')= \Omega(f')\circ \Omega(f') \cong f \circ f = f \cong \Omega(f'),
        \end{align*}
        the morphism $f'=f' \circ f'$ is also an idempotent. By definition, see Remark \ref{essential injectivity omega}, we get
        \begin{align*}
            \Omega( ([1]\oplus ([2],\Fourlegs))^{\otimes k},f')\cong ([\tilde{k}],f).
        \end{align*}
        This concludes the proof.
\end{proof}

\begin{rem} \label{essential injectivity omega}
     By the universal property of the Karoubian envelope the functor $\Omega$ is only defined up to isomorphism, so there are multiple choices if we want to define a particular functor. But in this particular case, $\Omega$ being a functor from one Karoubian envelope into another, we can make a canonical choice for the images of the objects in $\underline{\text{Rep}}(H_{t})$, namely 
\begin{align*}
    \Omega(([k],f))=([\tilde{k}],\Omega_{0}(f))=([\tilde{k}],(e')^{\otimes k} \circ f \circ (e')^{\otimes k}).
\end{align*}
We fix this choice from now on.
\end{rem}

\subsubsection{An application: indecomposable objects}

Knop showed in \cite[Theorem 6.1]{Knop_2007} that in the semisimple case $t\neq 2 \N$, the irreducible objects in $\text{Par}(\Z_{2},t)^{Kar}$ are classified by the set of all bipartitions, see also \cite[Chapter 9]{Nyobe_Likeng_2021}. The equivalence of categories allows us to extend this result to the classification of the indecomposable objects in $\text{Par}(\Z_{2},t)^{Kar}$ for the non-semisimple cases $t \in 2\N \backslash \{0\}$ using \cite[Proposition 5.12]{Flake_2021}. 

\begin{prop} \label{indecomposables in repht}
    Let $t\in \C \backslash \{0\}$. Then there is a bijection between the set of bipartitions $\lambda=(\lambda_{1},\lambda_{2})$ of arbitrary size and the set of isomorphism classes of non-zero indecomposable objects in $\text{Par}(\Z_{2},t)^{Kar}$. 
\end{prop}



\subsection{Semisimplification} \label{sec:ss}

Obviously $\Omega$ induces a symmetric monoidal equivalence \begin{align*} \widehat{\Omega}:\widehat{\underline{\text{Rep}}(H_{n})}\to \widehat{\text{Par}(\Z_{2},2n)^{Kar}}. \end{align*} between the semisimplifications. This equivalence is the identity if we work with our fixed choices of $H$ and $\Omega$.

\begin{cor} \label{commutative square}
     The equivalence $\Omega:\underline{\text{Rep}}(H_{n})\to \text{Par}(\Z_{2},2n)^{Kar}$ makes the following square
     \begin{equation}
\begin{tikzcd}
\underline{\text{Rep}}(H_{n}) \arrow{r}{G} \arrow[d,"\Omega"] & \text{Rep}(H_{n})\arrow{d}{=} \ \\ \text{Par}(\mathbb{Z}_{2},2n)^{Kar} \arrow{r}{H} & \text{Rep}(H_{n})
\end{tikzcd}
\end{equation}
commute for all $n\in \N.$
\end{cor}

\begin{proof} 
By the universal property of the Karoubian envelope, the corollary will follow immediately if we can show that 
\begin{align*}
    H \circ \Omega_{0}=G\circ \iota_{G} = G':\underline{\text{Rep}}_{0}(H_{n}) \to \text{Rep}(H_{n})
\end{align*}
for some choice of $H$. We will work with the version of $H$ which sends $([\tilde{k}],(e')^{\otimes k})$ to $u^{\otimes k}$ for all $k\in \N$. This is possible because $u^{\otimes k} \cong im(e^{\otimes k})$, where $e:V\to V$ was defined in the beginning of Section 3.2.

We want to show that the diagram

\begin{equation*}
\begin{tikzcd}
\underline{\text{Rep}}_{0}(H_{n}) \arrow{r}{\iota_{G}} \arrow{dr}{\Omega_{0}} &\underline{\text{Rep}}(H_{n}) \arrow{r}{G}  & \text{Rep}(H_{n})\arrow{d}{=} \ \\  & (Par(\mathbb{Z}_{2},2n))^{Kar} \arrow{r}{H} & \text{Rep}(H_{n})
\end{tikzcd}
\end{equation*}
commutes strictly for our choice of $H$. For an object $[k] \in \underline{\text{Rep}}_{0}(H_{n})$ it is clear that 
\begin{align*}
    H \circ \Omega_{0}([k]) =H(([\tilde{k}],id_{[\tilde{k}]})= u^{\otimes k}=G([k])=G \circ \iota_{G}([k])
\end{align*} 
for this choice of $H$. Now we want to prove that the diagram commutes for morphisms. For this we use the alternative descriptions of the functors $G'$ and $H'$ from subsections 5.1.2 and 5.1.1 respectively.  The equivalences $\tilde{G}$ and $\tilde{H}$  reduce the problem to showing that the corresponding diagram 
\begin{equation*}
\begin{tikzcd}
Par_{n} \arrow{r}{ (\tilde{G}^{Kar})^{-1} \circ \iota_{G} \circ \tilde{G}} \arrow{dr}{(\tilde{H}^{Kar})^{-1} \circ \Omega_{0}\circ \tilde{G}} &(Par_{n})^{Kar}  \arrow{r}{G\circ \tilde{G}^{Kar}}  & \text{Rep}(H_{n})\arrow{d}{=} \ \\  & (Par(\mathbb{Z}_{2})/\sim_{2n}))^{Kar} \arrow{r}{H\circ \tilde{H}^{Kar}} & \text{Rep}(H_{n})
\end{tikzcd}
\end{equation*}
commutes for morphisms.
So we have to prove that the images of the generating morphisms in Definition \ref{generator relations repH0 definition} under the functors 
\begin{align*}
    H\circ \tilde{H}^{Kar} \circ (\tilde{H}^{Kar})^{-1} \circ \Omega_{0}\circ \tilde{G} = H \circ \Omega_{0} \circ \tilde{G} 
\end{align*}
and
\begin{align*}
   G\circ \tilde{G}^{Kar}\circ (\tilde{G}^{Kar})^{-1} \circ \iota_{G} \circ \tilde{G}=G \circ \iota_{G} \circ \tilde{G}=G' \circ \tilde{G}= G''
\end{align*}
are equal.
It follows directly from the functoriality of all involved functors that the diagram commutes for the identity $\idstrand$ . Now we want to show that the diagram commutes for $\begin{tikzpicture}[anchorbase]
            \draw (-0.2,-0.4) -- (0,-0.2) -- (0.2,-0.4);
            \draw (-0.2,0.4) -- (0,0.2) -- (0.2,0.4);
            \draw (0,-0.2) -- (0,0.2);
        \end{tikzpicture}$. We first see that 
\begin{align*}
    G''  ( \begin{tikzpicture}[anchorbase]
            \draw (-0.2,-0.4) -- (0,-0.2) -- (0.2,-0.4);
            \draw (-0.2,0.4) -- (0,0.2) -- (0.2,0.4);
            \draw (0,-0.2) -- (0,0.2);
        \end{tikzpicture})(e_{i}\otimes e_{j})=\delta_{i,j}e_{i}\otimes e_{i}
\end{align*} for $i,j\in \{1,\ldots, n\}$ and $e_{i}$ a canonical basiselement of $u$. To use the explicit description of $H'$ from subsection 5.1.1, we must first consider the case where $H$ sends $([\tilde{k}],(e')^{\otimes k})$ to $(\tilde{u})^{\otimes k}$, a tensor powers of $\tilde{u}$, the subrepresentation of $V$ isomorphic to $u$, see Remark \ref{u subrepresentatino of V}. After we have done this, we can use the isomorphism 
\begin{align*}
     \tilde{u} &\to u\\
     e_{1}^{i}-e_{-1}^{i} &\mapsto e_{i} \text{ for } i\in \{1,\ldots,n\}
\end{align*}
to find out the image of $\begin{tikzpicture}[anchorbase]
            \draw (-0.2,-0.4) -- (0,-0.2) -- (0.2,-0.4);
            \draw (-0.2,0.4) -- (0,0.2) -- (0.2,0.4);
            \draw (0,-0.2) -- (0,0.2);
        \end{tikzpicture}$ under the functor $H\circ \tilde{H}^{Kar} \circ (\tilde{H}^{Kar})^{-1} \circ \Omega_{0}\circ \tilde{G}$ for our original choice of $H$.
So we will assume now that $H(([\tilde{k}],(e')^{\otimes k}))=(\tilde{u})^{\otimes k}$ until further notice. We already saw that the equality 
\begin{align*}
    H(e')((e_{1}^{i}-e_{-1}^{i}))&= \frac{1}{2}(e_{1}^{i}-e_{-1}^{i}-e_{-1}^{i}+e_{1}^{i})\\
    &= \frac{1}{2}2(e_{1}^{i}-e_{-1}^{i})\\
    &= (e_{1}^{i}-e_{-1}^{i})
\end{align*}
 holds, which is obvious because $e_{1}^{i}-e_{-1}^{i}$ lies in the image of $e=H(e')$. Note that $H(e')=H''(e')$ where we use the corresponding definitions of $e'$ in $(\text{Par}(\Z_{2},2n))^{Kar}$ and $\text{Par}(\Z_{2})/\sim_{2n}$ respectively. We use $\iota_{\tilde{H}}:\text{Par}(\Z_{2})/\sim_{2n} \to (\text{Par}(\Z_{2})/\sim_{2n})^{Kar}$ to denote the $\C$-linear full embedding of the category into its Karoubian envelope. Then $H\circ \tilde{H}^{Kar}\circ \iota_{\tilde{H}}=H' \circ \tilde{H}=H''$ implies that
\begin{align*}
    H\circ \tilde{H}^{Kar} \circ (\tilde{H}^{Kar})^{-1} &\circ \Omega_{0}\circ \tilde{G}
    (\begin{tikzpicture}[anchorbase]
            \draw (-0.2,-0.4) -- (0,-0.2) -- (0.2,-0.4);
            \draw (-0.2,0.4) -- (0,0.2) -- (0.2,0.4);
            \draw (0,-0.2) -- (0,0.2);
        \end{tikzpicture})((e_{1}^{i}-e_{-1}^{i})\otimes (e_{1}^{j}-e_{-1}^{j}))\\
    &=H\circ \tilde{H}^{Kar}(2 e'\otimes e' \circ  \begin{tikzpicture}[anchorbase]
            \draw (-0.2,-0.4) -- (0,-0.2) -- (0.2,-0.4);
            \draw (-0.2,0.4) -- (0,0.2) -- (0.2,0.4);
            \draw (0,-0.2) -- (0,0.2);
        \end{tikzpicture} \circ  e'\otimes e')((e_{1}^{i}-e_{-1}^{i})\otimes (e_{1}^{j}-e_{-1}^{j}))\\
    &=H\circ \tilde{H}^{Kar}\circ \iota_{\tilde{H}}(2 e'\otimes e' \circ  \begin{tikzpicture}[anchorbase]
            \draw (-0.2,-0.4) -- (0,-0.2) -- (0.2,-0.4);
            \draw (-0.2,0.4) -- (0,0.2) -- (0.2,0.4);
            \draw (0,-0.2) -- (0,0.2);
        \end{tikzpicture} \circ  e'\otimes e')((e_{1}^{i}-e_{-1}^{i})\otimes (e_{1}^{j}-e_{-1}^{j}))\\
     &=H'\circ \tilde{H}(2 e'\otimes e' \circ  \begin{tikzpicture}[anchorbase]
            \draw (-0.2,-0.4) -- (0,-0.2) -- (0.2,-0.4);
            \draw (-0.2,0.4) -- (0,0.2) -- (0.2,0.4);
            \draw (0,-0.2) -- (0,0.2);
        \end{tikzpicture} \circ  e'\otimes e')((e_{1}^{i}-e_{-1}^{i})\otimes (e_{1}^{j}-e_{-1}^{j}))\\
    &=H''(2 e'\otimes e' \circ  \begin{tikzpicture}[anchorbase]
            \draw (-0.2,-0.4) -- (0,-0.2) -- (0.2,-0.4);
            \draw (-0.2,0.4) -- (0,0.2) -- (0.2,0.4);
            \draw (0,-0.2) -- (0,0.2);
        \end{tikzpicture} \circ  e'\otimes e')((e_{1}^{i}-e_{-1}^{i})\otimes (e_{1}^{j}-e_{-1}^{j}))\\   
        &= 2H''( e'\otimes e' \circ  \begin{tikzpicture}[anchorbase]
            \draw (-0.2,-0.4) -- (0,-0.2) -- (0.2,-0.4);
            \draw (-0.2,0.4) -- (0,0.2) -- (0.2,0.4);
            \draw (0,-0.2) -- (0,0.2);
        \end{tikzpicture} \circ  e'\otimes e')((e_{1}^{i}-e_{-1}^{i})\otimes (e_{1}^{j}-e_{-1}^{j}))\\
        &= 2H''( e'\otimes e' \circ  \begin{tikzpicture}[anchorbase]
            \draw (-0.2,-0.4) -- (0,-0.2) -- (0.2,-0.4);
            \draw (-0.2,0.4) -- (0,0.2) -- (0.2,0.4);
            \draw (0,-0.2) -- (0,0.2);
        \end{tikzpicture})((e_{1}^{i}-e_{-1}^{i})\otimes (e_{1}^{j}-e_{-1}^{j}))\\
        &=2H''( e'\otimes e' \circ  \spliter \circ \merge)((e_{1}^{i}-e_{-1}^{i})\otimes (e_{1}^{j}-e_{-1}^{j}))\\
        &= 2 \delta_{i,j} H''( e'\otimes e' \circ  \spliter)( e_{1}^{i} + e_{-1}^{i})\\
        &=\delta_{i,j}2H''( e'\otimes e')  ( e_{1}^{i}\otimes e_{1}^{i} + e_{-1}^{i} \otimes e_{-1}^{i})\\
        &=\delta_{i,j}2 \frac{1}{4}(2 e_{1}^{i}\otimes e_{1}^{i} -2e_{-1}^{i}\otimes e_{1}^{i} -2e_{1}^{i}\otimes e_{-1}^{i}+2e_{-1}^{i}\otimes e_{-1}^{i})\\
        &=\delta_{i,j} \frac{1}{2}2( e_{1}^{i}\otimes e_{1}^{i} -e_{-1}^{i}\otimes e_{1}^{i} -e_{1}^{i}\otimes e_{-1}^{i}+e_{-1}^{i}\otimes e_{-1}^{i})\\
        &=\delta_{i,j}(e_{1}^{i}-e_{-1}^{i})\otimes (e_{1}^{j}-e_{-1}^{j}),
\end{align*} for $i,j\in \{1,\ldots, n\}$. By the previous remarks this shows that
\begin{align*}
     H\circ \tilde{H}^{Kar} \circ (\tilde{H}^{Kar})^{-1} &\circ \Omega_{0}\circ \tilde{G}
    (\begin{tikzpicture}[anchorbase]
            \draw (-0.2,-0.4) -- (0,-0.2) -- (0.2,-0.4);
            \draw (-0.2,0.4) -- (0,0.2) -- (0.2,0.4);
            \draw (0,-0.2) -- (0,0.2);
        \end{tikzpicture})(e_{i}\otimes e_{j})=\delta_{i,j}(e_{i} \otimes e_{i})=G''(\begin{tikzpicture}[anchorbase]
            \draw (-0.2,-0.4) -- (0,-0.2) -- (0.2,-0.4);
            \draw (-0.2,0.4) -- (0,0.2) -- (0.2,0.4);
            \draw (0,-0.2) -- (0,0.2);
        \end{tikzpicture})(e_{i} \otimes e_{j})
\end{align*}
for our original choice of $H$. This implies the
commutativity for $\begin{tikzpicture}[anchorbase]
            \draw (-0.2,-0.4) -- (0,-0.2) -- (0.2,-0.4);
            \draw (-0.2,0.4) -- (0,0.2) -- (0.2,0.4);
            \draw (0,-0.2) -- (0,0.2);
        \end{tikzpicture}$. The commutativity for $  \begin{tikzpicture}[anchorbase]
            \draw (-0.2,-0.4) -- (0,-0.2) -- (0.2,-0.4);
            \draw (0,-0.2) -- (0,0.2);
            \opendot{0,0.2};
        \end{tikzpicture}$ and $ \begin{tikzpicture}[anchorbase]
            \draw (-0.2,0.4) -- (0,0.2) -- (0.2,0.4);
            \draw (0,-0.2) -- (0,0.2);
            \opendot{0,-0.2};
        \end{tikzpicture}$ is proven similarly. Lastly we want to show the commutativity of the diagram for $\CROSS$.
        We see that 
        \begin{align*}
    G''  ( \CROSS)(e_{i}\otimes e_{j})=e_{j}\otimes e_{i}
\end{align*}
        and for the choice $H(([\tilde{k}],(e')^{\otimes k}))=(\tilde{u})^{\otimes k}$ we see that

\begin{align*}
    H\circ \tilde{H}^{Kar} \circ (\tilde{H}^{Kar})^{-1} &\circ \Omega_{0}\circ \tilde{G}
    (\CROSS)((e_{1}^{i}-e_{-1}^{i})\otimes (e_{1}^{j}-e_{-1}^{j}))\\
    &=H\circ \tilde{H}^{Kar}( e'\otimes e' \circ  \CROSS \circ  e'\otimes e')((e_{1}^{i}-e_{-1}^{i})\otimes (e_{1}^{j}-e_{-1}^{j}))\\
    &=H\circ \tilde{H}^{Kar}\circ \iota_{\tilde{H}}( e'\otimes e' \circ  \CROSS \circ  e'\otimes e')((e_{1}^{i}-e_{-1}^{i})\otimes (e_{1}^{j}-e_{-1}^{j}))\\
     &=H'\circ \tilde{H}( e'\otimes e' \circ  \CROSS \circ  e'\otimes e')((e_{1}^{i}-e_{-1}^{i})\otimes (e_{1}^{j}-e_{-1}^{j}))\\
    &=H''( e'\otimes e' \circ  \CROSS \circ  e'\otimes e')((e_{1}^{i}-e_{-1}^{i})\otimes (e_{1}^{j}-e_{-1}^{j}))\\   
        &= H''( e'\otimes e' \circ  \CROSS)((e_{1}^{i}-e_{-1}^{i})\otimes (e_{1}^{j}-e_{-1}^{j}))\\
        &= H''( e'\otimes e' )((e_{1}^{j}-e_{-1}^{j})\otimes (e_{1}^{i}-e_{-1}^{i}))\\
        &= (e_{1}^{j}-e_{-1}^{j})\otimes (e_{1}^{i}-e_{-1}^{i})\\
\end{align*}
for all $i,j\in \{1,\ldots,n\}$. By the same arguments that we used to prove the commutativity for $\begin{tikzpicture}[anchorbase]
            \draw (-0.2,-0.4) -- (0,-0.2) -- (0.2,-0.4);
            \draw (-0.2,0.4) -- (0,0.2) -- (0.2,0.4);
            \draw (0,-0.2) -- (0,0.2);
        \end{tikzpicture}$, this shows the commutativity for $\CROSS$ and we conclude the proof.
\end{proof}



\printbibliography

\end{document}